\documentclass{amsart}
\usepackage{amsmath,amsthm,amssymb,latexsym,epic,bbm,comment}
\usepackage{graphicx,enumerate,stmaryrd, xcolor, tikz-cd}
\usepackage[all,2cell]{xy}
\xyoption{2cell}
\usepackage[active]{srcltx}
\usepackage[parfill]{parskip}
\UseTwocells
\usepackage[
colorlinks=true,
linkcolor=black, 
anchorcolor=black,
citecolor=black, 
urlcolor=black, 
]{hyperref}
\usepackage[e]{esvect}
\usepackage{mathrsfs}
\usepackage{todonotes}

\newtheorem{theorem}{Theorem}
\newtheorem{lemma}[theorem]{Lemma}
\newtheorem{corollary}[theorem]{Corollary}
\newtheorem{proposition}[theorem]{Proposition}
\newtheorem{conjecture}[theorem]{Conjecture}

\theoremstyle{definition}

\newtheorem{example}[theorem]{Example}
\newtheorem{remark}[theorem]{Remark}

\newtheorem{definition}[theorem]{Definition}

\newcommand{\tto}{\twoheadrightarrow}

\newcommand{\cC}{\mathscr{C}}

\newcommand{\cE}{\mathscr{E}}

\newcommand{\cR}{\mathscr{R}}

\newcommand{\cS}{\mathscr{S}}

\newcommand{\cD}{\mathscr{D}}

\newcommand{\cA}{\mathscr{A}}

\newcommand{\cB}{\mathscr{B}}

\newcommand{\vpar}{\mathsf{v}}

\newcommand{\lmod}{\text{-}\mathrm{mod}}

\newcommand{\proj}{\mathrm{proj}}
\newcommand{\rproj}{\mathrm{proj}}
\newcommand{\bfrproj}{\boldsymbol{\mathrm{proj}}}

\newcommand{\wfmod}{\text{-}\mathrm{afmod}}
\newcommand{\swfmod}{\text{-}\mathrm{safmod}}
\newcommand{\cwfmod}{\text{-}\mathrm{cafmod}}
\newcommand{\copwfmod}{\text{-}\mathrm{copafmod}}
\newcommand{\ccopwfmod}{\text{-}\mathrm{ccopafmod}}
\newcommand{\scopwfmod}{\text{-}\mathrm{scopafmod}}

\newcommand{\comod}{\operatorname{comod}}
\newcommand{\rmod}{\operatorname{mod}}
\newcommand{\bfrmod}{\boldsymbol{\operatorname{mod}}}
\newcommand{\bproj}{\boldsymbol{\operatorname{proj}}}

\newcommand{\coker}{\operatorname{coker}}

\newcommand{\Hom}{\mathrm{Hom}}

\newcommand{\End}{\mathrm{End}}

\newcommand{\rad}{\operatorname{rad}}

\newcommand{\op}{\mathrm{op}} 
\newcommand{\ex}{\mathrm{ex}} 
\newcommand{\Ev}{\mathsf{Ev}}

\newcommand{\Vect}{\operatorname{Vect}}
\DeclareMathOperator{\ind}{ind}

\newcommand{\circh}{\circ_{\mathsf{h}}}
\newcommand{\circv}{\circ_{\mathsf{v}}}

\newcommand{\rA}{\mathrm{A}}
\newcommand{\rB}{\mathrm{B}}
\newcommand{\rC}{\mathrm{C}}

\newcommand{\rF}{\mathrm{F}}
\newcommand{\rG}{\mathrm{G}}
\newcommand{\rH}{\mathrm{H}}

\newcommand{\rK}{\mathrm{K}}

\newcommand{\rM}{\mathrm{M}}
\newcommand{\rN}{\mathrm{N}}

\newcommand{\rT}{\mathrm{T}}
\newcommand{\rX}{\mathrm{X}}
\newcommand{\rY}{\mathrm{Y}}
\newcommand{\rZ}{\mathrm{Z}}

\newcommand{\rBS}{\mathrm{BS}}

\newcommand{\B}{\mathcal{B}}
\newcommand{\C}{\mathcal{C}}
\def\D{{\mathcal{D}}}

\newcommand{\F}{\mathcal{F}}

\def\H{\mathcal{H}}
\newcommand{\I}{\mathcal{I}}
\newcommand{\J}{\mathcal{J}}

\def\L{{\mathcal{L}}}

\newcommand{\X}{\mathcal{X}}
\newcommand{\Y}{\mathcal{Y}}

\newcommand{\bbN}{\mathbb{N}}
\newcommand{\bbZ}{\mathbb{Z}}

\newcommand{\bbC}{\mathbb{C}}

\newcommand{\id}{\mathrm{id}}

\newcommand{\one}{\mathbbm{1}}

\newcommand{\ov}[1]{\overline{#1}}

\newcommand{\bfM}{\mathbf{M}}

\newcommand{\bfG}{\mathbf{G}}

\newcommand{\bfS}{\mathbf{S}}
\newcommand{\bfN}{\mathbf{N}}

\newcommand{\bfC}{\mathbf{C}}

\newcommand{\bfI}{\mathbf{I}}
\newcommand{\bfJ}{\mathbf{J}}
\newcommand{\bfH}{\mathbf{H}}

\newcommand{\bfm}{\mathbf{m}}
\newcommand{\bfu}{\mathbf{u}}

\newcommand{\ti}{\mathtt{i}}
\newcommand{\tj}{\mathtt{j}}
\newcommand{\tk}{\mathtt{k}}

\newcommand{\tx}{\mathtt{x}}
\newcommand{\ty}{\mathtt{y}}

\newcommand{\tB}{\mathtt{B}}

\newcommand{\tX}{\mathtt{X}}

\newcommand{\cpl}[1]{{#1}^\circ}
\newcommand{\copr}[1]{{#1}^\sqcup}
\newcommand{\cop}[1]{{#1}^\diamond}

\newcommand{\Soeaff}{\widehat{\cS}}
\newcommand{\Soeext}{\widehat{\cS}^{\mathrm{ext}}}
\newcommand{\Sfin}{\mathfrak{S}}
\newcommand{\Saff}{\widehat{\mathfrak{S}}}
\newcommand{\Saffext}{\widehat{\mathfrak{S}}^{\mathrm{ext}}}

\numberwithin{equation}{section}
\numberwithin{theorem}{section}

\title[Almost finitary birepresentation theory and Soergel bimodules]{Almost finitary birepresentation theory and applications to affine Soergel bimodules}

\author{Marco Mackaay}
\address{M.M.: Departamento de Matemática, FCT, Universidade do Algarve, Campus de Gambelas,
8005-139 Faro, Portugal \&
Center for Research and Development in Mathematics and Applications (CIDMA), 
Department of Mathematics, University of Aveiro, 3810-193 Aveiro, Portugal,  
\newline \href{https://fct.ualg.pt/bio/mmackaay}{https://fct.ualg.pt/bio/mmackaay}, \href{https://orcid.org/0000-0001-9807-6991}{ORCID 0000-0001-9807-6991}}
\email{mmackaay@ualg.pt}
\author{Vanessa Miemietz}
\address{V.M.: School of Engineering, Mathematics and Physics, University of East Anglia, Norwich NR4 7TJ, United Kingdom,  \newline \href{https://archive.uea.ac.uk/~byr09xgu/}{https://archive.uea.ac.uk/~byr09xgu/}}
\email{v.miemietz@uea.ac.uk}
\author{Pedro Vaz}
\address{P.V.: Institut de Recherche en Math{\'e}matique et Physique, 
Universit{\'e} catholique de Louvain, Chemin du Cyclotron 2,  
1348 Louvain-la-Neuve, Belgium, \newline \href{https://perso.uclouvain.be/pedro.vaz}{https://perso.uclouvain.be/pedro.vaz}, \href{https://orcid.org/0000-0001-9422-4707}{ORCID 0000-0001-9422-4707}}
\email{pedro.vaz@uclouvain.be}

\begin{document}

\begin{abstract}
In this article, we develop a generalization of finitary birepresentation theory applicable to Soergel bimodules for infinite Coxeter groups. We establish a reduction process for the classification of simple birepresentations of {\em almost finitary} bicategories, and consider in detail the case of Soergel bimodules in extended affine type $A$.
\end{abstract}

\maketitle

\tableofcontents

\section{Introduction}

Over the past 15 years, starting with \cite{MM1} and culminating in \cite{MMMTZ}, the theory of finitary birepresentations of finitary bicategories has been developed with the goal of understanding important classes of examples appearing in representation theory, such as categorified quantum groups (Kac--Moody $2$-categories) and Hecke algebras (Soergel bimodules). This has resulted in the classification of simple birepresentations of Soergel bimodules for finite Coxeter groups in characteristic $0$ (except for some exceptions in types $H_3$ and $H_4$, see below) \cite{MMMTZ2}, by connecting their birepresentation theory to that of fusion categories as developed in e.g. \cite{EGNO}. 

Unfortunately, the setup of finitary birepresentation theory is too restrictive to be applicable to Soergel bimodules for infinite Coxeter groups or Kac--Moody $2$-categories for infinite Dynkin types. Hence, in this paper, we generalize the main results from \cite{MMMTZ} to the almost finitary setting, where the bicategories have countably many objects and isomorphism classes of indecomposable $1$-morphisms and the birepresentations have countably many isomorphism classes of indecomposable objects (instead of finitely many, as in the finitary case). 

We then apply these results to the almost finitary birepresentation theory of Soergel bimodules for Coxeter systems of finite rank. In particular, we relate birepresentations whose apex (see below for more detail) contain the longest element of a finite parabolic subgroup to those of Lusztig's associated asymptotic category for the corresponding $H$-cell. We go into further detail on some cases where these asymptotic categories are well-understood, such as (extended) affine type $A$ Soergel bimodules or universal Coxeter groups.

We will now explain the results from \cite{MMMTZ, MMMTZ2} in more detail and sketch the generalisations needed.

The main goal in \cite{MMMTZ} was to develop a method to classify the (equivalence classes of) simple birepresentations of fiab bicategories, where {\em fiab} stands for finitary together with a duality structure 
given by a functorial involution. In that paper, these birepresentations were still called {\em simple transitive}, but all simple birepresentations are automatically transitive, so we have decided to just call them {\em simple} from now on. The method consisted of various reduction steps to simplify the Classification Problem (as it will be called, for short) and led to its solution for Soergel bimodules of all finite Coxeter types (with only two exceptions) in \cite{MMMTZ2}, where an additional reduction step was 
introduced as well. To explain the new technical hurdles in the almost finitary setting, we will briefly recall this method 
and comment on the difference between the finitary and the almost finitary case in every reduction step. It would be very hard to make this paper completely self-contained, so we assume some familiarity with the results in \cite{MMMTZ} and \cite{MMMTZ2}. 

The first reduction step of the Classification Problem for a finitary bicategory $\cC$ uses the fact that every transitive 
finitary birepresentation $\mathbf{M}$ of $\cC$ has an {\em apex}, which is the unique maximal, w.r.t. the two-sided partial order, 
two-sided cell of $\cC$ not annihilated by $\mathbf{M}$. This fact, first proved in \cite[Lemma 1]{CM}, 
allows us to address the Classification Problem per apex. Moreover, for every two-sided cell $\J$ of $\cC$, there is a subquotient bicategory $\cC_\J$ of $\cC$ (the {\em $J$-simple quotient}) whose only two-sided cells are the trivial cell and $\J$, and there is a biequivalence between the 2-category of simple finitary $\cC$-birepresentations with apex $\J$ and the $2$-category of simple finitary $\cC_\J$-birepresentations with apex $\J$.

To guarantee existence and uniqueness of the apex in the almost finitary setting, we introduce the {\em $J$-artinian property} for almost finitary bicategories (Definition~\ref{cellfin}) and prove that every transitive almost finitary birepresentation of an almost finitary $J$-artinian bicategory has a unique apex (Lemma~\ref{lem:cellfin}). This property is certainly satisfied by any almost finitary 
bicategory with finitely many two-sided cells, such as Soergel bimodules for finite Coxeter groups, (extended) affine Weyl groups and 
universal Coxeter groups, but not all $J$-artinian almost finitary bicategories have only finitely many two-sided cells, e.g., the bicategory of all finite dimensional bimodules over dual numbers has infinitely many two-sided cells and is $J$-artinian, see \cite[Section 3.1]{Jo}.

In \cite[Section 13.12]{Lu}, Lusztig conjectured that any finite rank Coxeter group with bounded $\mathtt{a}$-function has finitely many two-sided Kazhdan--Lusztig cells. Recently, it was proved in~\cite{CH} that the $\mathtt{a}$-function on any finite rank Coxeter group is bounded, hence, if Lusztig's conjecture is true, it would imply that the bicategory of Soergel bimodules for any finite rank Coxeter group is $J$-artinian.

The second reduction step of the Classification Problem requires $\cC$ to be fiab. The functorial involution preserves 
the two-sided cells of $\cC$ and interchanges the left and the right cells within each two-sided cell. Given a left cell $\L$ 
in a two-sided cell $\J$, the {\em diagonal $H$-cell} $\H:=\L\cap \L^*$ is preserved by the involution. 
There is a $2$-full subbicategory $\cC_\H$ of $\cC_\J$ associated to $\H$, whose only two-sided cells are the 
trivial cell and $\H$. {\em Strong $H$-reduction} establishes a biequivalence between the $2$-category of simple finitary 
$\cC_\J$-birepresentations with apex $\J$ and the $2$-category of simple finitary $\cC_\H$-birepresentations with apex $\H$. Note that strong $H$-reduction is valid for any diagonal $H$-cell in $\J$, 
so the bicategories $\cC_\H$, for all diagonal $H$-cells $\H$ in $\J$, 
are Morita biequivalent to each other. This is significant, because they may not be biequivalent to each other, and it is also known 
to be false on the decategorified level, e.g., the subquotients of the Hecke algebras associated to diagonal $H$-cells in a two-sided cell 
of some Coxeter group are not Morita equivalent in general. 

Let $\X$ denote a two-sided cell $\J$ or a diagonal $H$-cell $\H$ in $\J$. In the finitary setting, 
the proof of strong $H$-cell reduction uses the correspondence between equivalence classes of simple birepresentations 
$\mathbf{M}$ of $\cC_\X$ with apex $\X$ on one hand and Morita equivalence classes of simple algebra $1$-morphisms $\rA$ in, a priori, the abelianisation of
the additive closure $\X^\oplus$ of $\X$ in $\cC_\X$ on the other hand. A major reduction in \cite{MMMTZ} is the proof that the relevant algebra $1$-morphisms actually exist in $\X^\oplus$ itself without needing an abelianisation. In the almost finitary setting, the algebra 
$1$-morphism $\rA$ can be constructed abstractly in a completion of $\X^{\oplus}$, and we develop a theory of completed bicategories and birepresentations that eventually allows us to conclude that $\rA$ can be chosen as a possibly countably infinite coproduct of indecomposable $1$-morphisms in $\X$. This correspondence only works if $\X$ satisfies certain technical conditions, which are automatic in the finitary case, but not in the almost finitary case. These conditions are not hard to formulate, but can be quite challenging to check in practice. Fortunately, there are cases, such as for Soergel bimodules in extended affine type $A$, in which each two-sided cell $\J$ contains a diagonal $H$-cell satisfying these conditions, as we will see. 

This brings us to the last reduction step. Any diagonal $H$-cell $\H$ inside a two-sided cell $\J$ 
contains a special indecomposable $1$-morphism $\mathrm{D}$, called the {\em Duflo involution}, which has a natural 
algebra structure in $\cC_\H$. Consider the bicategory of biprojective $\mathrm{D}\text{-}\mathrm{D}$-bimodule $1$-morphisms 
in $\cC_\H$. This bicategory has a maximal two-sided cell containing the (isomorphism classes of the) indecomposable projective bimodule $1$-morphisms, which can be identified with $\H$ by a double centraliser theorem \cite[Section 5]{MMMTZ}, and we denote the $H$-simple subquotient bicategory by $\cB_\H$.
In the finitary setting, if $D$ is a simple separable Frobenius algebra $1$-morphism, then $\cB_\H$ is a fusion category (with a unique cell) and there is a biequivalence between the $2$-category of simple birepresentations of $\cC_\H$ with apex $\H$ and the $2$-category of simple birepresentations of $\cB_\H$, thanks to a {\em Morita context} in which the double centraliser theorem and 
strong $H$-reduction can be applied. 

This biequivalence yields a significant reduction of the Classification Problem for Soergel bimodules $\cS=\cS(W)$ associated to 
a finite Coxeter group $W$, because $\mathrm{D}$ is always a simple separable Frobenius algebra $1$-morphism in that case, see \cite[Section 6]{MMMTZ2}. Moreover, by \cite[Proposition 7.5]{MMMTZ2}, the fiab bicategory $\cB_\H$ is biequivalent to the trivial $\bbZ$-cover of Lusztig's (one-object) semisimple {\em asymptotic bicategory} $\cA_\H$ associated to $\H$, which is almost always biequivalent to one of the well-known fusion categories whose simple birepresentations have been completely classified 
in the literature, see \cite[Section 8]{MMMTZ2} and the references therein. The only exceptions occur in Coxeter types $H_3$ and $H_4$, where some two-sided cells do not contain a diagonal $H$-cell $\H$ for which the structure of $\cA_\H$ is known, let alone its simple birepresentations.  

In the almost finitary setting, the story is more complicated. Provided $\H$ and $D$ satisfy the aforementioned conditions, 
the bicategory $\cB_\H$ is a semisimple almost fiab one-object bicategory (an almost fusion category), and there is still a Morita context, in which we can apply the analogous double centraliser theorem and strong $H$-reduction, so that 
there is a biequivalence between the $2$-category of simple birepresentations of $\cC_\H$ with apex $\H$ and the $2$-category of simple birepresentations of $\cB_\H$.

However, in the case of Soergel bimodules for an infinite Coxeter group $W$ of finite rank, we can only prove that $\H$ and $D$ 
satisfy the aforementioned conditions when $D$ is associated to the longest element of a finite parabolic subgroup of $W$. In this case, the bicategory $\cB_\H$ is again biequivalent to the trivial $\bbZ$-cover of Lusztig's almost fusion category $\cA_\H$. 
We conjecture that these results hold for all diagonal $H$-cells in any two-sided cell of any Coxeter group $W$ of finite rank, but we do not 
know how to prove this in general at the moment. 

Given a two-sided cell $\J$ in any Coxeter group $W$ of finite rank, the asymptotic bicategories $\cA_\H$, for all diagonal $H$-cells $\H$ in $\J$, are Morita biequivalent to each other. This is a consequence of the fact that $\cA_\J$, for the whole $\J$, always satisfies the aforementioned conditions for strong $H$-cell reduction, as we will show. In practice, one can thus choose one diagonal $H$-cell $\H_1$ in $\J$ for which one checks the above-mentioned conditions, and another diagonal $H$-cell $\H_2$ in $\J$ for which the precise structure of $\cA_{\H_2}$ is known, to address the Classification Problem at the asymptotic level. This is exactly what we will do in (extended) affine type $A$. As a matter of fact, in this case, the asymptotic bicategories $\cA_\H$, for all diagonal $H$-cells in $\J$, are even biequivalent (not just Morita biequivalent).

If $W$ is an (extended) affine Weyl group, then $\cA_\H$ is biequivalent to $\mathrm{Rep}(G)$, the monoidal category of finite dimensional rational representations of some reductive complex algebraic group $G$, as was shown in \cite{Be} and \cite{BO}. In particular, in extended affine type $A$, every two-sided cell contains a Duflo involution associated to the longest element of a finite parabolic subgroup, so our reduction steps hold for every apex, and $G$ is always a finite product of complex general linear groups. Although there is no general solution of the Classification Problem for $\mathrm{Rep}(G)$ either, there are interesting families of examples, as we will recall in Section~\ref{secsimpleextA}, which all give rise to simple birepresentations of $\cS$ by our results.   

This paper has two different parts (each containing several sections). In the first part, 
we develop the general theory of almost finitary birepresentations of almost finitary bicategories along the lines sketched above. In the second part, we show how these general results can be used to address the Classification Problem for Soergel bimodules of certain infinite Coxeter types. Since these bimodules naturally form a monoidal category (a one-object bicategory), we will switch to the terminology of monoidal categories and monoidal functors, although we will continue to use the terms birepresentation and morphism of birepresentations. We hope that this makes the second part more accessible to a wider audience, perhaps less familiar with bicategory theory, but we think that it is important to develop the abstract theory in the most general setting, because there 
are examples of bicategories with more than one object, such as Kac--Moody 2-categories and singular Soergel bimodules, whose birepresentations can also be studied with the results in the first part of this paper, although we will not do this here. Note also that the aforementioned Morita context involves a bicategory with two objects.  Moreover, as we explain in Remark~\ref{rem:asympJ}, there is a technical reason why $\cA_\J$ has to be seen as an almost finitary bicategory with countably infinitely many objects if $\J$ contains infinitely many left cells (as is known to be the case for some two-sided cells in 
certain infinite Coxeter groups), which is relevant for Theorem~\ref{thm:mainthmforJ}, so the theory really does require us to work in this more general setting.

{\bf Outline of the article.} In Section \ref{secaddcat}, we provide the necessary background on completions of additive $\Bbbk$-linear categories. In Section \ref{secbicatrepbg}, we recall some background on bicategories and birepresentations, adapted to our setting, before introducing their appropriate completions in Section \ref{sec:dayconv}. In Section \ref{inthom}, we internalise almost finitary birepresentation via algebra $1$-morphisms in the completions and prove our $H$-cell reduction theorem. In Section \ref{secalgcop}, we introduce further finiteness conditions that allow us to choose algebra $1$-morphisms as coproducts. In Section \ref{mainabs}, we prove our double centraliser theorem and the resulting biequivalence of $2$-categories of birepresentations described above. In Section \ref{sec:Soergel}, we consider the application of our results to Soergel bimodules, for cells containing the longest element of a finite parabolic. In Section \ref{secafftypeA}, we study in more detail the example of (extended) affine type $A$ before concluding with a short section on universal Coxeter groups.

{\bf Acknowledgments.} We thank Ben Elias, Nicol\'as Libedinsky and Victor Ostrik for many helpful email exchanges on various topics related to this paper over the last couple of years.

M.M. \ has been supported by CIDMA (https://ror.org/05pm2mw36)
under the Portuguese Foundation for Science and Technology (FCT, https://ror.org/00snfqn58), Grants 
UID/04106/2025 (https://doi.org/10.54499/UID/04106/2025) and UID/PRR/04106/2025 (https://doi.org/10.54499/UID/PRR/04106/2025). 
V.M.\ has been supported by EPSRC grants EP/S017216/1 and EP/Z533750/1.

\section{Additive $\Bbbk$-linear categories and their completions}\label{secaddcat}

Let $\Bbbk$ be an algebraically closed field. 

\subsection{Additive categories}

We denote by $\mathfrak{A}_{\Bbbk}$ the $2$-category whose objects are additive $\Bbbk$-linear (i.e. $\Vect_{\Bbbk}$-enriched)  categories, whose $1$-morphisms are $\Bbbk$-linear functors and whose $2$-morphisms are natural transformations.

Let $\C$ be an additive $\Bbbk$-linear category. For a subset $\tX$ of objects in $\C$, we denote by $\tX^{\oplus}_\C$ the additive closure of $\tX$ in $\C$, meaning the full subcategory of $\C$ whose objects are direct summands of direct sums of objects in $\tX$. If no confusion is possible about the category in which the additive closure is taken, we also omit the subscript and write  $\tX^{\oplus}$.

We write $\ov{\C}$ for the category
\begin{itemize}
\item  whose objects are diagrams of the form $X_1\xrightarrow{x}X_0$ in $\C$;
\item  whose morphism spaces consist of pairs $(\phi_0,\phi_1)$ of morphisms in $\C$ producing solid commutative diagrams of the form
\begin{equation*}
\vcenter{\hbox{
\xymatrix{ X_1 \ar[rr]^{x} \ar[d]^{\phi_1}&&\ar@{-->}[dll]_{\eta} X_0 \ar[d]^{\phi_0}  \\ 
Y_1\ar[rr]^{y} &&Y_0,
}}}
\end{equation*}
modulo the subspace generated by diagrams where there exists a morphism $\eta$, as indicated by the dashed arrow, such that $\phi_0=y\eta$.
\end{itemize}

\subsection{The presheaf category}\label{sec:presheaf}

Let $\C$ be a small additive $\Bbbk$-linear category. 

We note that $\Vect_{\Bbbk}$ is complete and cocomplete with respect to small weighted (and hence, in particular, conical) limits and colimits. This follows from \cite[Corollary 7.6.4]{Ri} using \cite[Example 3.7.5]{Ri} and the fact that $\Vect_{\Bbbk}$ is enriched over itself. 

This implies that the category of $\Bbbk$-linear functors $\cpl\C =\F un_\Bbbk(\C^{op},\Vect_{\Bbbk})$,  is also (co)complete under weighted (co)limits by \cite[Section 3.3]{Ke} with (co)limits being computed object-wise. Moreover, given that (co)limits are computed in the abelian category $\Vect_{\Bbbk}$,  the category $\cpl\C$ is also abelian.

As usual, we identify an object $X$ in $\C$ with its representables $X^\vee=\Hom_{\C}( -, X)$ under the Yoneda embedding. We denote by $\C^\vee$ the image of $\C$ under this embedding (which is, of course, equivalent to $\C$).

Note that we can realise a product indexed by a set $J$ as the conical limit over the diagram $J \to \C\colon j \mapsto X_j$. This is a special case of the weighted limit of the $\Vect_{\Bbbk}$-enriched functor $\tilde J \to \C$ where $\tilde J$ is the free $\Vect_{\Bbbk}$-enriched category on $J$ and the weight $\tilde J\to \Vect_{\Bbbk}$ is just the constant functor $j\mapsto \Bbbk$. Similarly, a coproduct is the conical colimit over the same diagram.
 
 Since $\C$ and hence $\cpl \C$ is additive, in particular has a zero object, there is always a canonical monomorphism 
 $$\nu_{(X_j)_{j\in J}} \colon \coprod_{j\in J} X_j \to \prod_{j\in J } X_j.$$ Explicitly, denoting by $\pi_i \colon \prod_{j\in J } X_j\to X_i$ and $\iota_i\colon X_i \to \coprod_{j\in J } X_j$ the canonical projections and injections, respectively, $\nu$ is defined by the condition that $$\pi_i\nu\iota_k = \begin{cases} \id_{X_i} & \text{if } i=k,\\ 0&\text{otherwise.} \end{cases}$$

 By the co-Yoneda Lemma, every object $X$ in $\cpl{\C}$ is isomorphic to a colimit of representables, namely
 $$X = \int^{D\in \C} \Hom_\C(-, D)\otimes X(D).$$
Expressed as a coequaliser, this means that for $C\in \C$, $X(C)$ is the cokernel of the map
$$\xymatrix{\underset{D_1,D_2\in \C}{\coprod} \Hom_{\C}(C,D_1)\otimes\Hom_{\C}(D_1,D_2)\otimes X(D_2)\ar^{\rho}[d]
\\
\underset{D\in \C}{\coprod} \Hom_{\C}(C,D)\otimes X(D)}$$
where $\rho(f\otimes g \otimes x)  = g\circ f \otimes x - f\otimes X(g)x$.

For a subset $\tX$ of objects of $\C$, we denote by $\cpl\tX$  the completion $\tX^{\oplus}$, i.e. the full subcategory of $\cpl{\C}$ whose objects are isomorphic to a cokernel of a map between coproducts of objects in $\tX$.

\subsection{Extensions of $\Bbbk$-linear functors}

A $\Bbbk$-linear functor $F\colon \C\to \D$ between additive $\Bbbk$-linear categories induces a restriction functor $F^*\colon\cpl{\D}\to \cpl{\C}$ by precomposition.

Define $F_* \colon  \cpl{\C}\to  \cpl{\D}$ by defining, for every $M\in  \cpl{\C}$, the evaluation of $F_*(M)$ at an object $Y\in \D$ by $F_*(M)=  \coker \psi$ where 
\begin{align*}
\coprod_{X_1,X_2\in \C} M(X_2)\otimes \Hom_\C(X_1,X_2)\otimes \Hom_\D(Y, FX_1)& \xrightarrow{\psi_Y} \coprod_{X\in \C} M(X)\otimes \Hom_\D(Y,FX) \\
m\otimes f\otimes g &\mapsto M(f)(m)\otimes g - m\otimes F(f)g.
 \end{align*}

{\bf Claim:} For $Z\in \C$, we have $F_*(Z^\vee) = (FZ)^\vee$, so $F_*$ indeed extends the Yoneda embedding.

{\it Proof of Claim.} For each $Y\in \D$, we have a morphism 
$$\coprod_{X\in \C} \Hom_\C(X,Z)\otimes \Hom_\D(Y,FX) \xrightarrow{\gamma_Y} \Hom_\D(Y,FZ) \colon f\otimes g \mapsto F(f)g$$
and it is easy to check that $\gamma_Y\psi_Y = 0$, so we obtain a unique morphism $F_*(Z^\vee) \to (FZ)^\vee$. This is clearly surjective, since $f\in \Hom_\D(Y,FZ)$ has preimage given by $\id_Z\otimes f$ in the $Z$-component and $0$ elsewhere.
Moreover, if $(f_X\otimes g_X)_{X\in \C}$ is such that $\sum_{X\in \C} F(f_X)g_X = 0$ (noting this makes sense, as an element in the coproduct $\coprod_{X\in \C} \Hom_\C(X,Z)\otimes \Hom_\D(Y,FX)$ of vector spaces only has finitely nonzero components), then $(f_X\otimes g_X)_{X\in \C}$ is the image under $\psi_Y$ of $(\id_Z\otimes f_X\otimes g_X)_{X\in \C}$ (which is a tuple in the domain of $\psi_Y$ where each component for $X_2\neq Z$ is zero).\endproof

For the convenience of the reader, we include a proof of the following well-known category-theoretic fact.
{\bf Claim:} $F_*(\coprod_{j\in J} Z_j^\vee) = \coprod_{i\in J} (FZ_j)^\vee$.

{\it Proof of Claim.}
By definition $F_*(\coprod_{j\in J} Z_j^\vee)(Y)$ is the cokernel of the morphism
\begin{align*}
\coprod_{X_1,X_2\in \C} \left(\coprod_{j\in J}\Hom_\C(X_2, Z_j)\right)\otimes &\Hom_\C(X_1,X_2)\otimes \Hom_\D(Y, FX_1)\\& \xrightarrow{\psi_Y} \coprod_{X\in \C} \left(\coprod_{j\in J}\Hom_\C(X, Z_j)\right)\otimes \Hom_\D(Y,FX) \\
(h_j)_{j\in J}\otimes f\otimes g &\mapsto (h_j\circ f)_{j\in J}\otimes g - (h_j)_{j\in J}\otimes F(f)g.
 \end{align*}
Under the isomorphisms
\begin{equation*}\begin{split}
\coprod_{X_1,X_2\in \C} &\left(\coprod_{j\in J}\Hom_\C(X_2, Z_j)\right)\otimes \Hom_\C(X_1,X_2)\otimes \Hom_\D(Y, FX_1)\\
&\cong \coprod_{j\in J}\big(\coprod_{X_1,X_2\in \C} \Hom_\C(X_2, Z_j)\otimes \Hom_\C(X_1,X_2)\otimes \Hom_\D(Y, FX_1)\big)
\end{split}\end{equation*}
and 
\begin{align*}
 \coprod_{X\in \C} \left(\coprod_{j\in J}\Hom_\C(X, Z_j)\right)\otimes \Hom_\D(Y,FX)\cong \coprod_{j\in J} \left( \coprod_{X\in \C}\Hom_\C(X, Z_j)\otimes \Hom_\D(Y,FX)\right)
 \end{align*}
 this corresponds to the map 
 \begin{equation*}\begin{split}(h_j\otimes f\otimes g)_{j\in J} &\mapsto (h_j\circ f\otimes g)_{j\in J} - (h_j\otimes F(f)g)_{j\in J}\\
 &= (h_j\circ f\otimes g - h_j\otimes F(f)g)_{j\in J},\end{split}\end{equation*}
hence the cokernel is indeed given by the coproduct of the cokernels and $F_*(\coprod_{j\in J} Z_j^\vee) = \coprod_{j\in J} (FZ_j)^\vee$.\endproof

\subsection{Coproduct completions of Krull-Schmidt categories}

We define $\copr{\C}$ as the full subcategory of $\cpl{\C}$ where we close $\C^\vee$ under set-indexed coproducts. We then denote by $\cop{\C}$ its idempotent completion.
 Note that both are again additive $\Bbbk$-linear categories.  Moreover, $\cop{\C}$ is closed under coproducts.

For a subset $\tX$ of objects of $\C$, we denote by $\copr\tX$ the full subcategory of $\copr{\C}$, whose objects are coproducts of objects in $\tX^{\oplus}$.
 
We call a $\Bbbk$-linear additive idempotent complete category which contains set-indexed coproducts {\em coproduct Krull--Schmidt} ({\em cop-KS} for brevity) provided every object can be written uniquely as $\coprod_{i\in I} X_i$ where each $X_i$ has local endomorphism ring. Here, by uniqueness we mean that if $\coprod_{i\in I} X_i= \coprod_{j\in J} Y_j$ are two decompositions of the same object into objects with local endomorphism rings, then there exists a bijection $I\xrightarrow{\sigma} J$ such that $X_i\cong Y_{\sigma(i)}$ for each $i\in I$.

\begin{lemma}\label{CtocopC}
Assume $\C$ is Krull--Schmidt. Then $\copr{\C}= \cop{\C}$ is cop-KS. Moreover, the full subcategory of $\copr{\C}$ consisting of objects which are compact/finitely presented in $\cpl{\C}$ is equivalent to $\C$.
\end{lemma} 

\proof
By \cite[(1.2) Theorem]{CB}, the finitely presented objects in $\cpl{\C}$ are cokernels of morphisms between representable functors (i.e.\ equivalent to $\overline{\C}$). Such a cokernel is in $\copr{\C}$, i.e. a coproduct of representable functors, if and only if it is a representable functor itself and hence in $\C^\vee$, which is equivalent to $\C$. 

To see that $\copr{\C}$ is idempotent complete, note that the hypotheses of \cite[Theorem 4.1]{Br} apply to $\cop{\C}$, so $\cop{\C}$ is cop-KS with enough compact objects, the latter being the objects in $\C^\vee$. Hence every object in $\cop{\C}$ can be written (uniquely) as a coproduct of objects in $\C^\vee$. This implies that any object in $\cop{\C}$ is already in $\copr{\C}$ and hence $\copr{\C}$ is also cop-KS.
\endproof

\begin{lemma}\label{copCtoC}
Let $\D$ be a cop-KS category. Let $\C$ be the subcategory of finitely presented/compact objects in $\D$. Then $\C$ is Krull--Schmidt and $\D$ is equivalent to $\copr{\C}$.
\end{lemma}

\proof
By assumption, $\D$ is idempotent complete, so every object in $\D$ can be written uniquely as $\coprod_{i\in I} X_i$ where the $X_i$ have local endomorphism rings. It is easy to see that such an object is finitely presented if and only if this coproduct is finite. It is also easy to see that $\C$ inherits idempotent completeness from $\D$ and since every object in $\C$ can be written uniquely as a finite direct sum of objects with local endomorphism rings, it follows that $\C$ is Krull--Schmidt. The fact that $\D$ is equivalent to $\copr{\C}$ is immediate.
\endproof

A $\Bbbk$-linear functor $\rF\colon \C\to \D$ between additive $\Bbbk$-linear categories induces $\Bbbk$-linear functors (for which we use the same notation) $\rF\colon \copr{\C} \to \copr{\D}$, sending an object $\coprod_{j\in J}X_j$ to $\coprod_{j\in J}\rF(X_j)$, and hence $\rF\colon \cop{\C} \to \cop{\D}$.

\subsection{Almost finitary categories and cop-af categories}

We say that a category $\C$ is {\it almost finitary} if it is an essentially small additive $\Bbbk$-linear Krull--Schmidt
category with at most countably many isomorphism classes of indecomposable objects and morphism spaces are finite-dimensional. 
In the terminology of \cite{Macph}, this means that an almost finitary category is a wide finitary category that is hom-finite.
We denote by $\ind{\C}$ a set of representatives of isomorphism classes of indecomposable objects in $\C$.

We denote the $2$-category whose objects are almost finitary categories, whose $1$-morphisms are $\Bbbk$-linear functors, and whose $2$-morphisms are natural transformations by $\mathfrak{A}^{af}_{\Bbbk}$.

Note that for any almost finitary category $\C$, we have $\cop\C = \copr\C$ by Lemma \ref{CtocopC}.

We say a category $\C$ is {\it cop-af (coproduct almost finitary)} if it is an additive $\Bbbk$-linear cop-KS
category with at most countably many isomorphism classes of indecomposable objects and finite-dimensional morphism spaces between indecomposable objects.

For an additive $\Bbbk$-linear category $\C$, we denote the subcategory consisting of compact/finitely presented objects by $\C^{cp}$. Then
we immediately obtain the following corollary:

\begin{corollary}\label{wfcopKSf}
The maps $\C\mapsto \cop\C$ and $\D\mapsto \D^{cp}$ provide a bijection between
\begin{itemize}
\item almost finitary categories up to equivalence and
\item cop-af categories, up to equivalence.
\end{itemize}
\end{corollary}

\section{Almost finitary bicategories and birepresentations} \label{secbicatrepbg}

\subsection{Almost finitary bicategories}

We say that a bicategory $\cC$ is {\it almost finitary} if 
\begin{itemize}
\item it has countably many objects;
\item $\cC(\ti,\tj)$ is in $\mathfrak{A}^{af}_{\Bbbk}$ for all $\ti,\tj \in \cC$;
\item horizontal composition is biadditive and $\Bbbk$-linear;
\item the identity $1$-morphism $\one_\ti$ is indecomposable for any $\ti\in\cC$.
\end{itemize}
In other words, an almost finitary bicategory is a locally hom-finite wide finitary bicategory in the sense of \cite{Macph}.
 
We use $\circh$ resp.\ $\circv$ for horizontal respectively vertical composition in a bicategory. 

\begin{remark}
When giving explicit constructions, we will generally omit unitors and associators in the formulae, in order to not overload notation.
\end{remark}

An almost finitary bicategory $\cC$ is said to be {\it almost quasi-fiab} if there exists a weak
equivalence $(-)^*\colon \cC\to \cC^{\mathsf{co},\mathsf{op}}$ such that, for any $1$-morphism $\rF\in \cC(\ti,\tj)$, there are
natural $2$-morphisms $\varepsilon_\rF\colon \rF\rF^* \to\one_\tj, \eta_\rF \colon \one_\ti\to \rF^*\rF$ satisfying the usual adjunction triangles, i.e. $(\alpha\circh\id_{\rF})\circv (\id_\rF \circh\beta) = \id_\rF$ and $(\beta\circh\id_{\rF^*})\circv(\id_{\rF^*}\circh\alpha)=\id_{\rF^*}$.
We write ${}^*(-)$ for a weak inverse, i.e. ${}^*(\rF^*)\cong \rF$. If $-^*$ is a weak involution, we say $\cC$ is {\it almost fiab}.

For any bicategory $\cC$ and an object $\ti\in \cC$, we denote by $\cC_\ti$ the endomorphism bicategory of $\ti$, that is the bicategory with single object $\ti$ and $\cC_\ti(\ti,\ti) = \cC(\ti,\ti)$ with composition inherited from $\cC$.

\subsection{Almost finitary birepresentations}

Throughout this subsection, let $\cC$ be an almost finitary bicategory.

\begin{definition}
 An {\it almost finitary birepresentation} of $\cC$ is a pseudofunctor from $\cC$ to $\mathfrak{A}^{af}_{\Bbbk}$.
\end{definition}

\begin{definition}
\begin{itemize}
\item Let $\bfM$ be an almost finitary birepresentation of $\cC$ and $X\in \bfM(\ti)$. We denote by $\bfG_\bfM(X)$ the sub-birepresentation of $\bfM$ such that $\bfG_\bfM(\tj)  = \{\bfM(\rF)X\mid \rF\in \cC(\ti,\tj)\}^\oplus$. We call $\bfG_\bfM(X)$ the sub-birepresentation of $\bfM$ {\it generated by} $X$.
\item An {\em ideal} $\bfI$ of an almost finitary birepresentation  $\bfM$ of an almost finitary bicategory $\cC$ consists of the data of an ideal $\bfI(\ti)\subset \bfM(\ti)$ for each $\ti \in \cC$, which is invariant under the action of $\cC$, meaning $\bfM(\rF)(f)\in \bfI(\tj)$ for any $\rF\in\cC(\ti,\tj)$, provided $f\in\bfI(\ti)$. Given an ideal $\bfI$ of $\bfM$, we can form the {\em quotient birepresentation} $\bfM/\bfI$ with $\left(\bfM/\bfI\right)(\ti) = \bfM(\ti)/\bfI(\ti)$ carrying the induced action of $\cC$.
\end{itemize}
\end{definition}

\begin{definition}
Let $\bfM$ be an almost finitary birepresentation of $\cC$.
\begin{itemize}
\item If there exists an $X$ in some $\bfM(\ti)$ with $\bfM=\bfG_{\bfM}(X)$, we say $\bfM$ is {\em cyclic}.
\item If $\bfM=\bfG_{\bfM}(X)$ for any $X$ in any $\bfM(\ti)$, we say $\bfM$ is {\em transitive}.
\item If $\bfM$ has no proper ideals, we say $\bfM$ is {\em simple}. 
\end{itemize}
\end{definition}

Note that, in particular, every simple birepresentation is necessarily transitive. This justifies the change of terminology in this paper compared with previous papers by the authors and collaborators, where simple birepresentations were called simple transitive birepresentations.

\subsection{Cells and cell birepresentations}

In this subsection, let $\cC$ be an almost finitary bicategory.

As in the finitary case, we define the partial preorders $\leq_L, \leq_R, \leq_J$ and the corresponding equivalence relations. The equivalence classes are called left, right and two-sided {\em cells}, respectively, and are partially ordered. The natural birepresentation $\bfN_\L$ on the left ideal of $\cC$ generated by the $\id_\rF$ for all $\rF$ in a fixed left cell $\L$ has, as in the finitary case, a unique maximal ideal $\bfI_\L$, and the simple quotient $\bfC_\L:=\bfN_\L/\bfI_\L$ is called the \emph{cell birepresentation} associated to $\L$. See \cite[Section 3.2]{Macph} for details on the slightly more general case of wide finitary bicategories.

When $\cC$ is almost quasi-fiab, the weak equivalence sends left cells to right cells and vice-versa, and preserves 
two-sided cells. When $\cC$ is almost fiab, it also preserves {\em diagonal $H$-cells}. The latter are intersections of the form $\L\cap \L^*$, where $\L$ is a left cell. Finally, we call a two-sided cell $\J$ in an almost quasi-fiab {\em regular} if no distinct left (or equivalently right) cells contained in $\J$ are comparable in the left (resp.\ right) partial order.

\subsection{$\J$-simple quotients and $\cC_\H$}\label{sec:Jsimple}
Let $\cC$ be an almost finitary bicategory and fix a two-sided cell $\J$ in $\cC$. By the same arguments as in the finitary case (which only use the Krull--Schmidt property of the underlying categories), the quotient of $\cC$ by the ideal generated by all identities $\id_\rF$ for $\rF \nleq \J$ has a unique maximal biideal not containing any identity $2$-morphisms, and we denote by $\cC_{\leq \J}$ the quotient of $\cC$ by this ideal. If $\cC=\cC_{\leq \J}$, we say $\cC$ is \emph{$\J$-simple}. In general, we call $\cC_{\leq \J}$ the \emph{$\J$-simple quotient} of $\cC$.

We define $\cC_\J$ as the $2$-full sub-bicategory of $\cC_{\leq \J}$ whose objects are those appearing as sources or targets of $1$-morphisms in $\J$ and whose $1$-morphisms are in the additive closure of the identity $1$-morphisms and those in $\J$.

If $\cC$ is almost fiab and $\H$ is a diagonal $H$-cell in $\J$ with source and target $\ti$, we denote by $\cC_\H$ the $2$-full sub-bicategory of $\cC_\J$ with object $\ti$ and $1$-morphisms in the additive closure of $\one_\ti$ and those in $H$.

\subsection{Apex} 

\begin{definition}\label{cellfin}
We say that $\cC$ is {\em J-artinian} if every ascending chain of two-sided cells stabilizes.
\end{definition}

\begin{remark}
The choice of terminology is motivated by the fact that $\cC$ being J-artinian implies that there are no infinite descending chains of bi-ideals generated by identities on $1$-morphisms, which implies that the birepresentation theory of $\cC$ behaves in an artinian fashion.
\end{remark}

\begin{lemma}\label{lem:cellfin}
Let $\cC$ be an almost finitary $J$-artinian bicategory and $\bfM$ a transitive almost finitary birepresentation of $\cC$. Denote by $\cC_\bfM$ the quotient of $\cC$ by the biideal annihilated by $\bfM$. Then $\cC_{\bfM}$ has a unique maximal two-sided cell, where maximal means w.r.t. the partial order on two-sided cells induced by $\leq_J$.
\end{lemma}

\proof This is similar \cite[Lemma 1]{CM}, but given they work with direct sums over all (but finitely many) indecomposable objects up to isomorphism, we give a proof.

Note that $\cC_\bfM$ inherits the $J$-artinian property from $\cC$, so $\cC_\bfM$ has at least one maximal two-sided cell. Now, assume that $\J_1,\J_2$ are two distinct maximal two-sided cells in $\cC_{\bfM}$ and let $X\in \bfM(\ti)$. The assumptions imply that there exists an $\rF\in \J_1$ with $\rF\, X\neq 0$. Indeed, if $\rF\,X=0$ for all $\rF\in \J_1$, then by maximality of $\J_1$ in $\cC_\bfM$, $\rF\rH \,X = 0$ for all $1$-morphisms $\rH$ in $\cC_\bfM$ (since $\rF\rH\in \J_1$ or $\rF\rH=0$), but by transitivity of $\bfM$, the additive closure of the $\rH\,X$ is equal to $\coprod_{\ti\in\cC}\bfM(\ti)$, which is a contradiction to $\J_1$ not being annihilated by $\bfM$. 

Applying the same argument to $\rF\, X$, there exists a $\rG\in \J_2$ such that $\rG\rF\, X\neq 0$. But since $\J_1,\J_2$ are both maximal and distinct, $\rG\rF=0$, contradiction. 
\endproof

\section{Completions of almost finitary bicategories and birepresentations}
\subsection{Day convolution}\label{sec:dayconv}

Let $\cC$ be an almost finitary bicategory. We define the bicategory $\cpl{\cC}$ to be the bicategory on the same objects as $\cC$, but with morphism categories $\cpl{\cC}(\ti,\tj) =\cpl{ \cC(\ti,\tj)}$ and the horizontal composition given by Day convolution.

In particular, for $X\in \cpl{ \cC(\tj,\tk)},Y\in \cpl{ \cC(\ti,\tj)}$, the horizontal composition is thus defined as the coend
$$X\circ Y = \int^{\rG\in \cC(\ti, \tj),\rF\in\cC(\tj,\tk)} \Hom_{\cC(\ti,\tk)}(-, \rF\rG )\otimes_\Bbbk X(\rF) \otimes_\Bbbk Y(\rG).$$
Explicitly, we can describe the evaluation of $X\circ Y$ at $\rH\in \cC(\ti,\tk)$ as the cokernel of {\tiny
$$
\xymatrix{
 \underset{\substack{\rG_1,\rG_2\in \cC(\ti, \tj)\\ \rF_1,\rF_2\in \cC(\tj, \tk)}}{\coprod} \Hom_{\cC(\tj,\tk)} (\rF_1,\rF_2 ) \otimes \Hom_{\cC(\ti,\tj)} (\rG_1, \rG_2 ) \otimes
 \Hom_{\cC(\ti,\tk)} (\rH,\rF_1\rG_1 ) \otimes X(\rF_2)\otimes Y(\rG_2)
  \ar^{\psi}[d]\\
\underset{\rG\in \cC(\ti, \tj), \rF\in \cC(\tj, \tk)}{ \coprod}  \Hom_{\cC(\ti,\tk)} (\rH,\rF\rG ) \otimes_\Bbbk X(\rF)\otimes_\Bbbk Y(\rG).}$$
}
where $$\psi(\alpha\otimes\beta\otimes\gamma\otimes x\otimes y) = \gamma\otimes X(\alpha)(x)\otimes Y(\beta)(y) - (\alpha\circh\beta)\circv \gamma \otimes x\otimes y.$$
One verifies directly that for representables $X^\vee, Y^\vee$, we have $X^\vee \circ Y^\vee \cong (X\circ Y)^\vee$.

\subsection{Coproduct completions of bicategories}

Let $\cC$ be a locally $\Bbbk$-linear additive bicategory. We define the bicategory $\copr{\cC}$, resp.\ $\cop{\cC}$, to be the bicategory on the same objects as $\cC$, but with morphism categories $\copr{\cC}(\ti,\tj) =\copr{ \cC(\ti,\tj)}$, resp.\  $\cop{\cC}(\ti,\tj) =\cop{ \cC(\ti,\tj)}$, with horizontal composition given component-wise, i.e. 
$$\coprod_{i\in I}\rX_i\coprod_{j\in J}\rY_j = \coprod_{\substack{i\in I\\j\in J}}\rX_i\rY_j $$
and similarly for morphisms. Observe that $\copr{\cC}$ and $\cop{\cC}$ are locally additive and $\Bbbk$-linear, and that they naturally embed into the completion $\cpl{\cC}$. Observe also that, by Lemma \ref{CtocopC}, $\copr{\cC}=\cop{\cC}$ provided $\cC$ is almost finitary.

\subsection{Completions of almost finitary birepresentations}

In analogy to Section \ref{sec:dayconv}, we can extend a $\Bbbk$-linear additive birepresentation $\bfM$ of an almost finitary bicategory $\cC$ to a complete birepresentation $\cpl{\bfM}$ of $\cpl{\cC}$ by defining the action of $\rF\in \cpl{ \cC}(\ti,\tj)$ on $X\in \cpl{\bfM}(\ti) = \cpl{\bfM(\ti)}$ as the corresponding Day convolution. 

For any $\Bbbk$-linear additive birepresentation $\bfM$ of $\cC$, we can also consider the coproduct completions $\copr\bfM$ and $\cop\bfM$, which are defined by $\copr\bfM(\ti) = \copr{\bfM(\ti)}$ and $\cop\bfM(\ti) = \cop{\bfM(\ti)}$ with the obvious actions of $\copr\cC$ and $\cop\cC$, respectively. Note that by Lemma \ref{CtocopC}, $\cop\bfM = \copr\bfM$ provided $\bfM$ is almost finitary.

We say that a $\Bbbk$-linear additive birepresentation $\bfM$ of $\cC$ is {\it cop-af}, provided each $\bfM(\ti)$ is cop-af. By Corollary \ref{wfcopKSf}, there is a bijection between almost finitary birepresentations up to equivalence and cop-af birepresentations up to equivalence.

{\it Morphisms} of $\Bbbk$-linear additive birepresentations are given by strong natural transformations. We say that a morphism $\Psi\colon \bfM\to\bfN$  is {\it exact} if it extends to an exact morphism  $\cpl{\Psi}\colon \cpl{\bfM}\to\cpl{\bfN}$ between the completions, meaning that each component functor is an exact functor. 

\subsection{$2$-Categories of birepresentations}

For an almost finitary bicategory $\cC$, we denote by $\cC\wfmod$ the $2$-category of almost finitary birepresentations of $\cC$, with  exact morphisms of birepresentations as $1$-morphisms and modifications as $2$-morphisms. We denote by $\cC\cwfmod$ and $\cC\swfmod$ the $1$-full and $2$-full sub-$2$-categories of  cyclic and simple almost finitary birepresentations, respectively.

We further denote by $\cC\copwfmod$ the $2$-category of cop-af birepresentations of $\cC$, again with exact morphisms of birepresentations as $1$-morphisms and modifications as $2$-morphisms. Similarly to before, we denote by $\cC\ccopwfmod$ and  $\cC\scopwfmod$ the $1$-full and $2$-full sub-$2$-categories of coproduct completions of cyclic and simple birepresentations, respectively.

For a two-sided cell $\J$ of $\cC$, we denote by $\cC\swfmod_\J$ and $\cC\scopwfmod_\J$ the $1$-full and $2$-full sub-$2$-categories of 
$\cC\swfmod$ and $\cC\scopwfmod$ consisting of birepresentations with apex $\J$, respectively.

We observe that, in general, while there is a bijection between equivalence classes of objects in $\cC\wfmod$ and $\cC\copwfmod$ by the remark in the previous subsection, there is no biequivalence of $2$-categories in general, since not every morphism of cop-af birepresentations needs to restrict to one between the underlying almost finitary ones.

\section{Birepresentations via algebra $1$-morphisms}

\subsection{Internal hom}\label{inthom}

Let $\cC$ be an almost finitary bicategory and $\bfM$ an almost finitary birepresentation of $\cC$.

For $X\in \bfM(\ti)$, the evaluation $\Ev_X\colon \cC(\ti,\tj)\to \bfM(\tj), \rF\mapsto \bfM(\rF)\,X$ extends to the completions, yielding a functor $\Ev_X\colon \cpl{\cC}(\ti,\tj)\to \cpl{\bfM}(\tj), \rF\mapsto \cpl{\bfM}(\rF)\,X^\vee$. 
Note that both $\cpl{\cC}(\ti,\tj)$ and $\cpl{\bfM}(\tj)$ are locally presentable (in fact, both are Grothendieck using e.g. \cite[Theorem 14.2]{Fa} and the fact that $\Vect_{\Bbbk}$ is Grothendieck). Moreover, by construction of the action by Day convolution and hence as a coend, $\Ev_X$ preserves all small colimits. The adjoint functor theorem for locally presentable categories thus implies that $\Ev_X$ has a right adjoint $[X^\vee, -]$, called the internal hom. In particular, there is an isomorphism
$$\Hom_{\cpl{\cC}(\ti,\tj)}(\rF, [X^\vee, Y]) \cong   \Hom_{\cpl{\bfM}(\tj)}( \cpl{\bfM}(\rF)\,X^\vee, Y)$$
natural in $\rF\in  \cpl{\cC}(\ti,\tj)$ and $Y\in \cpl{\bfM}(\tj)$. To not overload notation, we will again identify an  object $X\in \bfM(\tj)$ with its representable $X^\vee\in \cpl\bfM(\tj)$ and henceforth write $[X,-]$ for the adjoint to $\Ev_X$.

We wish to describe the object $[X, Y]$ for $Y\in \bfM(\tj)$ better. Let $\mathtt{C}_{\ti,\tj}$ denote a set of representatives of isomorphism classes of indecomposables in $\cC(\ti,\tj)$ and, for each $\rF\in \mathtt{C}_{\ti,\tj}$, choose a basis $\{\alpha^\rF_k\,\vert\, k \in \tB_\rF\}$ (for some indexing set  $\tB_\rF$) of $\Hom_{\bfM(\tj)}(\rF X, Y)$. Set  
$$\rT_Y = \coprod_{\rF\in \mathtt{C}_{\ti,\tj}} \coprod_{k\in\tB_\rF} \rF_{(k)}$$ 
where each $\rF_{(k)}$ is just isomorphic to $\rF$ and let $f_Y \colon \rT_Y X\to Y$ be the unique morphism given by $\alpha^\rF_k$ on the component $\rF_{(k)}$. 

Then the (dual of the) {\em solution set condition} is satisfied in the sense that, for every $\rG\in \cC(\ti,\tj)$ and $g\colon \rG X\to Y$, there exists an $h\colon \rG\to \rT_Y$ such that the diagram
$$\xymatrix{
\rG X\ar_{\Ev_X(h)}[dr] \ar^g[rr]&& Y\\
&\rT_Y X \ar^{f_Y}[ur]&
}$$
commutes. Explicitly, writing $g = \sum_{k\in \tB_\rG} c_k\alpha^{\rG}_k$, such an h is given by having components $c_k \id_\rG \colon \rG\to \rG_{(k)},$ noting that this makes sense since only finitely many $c_k$ are nonzero.  Moreover, since for $\rG=\coprod_{i\in I} \rG_i$, maps $g\colon \rG X\to Y$ and $h\colon \rG\to \rT_Y$ are determined by their definitions on each component, the above condition is satisfied for any $\rG\in \copr\cC(\ti,\tj)$ and $g\colon \rG X\to Y$.

Thus, we have an epimorphism 
\begin{multline}
\Hom_{\copr\cC(\ti,\tj)}( - , \rT_Y) = \Hom_{\cpl\cC(\ti,\tj)}( - , \rT_Y) \tto \\ 
\Hom_{\cpl\bfM(\tj)}(\Ev_X(-), Y)\cong  \Hom_{\cpl\cC(\ti,\tj)}( - , [X,Y]).
\end{multline}
of functors defined on $\copr\cC(\ti,\tj)^{op}$. As a left exact functor, $ \Hom_{\cpl\cC(\ti,\tj)}( - , [X,Y]) $ is defined uniquely on injectives in $\cpl\cC(\ti,\tj)^{op}$, which are projective in $\cpl\cC(\ti,\tj)$, i.e. by Section \ref{sec:presheaf}, on $\copr\cC(\ti,\tj)$.
Consequently, we have an epimorphism of functors
$$\Hom_{\cpl\cC(\ti,\tj)}( - , \rT_Y) \tto \Hom_{\cpl\bfM(\tj)}(\Ev_X(-), Y)\cong  \Hom_{\cpl\cC(\ti,\tj)}( - , [X,Y])$$
defined on all of $\cpl\cC(\ti,\tj)^{op}$.

By the Yoneda lemma, this implies that $[X,Y]$ is a quotient of $\rT_Y$, and hence can be written as a cokernel of
$$ \rZ \xrightarrow{\zeta} \rT_Y$$
for some $\rZ\in \copr\cC(\ti,\tj)$.

\begin{remark}
In \cite{Macph}, the author proves the existence of the internal cohom in a pro-completion of $\cC(\ti,\tj)$, and we would analogously be able to obtain the  internal hom in an ind-completion. However, the chosen setup is better suited to proving the results in Section \ref{secalgcop}.
\end{remark}

\subsection{Categories of module $1$-morphisms and internalisation}\label{intsec}

For objects $\ti,\tj$ in a bicategory $\cC$ and an algebra $1$-morphism $\rA\in \cC(\ti,\ti)$, we denote by $(\rmod_\cC\rA)_\tj$ the category of right module $1$-morphisms over $\rA$ inside $\cC(\ti,\tj)$. The collection of these naturally carry the structure of a birepresentation of $\cC$ by left multiplication, which we denote by $\bfrmod_\cC\rA$.

The standard free-forgetful adjunction gives rise to isomorphisms
\begin{equation}\label{freeforget}
\begin{array}{rcl}
\Hom_{{\cC}(\ti,\tj)}(\rF,\rM)& \cong& \Hom_{(\rmod_\cC\rA)_\tj} (\rF\rA,\rM)\\
f &\mapsto &\rho_\rM \circv(f\circh \id_\rA)\\
g\circv(\id_\rF\circh \iota_\rA)&\mapsfrom &g\\
\end{array}
\end{equation}
for $F\in {\cC}(\ti,\tj)$ and $\rM\in(\rmod_\cC\rA)_\tj$.

The usual arguments show that, given an almost finitary birepresentation $\bfM$ of an almost finitary bicategory $\cC$ and an object $X\in \bfM(\ti)$, the internal hom $\rA = [X, X]$ has the structure of an algebra $1$-morphism in $\cpl{\cC}(\ti,\ti)$ and that $[X,Y]$ is a right module $1$-morphism over it. Thus $[X, -]$ defines a functor from $\cpl\bfM(\tj)$ to $(\rmod_{\cpl\cC}\rA)_\tj$ for each $\tj\in \cC$.

\begin{lemma}\label{intmor}
Let $\cC$ be an almost quasi-fiab bicategory, $\bfM$ an almost finitary birepresentation of $\cC$ and $X\in \bfM(\ti)$.
\begin{enumerate}[(a)]
\item\label{intmor2}  The functors $[X, -]$ assemble to a morphism $[X, -]\colon \cpl\bfM \to \bfrmod_{\cpl\cC}\rA$  of birepresentations of $\cC$ .
\item\label{intmor3} If $X$ is a generator for $\bfM$, the restriction of  $[X, -]$ to $\bfM$ is fully faithful.
Setting $\rA = [X,X]$, the essential image of the restriction of  $[X, -]$ to $\bfM(\tj)$ is given by $\{ \rF\rA\,\vert\, \rF\in \cC(\ti,\tj)\}^\oplus_{(\rmod_{\cpl\cC}\rA)_\tj}$.
\end{enumerate}
\end{lemma}

\proof
\eqref{intmor2} Mutatis mutandis from the finitary situation for coalgebras. The key ingredient is that the existence of adjunctions provides, for all $\rF\in \cpl\cC(\ti,\tk), \rG\in \cC(\tj, \tk), Y\in \bfM(\tj)$,  an isomorphism
\begin{equation*}\begin{split}
\Hom_{\cpl\cC(\ti,\tk)}(\rF, [X, \rG Y])&\cong \Hom_{\cpl\bfM(\tk)}(\rF X, \rG Y)\\
&\cong  \Hom_{\cpl\bfM(\tj)}( \rG^*\rF X, Y)\\
&\cong \Hom_{\cpl\cC(\ti,\tj)}(\rG^*\rF, [X,  Y])\\
&\cong  \Hom_{\cpl\cC(\ti,\tj)}(\rF, \rG[X,  Y])
\end{split}\end{equation*}
which implies an isomorphism $[X, \rG Y]\cong \rG[X,  Y]$ for all $\rG\in \cC(\tj, \tk)$. The fact that these isomorphisms satisfy the required naturality conditions follows as in \cite[Lemma 4.4]{MMMT}.

\eqref{intmor3} Since $X$ is a generator, it suffices to verify fullness and faithfulness for objects of the form $\rF X$. For these, it follows in a similar way as in the proof of \cite[Theorem 4.7]{MMMT} or \cite[Proposition 4.9]{LM1}. For the second claim, recall that since $X$ is a generator of $\bfM$, we have $\bfM(\tj) = \{ \rF X\,\vert\, \rF\in \cC(\ti,\tj)\}^\oplus_{\bfM(\tj)} .$ Since $[X,-]$ is fully faithful and both ${\bfM(\tj)}$ and $(\rmod_{\cpl\cC}\rA)_\tj$ are idempotent complete, the claim follows.
\endproof

\begin{corollary}
In the setup of Lemma \ref{intmor}, set $\bfM_\rA(\tj) = \{ \rF\rA\,\vert\, \rF\in \cC(\ti,\tj)\}^\oplus_{(\rmod_{\cpl\cC}\rA)_\tj}$. Then  we obtain a birepresentation $\bfM_\rA$ of $\cC$ together with an  equivalence of birepresentations $[X,-]\colon\bfM \to \bfM_\rA$.
\end{corollary}

\begin{lemma}\label{projectivesinmodcplC}
Assume $\cC$ is an almost quasi-fiab bicategory, $\bfM$ an almost finitary 2-representation of $\cC$, $X$ a generator for $\bfM$. Set $\rA = [X,X]$. Then
$$\bfrproj_{\cpl\cC}\rA  \cong  \cop\bfM_\rA$$
where $\bfrproj_{\cpl\cC}\rA $ denotes the restriction of $\bfrmod_{\cpl\cC}\rA$ to the subcategories of projective objects in each $(\rmod_{\cpl\cC}\rA)_\tj$.
\end{lemma}

\proof
First note that every object in $\cop\bfM_\rA(\tj) = \cop{\{ \rF\rA\,\vert\, \rF\in \cC(\ti,\tj)\}}_{(\rmod_{\cpl\cC}\rA)_\tj}$ is projective in $(\rmod_{\cpl\cC}\rA)_\tj$ since $\rA$ is projective, the existence of adjunctions implies that left translation by $\rF$ preserves projectives, and coproducts and retracts of projectives are projective.
Let $X\in (\rmod_{\cpl\cC}\rA)_\tj$ be given as the cokernel of  $\coprod_{j\in J} X_{1,j}\xrightarrow{x}\coprod_{i\in I} X_{0,i}$ for objects $X_{1,j}, X_{0,i} \in \cC(\ti,\tj)$. Then,  as usual, we have a sequence of epimorphisms of right $A$-comodules
$$\xymatrix{
0 \ar[rr]\ar[d]&&\coprod_{i\in I} X_{0,i}\rA\ar^{\id}[d] \\
(\coprod_{j\in J} X_{1,j}\ar^{x}[rr]\ar[d]&&\coprod_{i\in I} X_{0,i})\rA\ar[d]\\
\coprod_{j\in J} X_{1,j}\ar^{x}[rr]&&\coprod_{i\in I} X_{0,i}
}$$
which shows that projective objects in $(\rmod_{\cpl\cC}\rA)_\tj$ are precisely those in the idempotent completion of objects of the form $\coprod_{i\in I} X_{0,i}\rA$, i.e. those in $\cop\bfM_\rA(\tj)$. \endproof
 
\begin{lemma}\label{lem:copKSbirep}
Assume $\cC$ is an almost quasi-fiab bicategory, $\bfM$ an almost finitary birepresentation of $\cC$, $X$ a generator for $\bfM$. Set $\rA = [X,X]$. Then each $\bfrproj_{\cpl\cC}\rA(\tj)  \cong  \cop\bfM_\rA(\tj) $ is locally cop-KS. Moreover, for each $\tj\in \cC$, $\bfM_\rA(\tj)$ consists of the compact objects in $\cop\bfM_\rA(\tj)$.
\end{lemma}

\proof
This follows immediately from Lemmas \ref{CtocopC} and \ref{copCtoC}.
\endproof

We say that two algebra $1$-morphisms $\rA\in \cpl{\cC}(\ti,\ti)$ and $\rA'\in \cpl{\cC}(\tj,\tj)$ are {\em 
Morita equivalent} if the birepresentations $\bfrmod_{\cpl\cC}\rA$ and  $\bfrmod_{\cpl\cC}\rA'$ are equivalent. By Lemma~\ref{lem:copKSbirep}, this is equivalent to the birepresentations $\bfrproj_{\cpl\cC}\rA \cong\cop\bfM_\rA$ and $\bfrproj_{\cpl\cC}\rA'\cong \cop\bfM_{\rA'}$ being equivalent. By the usual arguments, this is equivalent to the existence of an $\rA$-$\rA'$-bimodule $\rM$ in $\cpl{\cC}(\tj,\ti)$ and an $\rA'$-$\rA$-bimodule $\rN$ in $\cpl{\cC}(\ti,\tj)$ such that $\rM\circ_{\rA'}\rN\cong \rA$ and $\rN\circ_{\rA}\rM\cong \rA'$. Here $\rM\circ_{\rA'}\rN$ denotes the relative horizontal composition, defined as the cokernel of 
$$\rM\rA'\rN \xrightarrow{\id_\rM\circh\lambda_{\rA',\rN} - \rho_{\rM,\rA'}\circh\id_\rN} \rM\rN $$
for $\rho_{\rM, \rA'}, \lambda_{\rA',\rN}$ the right, resp. left, $\rA'$-actions on $\rM$, resp. $\rN$. Note that in particular, $\rM$ and $\rN$ have to be biprojective, i.e.\ projective on both sides.

If $\bfM$ is an almost finitary birepresentation of an almost fiab bicategory $\cC$ and $X$ and $X'$ are two generators for $\bfM$, then, setting $\rA=[X,X]$ and $\rA'=[X',X']$, we have $\bfM_\rA\cong \bfM\cong \bfM_{\rA'}$ by Lemma \ref{intmor}, thus $\cop\bfM_\rA\cong  \cop\bfM_{\rA'}$, so $\rA$ and $\rA'$ are Morita equivalent. Conversely, if $\rA$ and $\rA'$ are Morita equivalent, then $\cop\bfM_\rA\cong  \cop\bfM_{\rA'}$, and hence $\bfM_\rA\cong \bfM\cong \bfM_{\rA'}$.

This shows that we have bijections between
\begin{itemize}
\item almost finitary birepresentations of $\cC$ up to equivalence;
\item algebra $1$-morphisms in $\cpl\cC$ up to Morita equivalence;
\item cop-KS birepresentations of $\cC$ up to equivalence.
\end{itemize}

We say that an algebra $1$-morphism $\rA$ in any bicategory $\cC$ is {\em simple} if every morphism of algebras $\rA \to \rB$ in $\cC$, which is epimorphic as a $2$-morphism, is necessarily an isomorphism.

\begin{lemma}\label{simpleifsimple}
Let $\cC$ be almost quasi-fiab and $\rA$ an algebra $1$-morphism in $\cpl{\cC}$.
Then $\bfM_\rA$ is simple if and only if the algebra $1$-morphism $\rA$ is simple.
\end{lemma}

\proof
Suppose $\varphi$ is a morphism of algebras $\rA \to \rB$ in $\cpl{\cC}$, which is epimorphic as a $2$-morphism. Then the endomorphisms of $\rA$ which become zero when composed with $\varphi$ generate an ideal in $\bfM_\rA$. Thus $\bfM_\rA$ being simple implies that every such $\varphi$ is an isomorphism. Conversely, assume $\Phi\colon \bfM_\rA\to \bfN$ is a full and dense morphism of birepresentations (i.e. $\bfN$ is a quotient of $\bfM_\rA$). Then by the same arguments as in \cite[Proposition 4.16 (a), (d)]{LM1}, we obtain a morphism of algebras $\rA \to [\Phi(\rA),\Phi(\rA)]$ in $\cpl{\cC}$, which is epimorphic as a $2$-morphism. Thus, this necessarily being an isomorphism implies that $\Phi$ is an equivalence and $\bfM_\rA$ is simple.
\endproof

\subsection{Framing}

Let $\cC$ be a $\J$-simple almost quasi-fiab bicategory and $\bfM$ a transitive almost finitary birepresentation of $\cC$ with apex $\J$. Assume $X\in \bfM(\ti)$ is a generator for $\bfM$ and let $\rC=[X,X]$ be the corresponding algebra $1$-morphism in $\cpl\cC(\ti,\ti)$ with multiplication $\mathsf{m}$  and unit $\mathsf{u}$. 

\begin{lemma}\label{framing}
Let $\rF\in \cC(\ti,\tj)$ such that $\rF X\neq 0$.
Then $[\rF X,\rF X] \cong \rF \rC\rF^*$ with multiplication and unit given by 
$$\tilde{\mathsf{m}}=  (\id_\rF\circh\mathsf{m}\circh\id_{\rF^*})\circv(\id_{\rF\rC}\circh \varepsilon_\rF\circh \id_{\rC\rF^*}  ) \quad \text{and} \quad \tilde{\mathsf{u}} = (\id_\rF\circh \mathsf{u} \circh \id_{\rF^*})\circv\eta_\rF,$$
respectively.
In particular, $\rC$ and $\rF\rC\rF^*$ are Morita equivalent.
\end{lemma}

\proof
It is a straightforward computation to see that this does define an algebra structure on $\rF\rC\rF^*$. The isomorphism $[\rF X,\rF X] \cong \rF \rC\rF^*$ as objects in $\cpl\cC(\ti,\ti)$ follows exactly dually to the finitary case for coalgebras (\cite[(4.6)]{MMMTZ}). The proof that the  algebra structure on $[\rF X,\rF X]$ from the internal hom construction is the given one on $\rF\rC\rF^*$ is also completely analogous to the finitary coalgebra setting. The Morita equivalence follows immediately from transitivity of $\bfM$.
\endproof

\begin{corollary}\label{notniceinJH}
\begin{enumerate}
\item\label{notniceinJ} There is an algebra $1$-morphism $\rA  \in \cpl{\J}$
with $\bfM\simeq \bfM_\rA$.
\item\label{notniceinH} Let $\H$ be a diagonal $H$-cell in $\J$. There is an algebra $1$-morphism $\rA \in \cpl{H}$ 
with $\bfM\simeq \bfM_\rA$.
\end{enumerate}
\end{corollary}

\proof
\begin{enumerate}
\item This follows directly from Lemma \ref{framing} by choosing $\rF\in \J$ with $\rF X\neq 0$ and taking $[\rF X,\rF X] \cong \rF \rC\rF^*$.
\item The same arguments as in \cite[Lemmas 2.32, 2.33]{MMMTZ} show that there exists some $\rF\in \H$ such that $\rF X\neq 0$. The statement now follows from Lemma \ref{framing}.
 \end{enumerate}
\endproof

\subsection{Bimodule bicategories and $H$-cell reduction}\label{bimodsec}

For any bicategory $\cC$, let $\cB_\cC$ denote the bicategory 
\begin{itemize}
\item whose objects are algebra $1$-morphisms in $\cC$;
\item whose morphism bicategory $\cB_\cC(\rA,\rB)$ is given by the category of biprojective $\rA$-$\rB$-bimodules, for any algebra $1$-morphisms $\rA$ and $\rB$;
\item where composition of $\rM\in \cB_\cC(\rA,\rB)$ and $\rN\in \cB_\cC(\rB,\rC)$ is given by $\rM\circ_\rA\rN$.
\end{itemize}

In analogy to \cite[Theorem 4.26]{MMMTZ}, for almost quasi-fiab $\cC$ we have a biequivalence 
\begin{equation}\label{2repbim}
\cB_{\cpl{\cC}} \simeq \cC\ccopwfmod
\end{equation}
given by 
\begin{align*}
\rA&\mapsto \cop\bfM_\rA\\
{}_\rA\rM_\rB&\mapsto -\circ_\rA \rM\\
\sigma&\mapsto \circ_\rA \sigma
\end{align*}
for an algebra $\rA$, a biprojective bimodule $\rM$ and a morphism $\sigma$ of bimodules.

For $\J$ a two-sided cell in $\cC$, we write $\cB^s_{\cpl\cC, \cpl{\J}}$ for the $1$-full $2$-full subbicategory of $\cB_{\cpl\cC}$ whose objects are simple algebra $1$-morphisms in $\cpl\J$.

\begin{proposition}\label{Hcellreduc} Let $\cC$ be a $\J$-simple almost fiab bicategory and $\H$ a diagonal $H$-cell in $\J$.
There are equivalences of bicategories
\[\cC \scopwfmod_\J\simeq\cC_\J \scopwfmod_\J  \simeq \cB^s_{\cpl\cC_\J, \cpl{\J}} \simeq \cB^s_{\cpl\cC_\H, \cpl{\H}} \simeq \cC_\H \scopwfmod_\H.\]
\end{proposition}

\proof
Under the biequivalence  \eqref{2repbim}, $\cC \scopwfmod_\J$ corresponds to $\cB^s_{\cpl\cC, \cpl{\J}}$ by Corollary \ref{notniceinJH} and Lemma \ref{simpleifsimple}. Moreover, $\cB^s_{\cpl\cC, \cpl{\J}} = \cB^s_{\cpl\cC_\J, \cpl{\J}}$, and the latter, again using Corollary \ref{notniceinJH}(i) and Lemma \ref{simpleifsimple}, is biequivalent to $\cC_\J \scopwfmod_\J$. Similarly $\cB^s_{\cpl\cC_\H, \cpl{\H}} \simeq \cC_\H \scopwfmod_\H$.
It now remains to observe that $\cC_\J \scopwfmod_\J \simeq \cB^s_{\cpl\cC_\H, \cpl{\H}}$ by Corollary \ref{notniceinJH}(ii), since the algebra $1$-morphisms can be chosen in $\cpl{\H}$ and it is easy to see that, given algebras $\rA,\rB$  in $\cpl\cC_\J$ which live in $\cpl{\H}$,  $\rA$-$\rB$-bimodules in $\cpl\cC_\J$ are the same as $\rA$-$\rB$-bimodules in $\cpl\cC_\H$.
\endproof

\section{Algebras as coproducts}\label{secalgcop}
\subsection{Finiteness conditions for almost finitary birepresentations}

In this section, we will introduce certain finiteness conditions on the (bi)categories we consider, which will allow us to only work with algebra $1$-morphisms in the coproduct completion of an almost finitary bicategory.

\begin{definition}\label{cosparsedef}
We call an almost finitary category $\C$ {\it cosparse} (cf.\  \cite[Definition 5.2.3]{Macphthesis}) if, for any object $Y\in \C$, there exist only finitely many indecomposable $X\in \C$ up to isomorphism with $\Hom_\C(X,Y)\neq 0$.
\end{definition}

\begin{lemma}\label{whenabelian}
If $\C$ is a cosparse almost finitary category, then $\ov{\C}$ is an abelian finite-length category.
\end{lemma}

\proof
To show that $\ov{\C}$ is abelian, we need to verify that weak kernels exist in $\C$. 
So let $X\xrightarrow{f} Y$  be a morphism in $\C$. Let $K_1, \dots, K_r$ be the indecomposable objects (up to isomorphism) such that $\Hom_\C(K_i,X)\neq 0$ and for each $i$ let $\{ \kappa^{(i)}_{1},\ldots,  \kappa^{(i)}_{t_i}\}$ be a basis of the subspace of $\Hom_\C(K_i,X)$ consisting of morphisms $k_i$ with  $fk_i= 0$. We claim that setting $K=\bigoplus_{i=1}^r \bigoplus_{j=1}^{t_i} K_i$ and defining $K \xrightarrow{\kappa} X$ as $\kappa^{(i)}_j$ on the component of the direct sum indexed by $(i,j)$, $\kappa$ is a weak kernel of $f$.

By definition $f\kappa = 0$, since its restriction to each direct summand of $K$  is zero. Now let $Z\xrightarrow{h} X$ be a morphism in $\C$ such that $fh=0$. Note that, without loss of generality, we may assume $Z$ is indecomposable. If $h=0$, it trivially factors over $\kappa$ via $Z\xrightarrow{0} K$, so assume $h\neq 0$. This implies $Z\cong K_s$ for some $s$ and $h = \sum_{j=1}^{t_s} c_j\kappa_j^{(s)}$ for some scalars $c_j$. Then defining $Z\xrightarrow{g} K$ as the unique morphism whose composition with projection onto the $(i,j)$ component is $c_j\id_{K_s}$ for $i=s$  and $0$ otherwise, we obtain $h = \kappa g$, which proves the claim, and hence the fact that $\ov{\C}$ is abelian.

Since $\C$ is cosparse, the total morphism space $\prod_{X}\Hom_{\C}(X,Y)$, where $X$ runs over a set of representatives of isomorphism classes of indecomposable objects in $\C$, is finite-dimensional, which implies that $\ov{\C}$ is a finite-length category.
\endproof

\begin{definition}\label{HfromfinFG} 
Let $\cC$ be an almost finitary bicategory.
\begin{enumerate}[(a)]
\item We say that a left cell $\L$ in $\cC$ is {\it connection-finite} if for any indecomposable $1$-morphisms $\rH$ in $\L$ and $\rF$ in the right cell of $\rH$, there exist at most finitely many indecomposable $1$-morphisms $\rG$ in $\L$ up to isomorphism such that $\rH$ appears as a direct summand of $\rF\rG$. 
\item We say that a two-sided cell $\J$ in $\cC$ is {\it adjoint connection-finite} if, for any pair of indecomposable $1$-morphisms $\rF,\rH$ in $\J$, there are at most finitely many $1$-morphisms $\rG$ in $\cC$  up to isomorphism such that $\rH$ appears as a direct summand of $\rF\rG\rF^*$.
\end{enumerate}
\end{definition}

\begin{proposition}\label{cellrepfinrad}
Let $\cC$ be an almost finitary bicategory and $\L$ a left cell with domain $\tk$ inside a two-sided cell $\J$ of $\cC$.
Suppose that $\L$ is connection-finite. Then there exists an integer $t>0$, such that the $t$-th power of the radical of each $\bfN_\L(\ti)$ is contained in $\bfI_\L(\ti)$.
\end{proposition}

\proof
Let $f \in \rad^s\Hom_{\bfN_\L(\ti)}(\rG_0,\rG_s)$ and write $f = f_1\circv \cdots\circv f_s$ with $f_i\in \rad\Hom_{\bfN_\L(\ti)}(\rG_{i-1},\rG_i)$, for indecomposable $\rG_0,\ldots \rG_s\in \L$. Assume that there exists a $1$-morphism $\rF\in \cC(\ti,\tj)$, which we can without loss of generality assume to be in $\J$, such that $\rF f$ is not in $\rad\Hom_{\bfN_\L(\tj)}(\rF\rG_0,\rF\rG_s)$ and has an identity component $\id_\rH$ for some indecomposable $1$-morphism $\rH$ in $\L$. Since $\rH$ appears as a direct summand of each $\rF\rG_i$, only finitely many non-isomorphic $1$-morphisms $\rK_1,\ldots,\rK_{m}$ can appear in the list $\rG_0,\dots, \rG_s$ by the assumption that $\L$ is connection-finite. For each such $1$-morphism $\rK_r$, $r=1\dots, m$, hom-finiteness guarantees that there is an integer $\ell_r$ such that $\rad^{\ell_r}\End_{\cC(\tk,\ti)}(\rK_r)=0$, hence $\rK_r$ can appear in the list $\rG_0,\dots, \rG_s$ at most $\ell_r -1$ times. Taking $\ell$ to be the maximum of the $\ell_r$, $s$ is thus bounded by $\ell m$, and for any $t>\ell m$, we have $\rad^t\bfN_\L(\ti)\subset \bfI(\ti)$.
\endproof

\begin{definition}\label{radicalfgdef}
We call an almost finitary category $\C$ {\it radical-f.g.} if its radical is locally finitely generated in the sense that  for any $X\in \C$, there exist indecomposable $W_1,\ldots , W_n$ and $g_i\colon W_i\to X \in \rad\Hom_\C(W_i,X)$ for $i = 1,\ldots, n$ , such that any  radical morphism with codomain $X$ factors over  $\bigoplus_{i=1}^{n} W_i \xrightarrow{(g_1,\ldots, g_n)} X$.
\end{definition}

\begin{lemma} \label{cellcosparse}Let $\cC$ be an almost finitary bicategory and $\L$ a left cell with domain $\tk$ inside a two-sided cell $\J$ of $\cC$.
Suppose that $\L$ is connection-finite and $\L^\oplus$ is radical-f.g.. 
Then, for any $\ti\in \cC$, $\rF\in \bfC_\L(\ti)$, the total morphism space $\prod_{\rG\in \L}\Hom_{\bfC_\L(\ti)}(\rG,\rF)$ is finite-dimensional. In particular, there are only finitely many $\rG\in \L$ with $\Hom_{\bfC_\L(\ti)}(\rG,\rF)\neq 0$, i.e. $\bfC_\L(\ti)$ is cosparse.
\end{lemma}

\proof 
The assumption that $\L^{\oplus}$ is radical-f.g.\ guarantees that for each $j\geq 0$, there exist only finitely many non-isomorphic indecomposable $\rG_i^{(j)}$ such that $\rad^j\Hom_{\bfC_\L(\ti)}(\rG_i^{(j)}, \rF)/\rad^{j+1}\Hom_{\bfC_\L(\ti)}(\rG_i^{(j)}, \rF)\neq 0$ and by hom-finiteness each such space is finite-dimensional. The statement now follows from the fact that a finite power of the radical of $\bfC_\L(\ti)$ is zero by Proposition \ref{cellrepfinrad}.
\endproof

\begin{corollary}\label{abfinlength}
Let $\L$ be a left cell with domain $\tk$ inside a two-sided cell $\J$ of $\cC$.
Suppose that $\L$ is connection-finite and $\L^\oplus$ is radical-f.g.. 
Then, for each $\ti\in \cC$, the category $\ov{\bfC_\L(\ti)}$ is an abelian finite-length category.
\end{corollary}

\proof
This follows directly from  Lemma \ref{cellcosparse} and Lemma \ref{whenabelian}.
\endproof

\subsection{Algebras in the coproduct completion}\label{projfuncsec}

Throughout this subsection, let $\cC$ be an almost quasi-fiab bicategory.

The following proposition is an almost finitary analog of \cite[Theorem 2]{KMMZ}.

\begin{proposition}\label{projfun}
Let $\L$ be a left cell with domain $\tk$ inside a regular two-sided cell $\J$ of $\cC$. Assume that each $\bfC_\L(\ti)$ is cosparse and that, moreover, each $\ov{\bfC_\L(\ti)}$ is Frobenius.
Then, for $\rF$ in $\J\cap \cC(\ti, \tj)$ and any $X\in \ov{\bfC_\L(\ti)}$, the object $\rF\,X$ is either zero or projective in $\ov{\bfC_\L(\tj)}$, i.e. $\rF\,X$ can be viewed as an object in $\bfC_\L(\tj)$.
\end{proposition}

\proof
As $\bfC_\L(\ti)$ is cosparse, Corollary~\ref{abfinlength} implies that each $\ov{\bfC_\L(\ti)}$ is an abelian finite-length category. 

Let $L$ be a simple in $\ov{\bfC_\L(\ti)}$ and note that all $\rF\in \J\cap \cC(\ti, \tj)$ with  $\rF\, L$ nonzero lie in the same left cell of $\J$, thanks to the appropriate generalisation of \cite[Lemma 12]{MM1} and the regularity of $\J$.

Now suppose that for some $\rF \in \J\cap \cC(\ti, \tj)$, the object $\rF\, L$ is nonzero and not projective in $\ov{\bfC_\L(\tj)}$. 
Then for any $\rG \in \J\cap \cC(\ti,\tk)$ such that $\rG\, L\neq 0$, the object $\rG\, L$ is also not projective, since $\rF$ is a direct summand of $\rG'\rG$ for some $\rG'$ by virtue of $\rF$ and $\rG$ being in the same left cell. Thus $\coprod_{\tj\in \cC}\{\rF L\vert \rF\in \J\cap\cC(\ti,\tj) \}^{\oplus_{\ov{\bfC_\L}(\tj)}}$ carries the structure of a proper almost finitary sub-birepresentation $\bfS$ of $\ov{\bfC_\L}$. This generates a proper ideal $\bfI$, which contains $\id_{\rF L}$. Suppose $\rH\to \rF \,L$ is the projective cover and $\rF\, L\to \rH'$ is the injective hull of $\rF\,L$. Then the natural map $\rH \to \rF\,L \to\H'$ is contained in the ideal $\bfI$. Since $\ov{\bfC_\L(\ti)}$ is assumed to be Frobenius, this is a morphism in $\bfC_\L(\ti)$ (without considering the abelianisation), so the intersection of $\bfI$ with $\bfC_\L$ is nontrivial - a contradiction to simplicity of $\bfC_\L$. 
\endproof

\begin{remark}
Note that the condition in Proposition \ref{projfun} of each $\bfC_\L(\ti)$ being cosparse  is satisfied by Proposition \ref{cellcosparse} provided that $\L$ is connection-finite and $\L^\oplus$ is radical-f.g..
\end{remark}

Now assume $\cC$ is $J$-simple for a regular two-sided cell $\J$. Set $\bfJ(\ti,\tj) =(\J\cap \cC(\ti,\tj))^{\oplus}$. Note that $\coprod_{(\ti,\tj)\in \cC\times\cC}\bfJ(\ti,\tj)$ carries the structure of a birepresentation of $\cC\boxtimes \cC^{\op}$, which we denote by $\bfJ$. Note further that this is the cell birepresentation of $\cC\boxtimes \cC^{\op}$ for the (left) cell $\J$ of $\cC\boxtimes \cC^{\op}$ and simple by construction. 

\begin{lemma}\label{splitepizero}
Suppose that each morphism category $\bfJ(\ti,\tj)$ is cosparse and each $\overline{\bfJ(\ti,\tj)}$ is a Frobenius category.
Let $$\bigoplus_{j=1}^r\rG_{1,j}\xrightarrow{(g_{i,j})_{i,j}}\bigoplus_{i=1}^s\rG_{0,i} \qquad \in \ov{\bfJ(\ti,\ti) }$$ for indecomposable $\rG_{0,i}, \rG_{1,j}$, and let $\rF\in \J\cap \cC(\ti, \tj)$. Then, for each $k\in \{1,\ldots, s\}$, the map $(\id_\rF\circh g_{k,j}\circh \id_{\rF^*})_{j} \colon \bigoplus_{j=1}^r\rF\rG_{1,j}\rF^*\to \rF\rG_{0,k}\rF^*$ is either zero or split epi. In particular, this means that 
\[
\bigoplus_{j=1}^r\rF\rG_{1,j}\rF^*\xrightarrow{(\id_\rF\circh g_{i,j}\circh \id_{\rF^*})_{i,j}}\bigoplus_{i=1}^s\rF\rG_{0,i}\rF^*
\]
is a projective object in $\ov{\bfJ(\tj, \tj)}$, hence isomorphic to an object in $\bfJ(\tj,\tj)$.
 \end{lemma}
 
 \proof
The functor $\rG\mapsto \rF\rG\rF^* \colon \bfJ(\ti,\ti)\to \bfJ(\tj,\tj)$ is precisely the action of $(\rF,\rF)$ via the cell birepresentation of $\cC\boxtimes\cC^{\op}$ corresponding to $\J$. The assumptions guarantee that we can apply Proposition \ref{projfun}, so $$\bigoplus_{j=1}^r\rF\rG_{1,j}\rF^*\xrightarrow{(\id_\rF\circh g_{i,j}\circh \id_{\rF^*})_{i,j}}\bigoplus_{i=1}^s\rF\rG_{0,i}\rF^*$$ is a projective object in $ \ov{\bfJ(\tj,\tj)}$. The statement of the corollary follows.
 \endproof

\begin{proposition}\label{complsplitepi}
Suppose that each morphism category $\bfJ(\ti,\tj)$ is cosparse and  each $\overline{\bfJ(\ti,\tj)}$ is a Frobenius category.
Let $$\rG = \left(\coprod_{j\in J}\rG_{1,j}\xrightarrow{\theta}\coprod_{i\in I}\rG_{0,i}\right) \qquad \in \cpl\bfJ(\ti,\ti)$$ for indecomposable $\rG_{0,i}, \rG_{1,j}$ and $\rF\in \J\cap \cC(\ti, \tj)$.  Assume that, for each $\rG_{0,k}$, there are at most finitely many $\rG_{0,i}$ with $\rG_{0,i}\cong \rG_{0,k}$.

Then, for each $k\in I$, the map $$(\id_\rF\circh \pi_k \circh \id_{\rF^*})\circv(\id_\rF\circh \theta \circh \id_{\rF^*}) \colon \coprod_{j\in J}\rF\rG_{1,j}\rF^*\to \rF\rG_{0,k}\rF^*,$$
 where $\pi_k\colon \coprod_{i\in I}\rG_{0,i} \to \rG_{0,k} $ is the natural projection, is either zero or split epi, and thus $\rF\rG\rF^* \in \copr\bfJ(\ti,\ti)$.
\end{proposition}

\proof

Fix $k\in I$ and let $I_k = \{i_1,\ldots, i_s\}\subset I$ be the set of precisely those indices with $\rG_{0,i}\cong \rG_{0,k}$ for $i\in I_k$.

Consider the quotient $$\coprod_{j\in J}\rG_{1,j}\xrightarrow{\theta}\bigoplus_{i\in I_k}\rG_{0,i}$$ of $\rG$. Observe that $\bigoplus_{i\in I_k}\rG_{0,i}$ can be viewed as a projective object in $\ov{\bfJ(\ti,\ti)}$. The latter is abelian and has finite length due to Lemma \ref{whenabelian}, hence there is a finite subset $J_k\subset J$ such that the cokernel of $$\coprod_{j\in J}\rG_{1,j}\xrightarrow{\theta}\bigoplus_{i\in I_k}\rG_{0,i}$$ in $\cpl\bfJ(\ti,\ti)$ is isomorphic to the cokernel of $$\bigoplus_{j\in J_k}\rG_{1,j}\xrightarrow{\theta}\bigoplus_{i\in I_k}\rG_{0,i},$$ 
which can be viewed as an object in $\ov{\bfJ(\ti,\ti)}$. The claim now follows from Lemma \ref{splitepizero}.
\endproof

\begin{remark}
Note that the assumption that each $\bfJ(\ti,\tj)$ be cosparse in Lemma \ref{splitepizero} and Proposition \ref{complsplitepi} is satisfied provided that $\J$ is  connection-finite as a left cell of $\cC\boxtimes \cC^{op}$ and $\J^\oplus$ is radical-f.g., due to Lemma \ref{cellcosparse}.
\end{remark}

We would like to apply Proposition \ref{complsplitepi} to deduce that we can choose our algebra $1$-morphisms in the coproduct completion rather than the full completion. For this, we need the following lemma.

\begin{corollary}\label{algshape} Let $\cC$ be a $\J$-simple almost quasi-fiab bicategory and $\bfM$ a transitive birepresentation with apex $\J$ and generator $X\in \bfM(\ti)$.
Assume that $\J$ is adjoint connection-finite. Then $[X,X]$ is Morita equivalent to a cokernel of 
$$ \coprod_{j\in J} \rG_{1,j} \xrightarrow{\theta} \coprod_{i\in I} \rG_{0,i} $$
where  each $\rG_{0,i}$ and $\rG_{1,j}$ is an indecomposable in $\J$. Moreover, for each $\rG_{0,i}$, there are at most finitely many $\rG_{0,k}$ with $\rG_{0,i}\cong \rG_{0,k}$.
\end{corollary}

\proof
By Section \ref{inthom}, $[X,X]$ can be written as a cokernel of $\rZ \xrightarrow{\zeta} \rT_X$
for some $Z\in \copr\cC(\ti,\ti)$.  Let $\rF\in \J$ and consider the diagram $\rF \rZ \rF^* \xrightarrow{\id_\rF\circh\zeta\circh\id_{\rF^*}}\rF \rT_X\rF^*$ whose components are all in $\J$. Moreover, since $\cC(\ti,\ti)$ is hom-finite, each isomorphism class of indecomposable $1$-morphisms occurs at most finitely many times in $\rT_X$ and hence byadjoint connection-finiteness of $\J$  only finitely many times in $\rF \rT_X\rF^*$. 

Thus $[\rF X,\rF X]$ has the desired form and, by Lemma \ref{framing}, it is Morita equivalent to $[X,X]$.  \endproof

 \begin{remark}
 Note that if $\J$ is connection-finite as a left cell of $\cC\boxtimes \cC^{op}$, then it is in particular adjoint connection-finite as a $J$-cell for $\cC$, since the latter is a special case for the former, where instead of the action of a $1$-morphism $(\rF,\rG)$ in $\cC\boxtimes \cC^{op}$, we only restrict our attention to $1$-morphisms of the form $(\rF,\rF)$.
 \end{remark}

We thus arrive at the main theorem of this section.

\begin{theorem}\label{algcoprod}
Let $\cC$ be a $\J$-simple almost quasi-fiab bicategory for a regular cell $\J$. Let $\bfJ$ be the cell birepresentation of $\cC\boxtimes \cC^{\op}$ corresponding to $\J$.
Assume one of the following:
\begin{enumerate}
\item $\J$ is adjoint connection-finite and each $\bfJ(\ti,\tj)$ is cosparse;
\item $\J$ is connection-finite as a left cell of $\cC\boxtimes \cC^{op}$ and $\J^{\oplus}$ is radical f.g..
\end{enumerate}
Moreover, assume that  each $\ov{\bfJ(\ti,\tj)}$ is Frobenius.

Let $\bfM$ be a simple almost finitary birepresentation of $\cC$ with apex $\J$.
Then there exists an algebra $1$-morphism $\rA\in \copr\J$ (i.e. $\rA = \coprod_{i\in I} \rA_i$ with each $\rA_i\in \J$) such that $\bfM$ is equivalent to $\bfM_\rA$.
\end{theorem}

\proof
Either choice of assumptions guarantees we can apply  Proposition \ref{complsplitepi} and Lemma \ref{algshape}.
Choosing $X\in\bfM(\ti)$ for some $\ti\in \cC$, we have $\bfM\cong \bfM_{[X,X]}$ by Section \ref{intsec}. By Corollary \ref{algshape}, possibly replacing $X$ by some $\rH X$, we may assume that $[X,X]$ is given as the cokernel of 
$$ \coprod_{j\in J} \rG_{1,j} \xrightarrow{\theta} \coprod_{i\in I} \rG_{0,i}, $$
where  each $\rG_{0,i}$ and $\rG_{1,j}$ is an indecomposable in $\J$ and, for each $\rG_{0,k}$, there are at most finitely many $\rG_{0,i}$ with $\rG_{0,i}\cong \rG_{0,k}$. 
By further framing, Proposition \ref{complsplitepi} guarantees that $[X,X]$ is Morita equivalent to the cokernel of some 
$$ \coprod_{j\in J} \rG_{1,j} \xrightarrow{\theta} \coprod_{i\in I} \rG_{0,i}, $$
where the composition with the natural projection onto each $\rG_{0,k}$ is either zero or split epi.

Let $I_0\subset I$ be the set of those $k\in I$,  such that $\pi_k\circ\theta = 0$. Then the cokernel of 
$$ \coprod_{j\in J} \rG_{1,j} \xrightarrow{\theta} \coprod_{i\in I} \rG_{0,i} $$ is isomorphic to $\rA = \coprod_{i\in I_0} \rG_{0,i}$ and the statement follows.
\endproof

\begin{corollary}\label{algcopinH}
In the setup of the Theorem \ref{algcoprod}, the algebra $1$-morphism $\rA$ can be chosen in $\copr\H$ for a diagonal $H$-cell in $\J$.
\end{corollary}

\proof
The algebra  $1$-morphism in $\copr\H$ is obtained by framing the algebra  $1$-morphism in $\copr\J$ provided by Theorem \ref{algcoprod} by a $1$-morphism in $\H$.\endproof 

We thus obtain a stronger version of Proposition \ref{Hcellreduc}.

\begin{theorem}\label{copHcellreduc}
Let $\cC$ be a $\J$-simple almost fiab bicategory, which satisfies the conditions in Theorem \ref{algcoprod}, and $\H$ a diagonal $H$-cell in $\J$.
There are equivalences of bicategories
\[\cC \scopwfmod_\J\simeq\cC_\J \scopwfmod_\J  \simeq \cB^s_{\cC_\J, \copr{\J}} \simeq \cB^s_{\cC_\H, \copr{\H}} \simeq \cC_\H \scopwfmod_\H.\]

Moreover, this restricts to a biequivalence 
\[\cC\swfmod_\J \simeq \cC_\H\swfmod_\H.\]
\end{theorem}

\proof
The first line of equivalences follows directly from Proposition \ref{Hcellreduc} and Theorem \ref{algcoprod}. To see that this restricts to a biequivalence between the bicategories of simple almost finitary birepresentations, it suffices to notice that choosing (simple) algebra $1$-morphisms $\rA$ and $\rB$ in $\copr{\H}$, a morphism between the corresponding cop-af birepresentations $\cop{\bfM}_\rA$ and $\cop{\bfM}_\rB$ corresponds to a biprojective $\rA$-$\rB$-bimodules $\rM$ in $\copr{\H}$, which restrict to morphisms of almost finitary birepresentations  $\bfM_\rA$ and $\bfM_\rB$ (of $\cC$ of $\cC_\H$) if and only if $\rM$ is compact as a right $\rB$-module.
\endproof

\begin{remark}\label{checkonH}
In particular, it suffices to check the finiteness conditions in Theorem \ref{algcoprod} for one diagonal $H$-cell inside $\J$. Thus the theorem applies if $\J$ satisfies one of the following two conditions:
\begin{enumerate}
\item $\J$ contains a diagonal $H$-cell $\H$ (say with source $\ti$) which is adjoint connection-finite and where $\bfH(\ti,\ti)$ is cosparse;
\item $\J$ contains a diagonal $H$-cell $\H$ (say with source $\ti$) which is connection-finite as a left cell of $\cC_\H\boxtimes \cC_\H^{op}$ and such that $\H^{\oplus}$ is radical f.g..
\end{enumerate}
\end{remark}

\section{Double centraliser theorem and its applications}\label{mainabs}

Throughout this section, let $\cC$ be an almost fiab bicategory of the form $\cC = \cC_\H$ for an $H$-cell $\H$ with source and target $\ti$.

\subsection{Double centraliser theorem}\label{sec:dcthm}

Let $\rA\in \cC(\ti,\ti)$ be a simple algebra $1$-morphism in ${\H}^{\oplus}$. 

Consider the bicategory $\cB:=\cB_{\rA}$ defined as the $1$-full $2$-full subbicategory of $\cB_\cC$ with single object $\rA$. In other words, it has one object $\rA$ whose endomorphism category $\cB(\rA,\rA)$ is the category of biprojective $\rA$-$\rA$-bimodules in $\cC(\ti,\ti)$ with composition given by $-\circ_\rA -$. Note that all biprojective $\rA$-$\rA$-bimodules live in ${\H}^{\oplus}$.


Then the left action of $\cC$ on $\X= \{ \rF\rA\,\vert\, \rF\in \cC(\ti,\ti)\}^\oplus_{(\rmod_{\cpl\cC}\rA)_\ti}$ defines a natural structure of an almost finitary birepresentation of $\cC$. Moreover the action given by $- \circ_\rA \rM$ for $\rM\in \cB(\rA,\rA)$ defines a right birepresentation of $\cB$ on $\X$ and the two actions of $\cC$ and $\cB$ commute. Since $\H$ is an $H$-cell, the action of $\cB$ on $\X$ is also transitive with generator $\rA$.

The following lemma is the algebra analogue of the construction below the statement of \cite[Theorem 5.3]{MMMTZ}. 

\begin{lemma}
Let $\rX$ be a projective right $\rA$-module in $\cC(\ti,\ti)$. Then $\rX^*$ is a 
projective left $\rA$-module with action given by
\[\delta_{\rA,\rX^*}\colon \, \rA\rX^*\xrightarrow{\eta_{\rX}\circh\id_{\rA\rX^*}} \rX^*\rX\rA\rX^*\xrightarrow{\id_{\rX^*}\delta_{\rX,\rA}\circh\id_{\rX^*}}  \rX^*\rX\rX^*\xrightarrow{\id_{\rX^*}\circh\varepsilon_\rX}\rX^*.\]

Similarly, for a projective left $\rA$-module $\rY$ in $\cC(\ti,\ti)$, one obtains a projective right $A$-module-structure $\delta_{\rY^*,\rA}$ on $\rY^*$.
\end{lemma}

\begin{lemma}\label{lem:AstarAproj}
The right birepresentation of $\cB$ on $\X$ is equivalent to $(\rA^*\rA)\bproj_\cB$, the equivalence being given by $\rX\mapsto \rA^*\rX$.
\end{lemma}

Note that  $(\rA^*\rA)\proj_{\cB(\ti,\ti)} = \{\rA^*\rA\circ_\rA\rM \,\vert \rM\in \cB(\rA,\rA)\}^{\oplus}$ with the  action of $\cB$ given by right translation.

\proof
First note that, using the structures from the previous lemma for $\rX = \rA$, we see that $\rA^*\rA$ is indeed an algebra $1$-morphism in $\cB$ with multiplication given by
\[\rA^*\rA\circ_\rA \rA^*\rA \cong \rA^* \rA^*\rA \xrightarrow{\id_{\rA^*}\circh\bfu_\rA^*\circh \id_{\rA}} \rA^*\rA\] 
and unit given by 
\[\rA = \rA\one_\ti \xrightarrow{\id_\rA\circh\eta_\rA}\rA \rA^*\rA\xrightarrow{\delta_{\rA,\rA^*}\circh\id_\rA} \rA^*\rA. \]

As $\rA$ generates $\X$ as a right birepresentation of $\cB$, it now suffices to show that for any $\rM\in \cB(\rA,\rA)$
\[ \Hom_{\X} (\rA\circ_\rA\rM,\rA) \cong  \Hom_{\cB(\rA,\rA)} (\rM,\rA^*\rA)\]
By definition,
\[ \Hom_{\X} (\rA\circ_\rA\rM,\rA) \cong \Hom_{\X} (\rM,\rA) =\Hom_{\rproj_{\cC(\ti,\ti)} \rA }(\rM,\rA) \]
and hence
\[\Hom_{\rproj_{\cC(\ti,\ti)} \rA }(\rM,\rA)\cong\Hom_{(\rA )\rproj_{\cC(\ti,\ti)} (\rA) }(\rA\rM,\rA)  = \Hom_{\cB(\rA ,\rA) }(\rA\rM,\rA) \]
where the isomorphism is the free forgetful adjunction, see e.g. \cite[Section 1.6]{MMS}.
Finally, we claim that 
\[\Hom_{\cB(\rA ,\rA) }(\rA\rM,\rA) \cong \Hom_{\cB(\rA ,\rA) }(\rM,\rA^*\rA) \]
which follows from the definition of the left $\rA$-module structure on $\rA^*$.
This completes the proof of the lemma.
\endproof

Now let $\cB$ be the one-object bicategory with single object $\rB:=\rA^*\rA$ and $\cB(\rB,\rB)$ the category of biprojective $\rB$-$\rB$-bimodules in $\cB(\rA,\rA)$ with horizontal composition given by $-\circ_\rB -$. Denote by $\I$ its maximal two-sided cell, which consists of projective $\rB$-$\rB$-bimodules.

\begin{lemma}\label{bimoddct}
There is an equivalence of categories between $\H^\oplus$ and $\I^\oplus$ given by
\[\Gamma\colon\rF\mapsto \rA^*\rF\rA\] with inverse \[\Gamma^{-1}\colon\rY\mapsto \rA\circ_{\rB} \rY\circ_{\rB} \rA^*\]
\end{lemma}

\proof
First note that $\rA^*$ is naturally a left $\rB$-module, so $\rA^*\rF\rA$ is indeed a projective $\rB$-$\rB$-bimodule and the functors are well-defined. 
We clearly have $\Gamma\Gamma^{-1}(\rY) = \rA^*\rA\circ_{\rB} \rY\circ_{\rB} \rA^*\rA\cong \rY$ so it suffices to prove that $\Gamma^{-1}\Gamma(\rF)=\rA\circ_{\rB} \rA^*\rF\rA\circ_{\rB} \rA^*\cong \rF$ for any $\rF\in \H$, or equivalently that $\rA\circ_{\rB} \rA^*\rF\cong \rF$ and $\rF\rA\circ_{\rB} \rA^*\cong \rF$ for any $\rF\in \H$. Since $\H$ is both a left and a right cell for $\cC=\cC_{\H}$, we can test the first of these isomorphisms for $\rF$ of the form $\rA\rG$ and the second for $\rF$ of the form $\rG\rA^*$, where they become obvious.
\endproof

Recall that there is a canonical pseudofunctor \[\mathrm{can}\colon \cC\to \cE nd_{\cE nd^{\ex}_{\cC}(\X)}(\X).\]
Noting that $\cB$ is equivalent to $\cE nd^{\ex}_{\cC}(\X)$ and $\cB$ is equivalent to $\cE nd_{\cB}(\X)$, we obtain the double centraliser theorem.

\begin{proposition}\label{dct}
The pseudofunctor $\mathrm{can}$ induces an equivalence between $\H^\oplus$ and $ \cE nd^{\proj}_{\cE nd^{\ex}_{\cC}(\X)}(\X)$.
\end{proposition}

\subsection{Frobenius algebras}

\begin{lemma}\label{Frobalgprojinj}
Let $\rA$ be a Frobenius algebra $1$-morphism in ${\cC}(\ti,\ti)$. Then the projective and the injective objects in each $(\rmod_{\cC}\rA)_\tj$ coincide.
\end{lemma}

\proof
Observe that, by exactly the same arguments as in \cite[Theorem 3.3]{A99}, the categories $(\rmod_{\cC}\rA)_{\tj}$ of right $\rA$-modules in $\cC$ and $(\comod_{\cC}\rA)_{\tj}$ are isomorphic and have the same underlying objects in $\cC(\ti,\tj)$. Explicitly, assume $\rA$ has comultiplication $\delta$, counit $\epsilon$, multiplication $\mu$ and unit $\eta$. Then, given a comodule $\rX$ with comultiplication $\delta_\rX$, we can define an action $\alpha_\rX\colon \rX$ by 
$$\rX\rA \xrightarrow{\delta_\rX\circh\id_\rA} \rX\rA\rA \xrightarrow{\id_\rX\circh\mu}\rX\rA \xrightarrow{\id_\rX\circh \epsilon} \rX$$
and given a module $\rX$ with action $\alpha_\rX$, we can define a coaction by
$$\rX\xrightarrow{\id_\rX\circh\eta} \rX\rA \xrightarrow{\id_\rX\circh\delta} \rX\rA\rA \xrightarrow{\alpha_\rX\circh\id_\rA} \rX\rA$$
and these constructions are mutually inverse. Moreover, module maps with respect to these are also comodule maps, and vice versa. This implies, in particular, that the categories of injective right $\rA$-comodules and that of injective right $\rA$-modules coincide and, moreover, that the categories of projective right $\rA$-comodules and that of projective right $\rA$-modules coincide.

Now $\rA$ is both a projective right $\rA$-module and an injective right $\rA$-comodule, hence a projective-injective object in $(\rmod_{\cC}\rA)_{\ti}$.
The action of a $1$-morphism by left translation is exact due to the existence of left and right adjoints, and hence preserves projectives and injectives, so each $\rF\rA$ is a projective-injective right comodule. Moreover, $\{ \rF\rA\,\vert\, \rF\in \cC(\ti,\tj)\}^\oplus_{(\rmod_{\cC}\rA)_\tj}$ coincides with the category of projective objects in $(\rmod_{\cC}\rA)_\tj$ by the analoguous argument as in the proof of Lemma \ref{projectivesinmodcplC} and with the category of injective objects in $(\rmod_{\cC}\rA)_\tj$ by the dual argument (see e.g. the proof of \cite[Theorem 4.7]{MMMT}), which completes the proof of the lemma.
\endproof

 \begin{lemma}\label{projfunFrob}
 Suppose $\X$ and $\Y$ are cosparse almost finitary categories such that $\ov{\X},\ov{\Y}$ are Frobenius. Then any projective functor 
$\Phi\colon \X\to\Y$ (i.e.\ projective in the category of exact $\Bbbk$-linear functors from $\ov{\X}$ to $\ov{\Y}$) is isomorphic to a retract of a coproduct of functors of the form $Y\otimes_\Bbbk \Hom_\X(X, -)$ for some $X\in\X, Y\in \Y$. Moreover, its adjoint 
is again a projective functor.
 \end{lemma}

\proof
Note that $\ov{\X}$ and $\ov{\Y}$ are abelian finite length categories, which are equivalent to the categories $\F un_\Bbbk^{fp}(\X^{\op}, \Bbbk\lmod)$ resp.\ $\F un^{fp}_\Bbbk(\Y^{\op}, \Bbbk\lmod)$ of finitely presented $\Bbbk$-linear functors with values in finite-dimensional vector spaces, the equivalence being induced by $X\mapsto X^\vee$, and as usual we identify $X \in \X$ with both $(0\to X) \in \ov{\X}$ and $X^\vee\in \F un_\Bbbk^{fp}(\X^{\op}, \Bbbk\lmod)$.

Every right exact functor from $\ov{\X}$ to $\ov{\Y}$ is uniquely defined by its values on $\X$, thus by a functor from $\X$ to $\ov{\Y}$. Now
\[ \F un_\Bbbk(\X^{\op}, \F un_\Bbbk(\Y^{\op}, \Bbbk\lmod))\simeq \F un_\Bbbk(\Y\otimes_\Bbbk\X^{\op}, \Bbbk\lmod)\]
and the projective objects in the latter category are retracts of coproducts of the representable functors $\Hom_\Y(-, Y)\otimes_\Bbbk \Hom_\X(X, -)$. 

Thus projective functors from ${\X}$ to ${\Y}$ are retracts of coproducts of the form $Y\otimes_\Bbbk \Hom_\X(X, -)$ and the right adjoint is thus computed via
\begin{equation*}\begin{split}
\Hom_{{\Y}}(Y\otimes_\Bbbk \Hom_\X(X, -), -) &\cong \Hom_\Bbbk(\Hom_\X(X, -), \Hom_\Y(Y, -)) \\
& \cong \Hom_\Bbbk(\Hom_\X(X, -), \Bbbk)\otimes_\Bbbk \Hom_\Y(Y, -)
\end{split}\end{equation*}

Since $\ov{\X}$ is assumed to be Frobenius, $\Hom_\Bbbk(\Hom_\X(X, -), \Bbbk)$ is again representable, thus proving that the right adjoint of a projective functor is again projective. Note that since projectives and injectives in $\ov{\X}$ coincide, as $X$ runs through all objects of $\X$, so does the object representing $\Hom_\Bbbk(\Hom_\X(X, -), \Bbbk)$, and the left adjoint of a projective functor is also projective.
 \endproof

\subsection{Simple separable algebras}

In this subsection, let $\rA$ be a simple separable algebra $1$-morphism in ${\H}^{\oplus}$, i.e. we assume that the multiplication $\bfm\colon \rA\rA\to\rA$ has a splitting $\sigma$ as a map of $\rA$-$\rA$-bimodules. This in particular implies that $\rA$ is a projective $\rA$-$\rA$-bimodule, since it is a direct summand of the projective bimodule $\rA\rA$.

As in the previous subsection, let $\cB$ be the bicategory with single object $\rA$ with endomorphism category $\cB(\rA,\rA)$ the category of biprojective $\rA$-$\rA$-bimodules in $\cC(\ti,\ti)$.

\begin{lemma}\label{AoneH}
The bicategory $\cB$ has only one $H$-cell, which we call $\H_\cB$, which consists of projective $\rA$-$\rA$-bimodules.\end{lemma}

\proof
Let $\rM$ be a biprojective  $\rA$-$\rA$-bimodule. Then there exist $\rF, \rG$ in $\H$ such that $\rM$ is a direct summand of $\rF\rA$ and of $\rA\rG$. Since $\rA$ is a direct summand of $\rA\rA$, we see that $\rM$ is a direct summand of $\rM\rA$ and of $\rA\rM$ by restricting the splittings from $\rF\rA\rA$ and $\rA\rA\rG$. Thus $\rM$ is a direct summand of $\rA\rM\rA$ and hence every biprojective $\rA$-$\rA$-bimodule is actually projective.

Moreover, by the properties of $H$-cells, there exists some $\rH$ in $\H$, such that $\rA$ is a direct summand of $\rH\rM$ and hence of $\rH\rF\rA$ as a right $\rA$-module.
Hence $\rA\rA$ is a bimodule direct summand of $\rA\rH\rF\rA = \rA\rH\rA\circ_\rA\rF\rA$, and so is $\rA$. This implies that $\rA$ is in the same left cell as $\rM$ and one similarly checks that it is in the same right cell. 
Thus the identity of $\cB$ is in the same $H$-cell as any other $1$-morphism, so there is only one $H$-cell.
\endproof

\begin{lemma}\label{Afiab}
The bicategory $\cB$ is almost quasi-fiab. If $\rA$ is Frobenius, then $\cB$ is almost fiab.
\end{lemma}

\proof 
By the previous lemma, every biprojective $\rA$-$\rA$-bimodule is a direct summand (in $\cB(\rA,\rA)$) of some $\rA\rF\rA$, so $\cB$ inherits the property of being almost finitary from $\cC$. Moreover, we claim that the right adjoint of $\rA\rF\rA$ is give by $\rA^*\rF^*\rA$ and the left adjoint is given by $\rA\rF\rA^*$.
\endproof

\begin{lemma}\label{AHsimple}
The bicategory $\cB$ is $\H_\cB$-simple.
\end{lemma}

\proof
We already know that $\H_\cB$ is the unique $H$-cell in $\cB$, so it suffices to show that for any $2$-morphism in $\cB$, the biideal generated by it contains the identity on some $1$-morphism. Moreover, since every $1$-morphism in $\cB$ is a projective $\rA$-$\rA$-bimodule, it suffices to consider morphisms between free bimodules.
Let therefore $f\in \Hom_{\cB(\rA,\rA)}(\rA\rF\rA,\rA\rG\rA)$ for some $\rF,\rG$ in $\cC(\ti,\ti)$. Then, in particular, 
$f\in \Hom_{\proj_{\cC(\ti,\ti)}\rA}(\rA\rF\rA,\rA\rG\rA)$ and $\bproj_\cC\rA$ is a simple birepresentation. Thus we can find a $1$-morphism in $ \H$ which, using the fact that $\H$ is an $H$-cell, we can take to be of the form $\rA\rH$, such that 
\[\id_{\rA\rH}\circh f \,\,\in\,\, \Hom_{\proj_{\cC(\ti,\ti)}\rA}(\rA\rH\rA\rF\rA,\rA\rH\rA\rG\rA) \] contains an isomorphism as a direct summand. 
Now we claim that $\id_{\rA\rH}\circh f $ is indeed a morphism of $\rA$-$\rA$-bimodules and hence lives in $\cB(\rA,\rA)$.

Indeed, under the free-forgetful adjunction, $\id_{\rA\rH}\circh f $ corresponds to 
\[(\id_{\rA\rH}\circh f)\circv(\bfm_\rA\circh \id_{\rH\rA\rF\rA}) = \bfm_\rA\circh \id_{\rH}\circh f \,\,\in\,\, \Hom_{\proj_{\cC(\ti,\ti)}\rA}(\rA\rA\rH\rA\rF\rA,\rA\rH\rA\rG\rA) \]
However, since $\bfm_\rA$ has an $\rA$-$\rA$-bimodule splitting $\sigma$ by separability of $\rA$, we see that 
\[(\bfm_\rA\circv \sigma)\circh \id_{\rH}\circh f  = \id_{\rA\rH}\circh f\]
was already in $\cB(\rA,\rA)$.
Viewing  $\id_{\rA\rH}\circh f$ as $\id_{\rA\rH\rA}\circ_\rA f\in \Hom_{\cB(\rA,\rA)}(\rA\rH\rA\circ_\rA \rA\rF\rA,\rA\rH\rA\circ_\rA \rA\rG\rA)$, 
the statement of the lemma follows.
\endproof

\begin{lemma}\label{AhomfinitecosparseifC}
Assume that $\H^{\oplus}$ is cosparse. Then the same is true for $\cB$ and its cell birepresentation.
\end{lemma}

\proof
By the free-forgetful adjunction, 
\[\Hom_{\cB(\rA,\rA)}(\rA\rF\rA,\rA\rG\rA) = \Hom_{{\cC(\ti,\ti)}}(\rF,\rA\rG\rA)\]
so hom-finiteness and cosparseness of $\cB$ are inherited from $\H^{\oplus}$. Since $\cB$ only has one $H$-cell, morphism spaces in its cell birepresentation are quotients of morphism spaces in $\cB$, so the cell birepresentation is again cosparse.
\endproof

\begin{proposition}\label{Afusion}
Assume $\H^{\oplus}$ is cosparse and $\rA$ is Frobenius. Then
$\cB$ is an almost fusion category (meaning it satisfies the conditions of a fusion category, except for having only finitely many indecomposable objects).
\end{proposition}

\proof
The proof is similar to \cite[Proposition 3.1]{HM}, noting that, given that $\cB$ is a quasi-fiab with only one $H$-cell  $\H_\cB$ and it is $\H_\cB$-simple, it remains to show that it is semisimple.

The assumptions on $\cC$ allow us to apply Lemma \ref{AhomfinitecosparseifC}, which, using Lemma \ref{whenabelian},  implies that, for the birepresentation $\bfC_{\H_\cB}$ for $\cB$, the category $\overline{\bfC_{\H_\cB}(\rA)}$ is an abelian finite-length category.
Since $\rA$ is Frobenius, Lemma \ref{Frobalgprojinj} implies that $\overline{\bfC_{\H_\cB}(\rA)}$ is actually Frobenius, allowing us to apply Proposition \ref{projfun} to see that every $1$-morphism in $\H_\cB$, and thus in particular the identity $1$-morphism, sends any object in $\overline{\bfC_{\H_\cB}(\rA)}$ to a projective object. This in turn implies that $\overline{\bfC_{\H_\cB}(\rA)}$ is semisimple. Thus the kernel of the cell birepresentation $\bfC_{\H_\cB}$ is the radical of $\H_\cB$, and the radical is a $2$-ideal. Now $\H_\cB$-simplicity implies this radical is zero, and hence $\cB(\rA,\rA)$ is semisimple as required.
\endproof

\subsection{Biequivalence}

Assume $\cC=\cC_\H$ is $\H$-simple, $\H^{\oplus}$ is cosparse, and $\rA$ is a simple separable Frobenius algebra $1$-morphism in $\H^\oplus$. As before, we denote by $\cB$ the one object bicategory whose $1$-morphisms are projective $\rA$-$\rA$-bimodules under composition $-\circ_\rA -$. 
\begin{theorem}\label{thm:biequivalence}
There is a biequivalence \[\cC\swfmod_\H \simeq \cB\swfmod. \] 
\end{theorem}

\proof
To prove the theorem, we set up a Morita context in analogy to the proof of \cite[Proposition 7.8]{MMMTZ2}.

Consider the bicategory $\cD$ with two objects $\tx$, which we identify with $\X := \proj_{\cC(\ti,\ti)}\rA$,  and $\ty$, which we identify with the category $\Y:=\cB(\rA,\rA)$, i.e. the principal birepresentation of $\cB$. Note that $\cB$ acts on the right of both $\tx$ and $\ty$. Morphism categories in $\cD$ are defined by
\begin{equation*}\begin{split}
\cD(\tx,\tx)&= \Hom_{\cB}^{\proj,\one}(\X,\X)\\
\cD(\tx,\ty)&= \Hom_{\cB}^{\proj}(\X,\Y)\\
\cD(\ty,\tx)&= \Hom_{\cB}^{\proj}(\Y,\X)\\
\cD(\ty,\ty)&= \Hom_{\cB}(\Y,\Y)
\end{split}\end{equation*}
where the superscript $\one$ on the first morphism category refers to adding the identity functor.

Writing $\cD_{\tx}$ and $\cD_{\ty}$ for the endomorphism bicategories of the two objects, we observe that $\cD_{\ty} = \cB$ by construction and $\cD_{\tx}\simeq \cC$ by Proposition \ref{dct}.

Moreover, $\cD$ is fiab, since taking adjoints preserves projective functors by Lemma \ref{projfunFrob}. It has precisely two two-sided cells, namely one consisting of only the identity functor on $\X$ (identified with $\one_\ti\in \cC(\ti,\ti)$), and one consisting of all other $1$-morphisms. The latter two-sided cell has two diagonal $H$-cells, one (with source and target $\tx$) being identified with $\H$ under Proposition \ref{bimoddct} and the other (with source and target $\ty$)  consisting of all indecomposable $1$-morphisms in $\cB$.

The theorem now follows from Theorem \ref{copHcellreduc}.
\endproof


\section{Simple birepresentations of Soergel bimodules}\label{sec:simplebirepsforS}
In this section, we apply the results from the previous sections to study the almost finitary birepresentation theory of $\cS=\cS(W,S)$, 
the one-object wide fiab bicategory of Soergel bimodules of arbitrary finite rank Coxeter type $(W,S)$. As we already 
explained in the introduction, such a bicategory can be seen as a pivotal category, so we will switch to the terminology 
of monoidal categories and monoidal functors in this part of the paper, although we will continue to use the terms birepresentation and morphism of birepresentations. By convention, the word monoidal will always mean {\em weak monoidal}, 
unless explicitly stated otherwise. Recall that $1$-morphisms and $2$-morphisms in a one-object bicategory correspond 
to objects and morphisms in the associated monoidal category.

As already mentioned in the introduction, our reduction results only hold for simple birepresentations whose apex contains the longest element of a finite parabolic subgroup of $W$. The classification of the simple birepresentations of $\cS$ with such an apex $\J$ reduces to the classification of the simple birepresentations of Lusztig's asymptotic almost fusion category $\cA_\H$, for any choice of diagonal $H$-cell in $\J$, thanks to Theorem~\ref{thm:biequivalence}

\subsection{Soergel bimodules}\label{sec:Soergel}
\subsubsection{Grading conventions}
Soergel bimodules are $\bbZ$-graded, so we first have to fix our grading conventions. 
Given a ($\mathbb{Z}$-)graded vector space $M=\oplus_{i\in \mathbb{Z}} M_i$, we define the shift $M(1)$ by 
$M(1)_i:=M_{i+1}$ and use the shorthand  
\[
M^{\oplus m}:= M(i_1)^{\oplus m_1} \oplus \cdots \oplus M(i_k)^{\oplus m_k}, 
\]
for $m=m_1\vpar^{i_1}+\cdots + m_k\vpar^{i_k}\in \mathbb{N}[\vpar, \vpar^{-1}]$.  Note that $M(1)=M^{\oplus \vpar}$ 
according to this convention, which is opposite to the convention in \cite{MMMTZ2}, but coincides with the convention in \cite{EMTW}. 

In this paper, we will mostly avoid the use of graded categories, graded functors, etc., which were used in \cite{MMMTZ2}. 
In the setup of this paper, the homogeneous bimodule maps between two Soergel bimodules always span an infinite dimensional graded $\Bbbk$-vector space (for the experts: in this paper, 
Soergel bimodules are defined over the usual polynomial ring, which is infinite dimensional, instead of the coinvariant algebra, which is finite dimensional for finite 
Coxeter types). Fortunately, for all $t\in \mathbb{Z}$, the 
homogeneous bimodule maps of degree $t$ span a finite dimensional $\Bbbk$-vector space, so it is more natural to work with categories $\cC$ admitting shifts and, for two objects $\rX,\rY$ in $\cC$ and an integer $t\in \bbZ$, 
to consider the finite dimensional $\Bbbk$-vector space $\mathrm{Hom}_\cC(\rX,\rY(t))$ of degree preserving morphisms in 
$\cC$. For more information, see \cite[Sections 2.6 and 2.7]{MMMTZ2} and Remark \ref{rem:HOM}.

\subsubsection{Soergel bimodules}
From now on, we take $\Bbbk=\mathbb{C}$. 

Let $(W,S)$ be any Coxeter system of finite rank with Coxeter matrix $(m_{st})_{s,t\in S}$. An {\em expression} is a sequence $\underline{w}=(s_1,s_2,\ldots,s_r)$ of elements in $S$ (called {\em simple reflections}) and the corresponding element in $W$ is given by the product $w=s_1s_2\cdots s_r$. We say that an expression $\underline{w}$ of length $r$ is a {\em reduced expression} (rex) for $w$ if there is no shorter expression $\underline{v}$ such that $v=w$, in which case we define the {\em length} $\ell(w)$ of $w$ to be $r$, which is independent of the choice of rex for $w$. Finally, we denote the Bruhat order by $\preceq$, and recall that, for $v,w\in W$, we have $v\preceq w$ if and only if there is a rex for $w$ which 
contains a rex for $v$ as a subexpression.

Let $\mathfrak{h}$ be a {\em real Soergel realization} of 
$(W,S)$. For the complete definition of the latter, we refer to \cite[Sections 3.1 and 3.2 ]{EW2}, but we recall that $\mathfrak{h}$ 
is a finite dimensional real representation of $W$ with certain special properties. Let 
$\mathfrak{h}_\mathbb{C}:=\mathfrak{h}\otimes_{\mathbb{R}}\mathbb{C}$ and take $R$ to be the complex symmetric algebra on the 
dual vector space $\mathfrak{h}_\mathbb{C}^*$, with $\mathbb{Z}$-grading determined 
by putting $\mathfrak{h}_\mathbb{C}^*$ in degree two. The contragredient action of $W$ on $\mathfrak{h}_\mathbb{C}^*$ extends to an action 
on $R$ by degree-preserving algebra automorphisms. 

\begin{remark} In their seminal paper \cite{EW1}, Elias and Williamson use a particular realization of $(W,S)$ over $\mathbb{R}$, due to Soergel. However, we need a slightly more general setup below and we need to work over $\mathbb{C}$ to use the geometric results in \cite{Be} and \cite{BO} for (extended) affine Weyl groups. For a thorough discussion of the general notion of a 
realization (over a commutative ring $\Bbbk$) of a Coxeter group, see \cite[Sections 3.1 and 3.2 ]{EW2}.
\end{remark}

For any parabolic subset $I\subset S$, let $W_I$ be the 
corresponding parabolic subgroup of $W$ and $R^I\subset R$ the subalgebra of elements that are $W_I$-invariant. When $I=\{s\}$, for some $s\in S$, we simply write $R^s:=R^I$. For any $s \in S$, define the graded 
$R$-$R$-bimodule $\rB_s$ (from now on, $R$-bimodule for short) by  
\begin{equation}\label{eq:Bi}
\rB_s:=R\otimes_{R^s} R(1).
\end{equation} 
Note that $\rB_s$ is an indecomposable $R$-bimodule since it is generated by $1\otimes 1$ (which lives in degree $-1$, due to the shift). For any expression 
$\underline{w}=(s_1,\ldots, s_r)$, the corresponding {\em Bott-Samelson bimodule} $\rBS(\underline{w})$ (from now on, BS-bimodule, for short) is defined as 
\begin{equation}\label{eq:BS}
\rBS(\underline{w}):=\rB_{s_1}\otimes_R \cdots \otimes_R \rB_{s_r}.
\end{equation}
Note that $\underline{w}$ is not required to be a rex. To simplify notation, we will write these tensor products simply as products from now on. In general, the 
BS-bimodules are not indecomposable. For example, if $m_{st}=3$, then there exists a degree-preserving isomorphism of $R$-bimodules
\[
\rB_s\rB_t\rB_s\cong \rB_{sts}\oplus \rB_s,
\]
where $\rB_{sts}$ is the simplest example of an indecomposable Soergel bimodule that is not a BS-bimodule. As the notation suggests, we have $\rB_{sts}\cong \rB_{tst}$, so the isomorphism class of the indecomposable bimodule does not depend on the choice of rex. 
As a matter of fact, one can show that 
\[
\rB_{sts}\cong \rB_{tst}\cong R\otimes_{R^{\{s,t\}}} R(3),
\]
see \cite[Example 4.39 and Exercise 4.41]{EMTW}. 

Now, let $R$-gbim be the monoidal category of all graded $R$-bimodules which are finitely generated both as left and right $R$-modules, and degree-preserving bimodule maps. All BS-bimodules are free of finite rank both as a left and as a right graded $R$-module (with the rank being equal in both cases), see \cite[Lemma 4.27]{EMTW}, hence they form a full monoidal subcategory of 
$R$-gbim, which has finite dimensional morphism spaces, but is not closed under degree shifts or taking direct sums of direct summands.

\begin{definition}\label{def:Soeaff} The category $\cS=\cS(W,S)$ of {\em Soergel bimodules} is the smallest full 
subcategory of $R$-gbim containing all $R$-bimodules that are isomorphic to direct sums of direct summands of degree shifted BS-bimodules. This is closed under tensoring over $R$, hence it is a monoidal category.
\end{definition}

\begin{remark}\label{rem:HOM} Note that $\cS$ is not a graded monoidal category, because its morphism spaces, which are finite dimensional vector spaces, are not graded. In particular, for any $S\in S$, we have $\rB_s\cong \rB_s\{ t\}$ if and only if $t=0$. 

However, sometimes it will be useful to consider the graded monoidal category $\cS^*$, which has the same objects as $\cS$, but morphism spaces defined by 
\begin{equation}\label{eq:HOM}
\mathrm{HOM}_{\cS^*}(\rB_x,\rB_y):=\bigoplus_{t\in \mathbb{Z}}\mathrm{Hom}_{\cS}(\rB_x,\rB_y(t)),
\end{equation}
for $x,y\in W$. In this category, we have $\rB_s\cong \rB_s\{ t\}$ for all $s\in S$ and $t\in \mathbb{Z}$. Its morphism spaces 
are infinite dimensional graded vector spaces, but they are free of finite rank both as left and right $R$-modules (with the rank being the same in both cases again). 
\end{remark}

The underlying category of $\cS$ is Krull--Schmidt and Soergel proved that there is a set of indecomposable Soergel bimodules $\{\rB_w\mid w\in W\}$, such that any indecomposable object in $\cS$ is isomorphic to $\rB_w(t)$, for some $w\in W$ and $t\in \mathbb{Z}$. Moreover, for any $w\in W$ and any rex $\underline{w}$ for $w$, the decomposition of $\rBS(\underline{w})$ into indecomposable objects contains precisely one direct summand isomorphic to $\rB_w$ and all other direct summands are 
isomorphic to shifts of $\rB_u$ with $u\prec w$. Given $\rB_u, \rB_v\in \cS$, for some $u,v\in W$, the product $\rB_u\rB_v$ can be decomposed as 
\begin{equation}\label{eq:proddecomp}
\rB_u\rB_v\cong \bigoplus_{w\in W} \rB_w^{\oplus h_{u,v,w}}
\end{equation}
with $h_{u,v,w}\in \mathbb{N}[\vpar, \vpar^{-1}]$. Moreover, these $h_{u,v,w}$ are  invariant under the $\vpar$-antilinear involution of $\mathbb{Z}[\vpar,\vpar^{-1}]$ defined by swapping $\vpar$ and $\vpar^{-1}$. The {\em Soergel categorification theorem} says that that there is an isomorphism of $\mathbb{Z}[\vpar,\vpar^{-1}]$-algebras between 
the split Grothendieck group of $\cS=\cS(W,S)$ and the Hecke algebra $\rH=\rH(W,S)$ which maps $[\rB_w]$ to the Kazhdan-Lusztig basis element $b_w$, for $w\in W$. 

As already remarked, the morphism spaces in $\cS^*$ are free left (resp. right) 
$R$-modules of finite rank, and {\em Soergel's hom-formula} expresses this rank in terms of a $\vpar$-sesquilinear form on $H$, see e.g. \cite[Theorem 5.27]{EMTW}. We do not need to recall the whole formula here, but we will use a special case of it, which can be stated for morphisms in $\cS$ (i.e., for degree-preserving morphisms). It says that, for any $u,v\in W$, we have  
\begin{equation}\label{eq:Soehomform1}
\dim_{\mathbb{C}}\left(\mathrm{Hom}_{\cS}(\rB_u, \rB_v(t))\right) =
\begin{cases}
\delta_{u,v}, & \text{if}\; t=0;\\
0, & \text{if}\; t<0,
\end{cases}
\end{equation}
where $\delta_{u,v}$ is the Kronecker delta.

\subsection{The asymptotic bicategory for a diagonal $H$-cell}
\label{sec:asymptotic}
Lusztig defined a function $\mathtt{a}\colon W\to \mathbb{N}_0$, which is constant on two-sided KL cells and, therefore, also on 
left, right and diagonal $H$-cells. It thus makes sense to speak of the $\mathtt{a}$-value of a cell $\mathcal{Z}$ and write 
$\mathtt{a}(\mathcal{Z})$. Lusztig proved that, for any $x,y,z\in W$, we have  
\begin{equation}\label{eq:gamma}
h_{x,y,z}=\vpar^{-\mathtt{a}(z)}\gamma_{x,y,z^{-1}} \left(\bmod\; \vpar^{-\mathtt{a}(z)+1}\mathbb{N}_0[\vpar]
\right)
\end{equation} 
for some $\gamma_{x,y,z^{-1}}\in \mathbb{N}_0$. Recall that $h_{x,y,z}$ is invariant under 
switching $\vpar$ and $\vpar^{-1}$, so the statement in \eqref{eq:gamma} is equivalent 
to Lusztig's statement in \cite[13.6 (c)]{Lu}. The following lemma is simply a copy of \cite[Lemma 3.3]{MMMTZ2}, which holds for any Coxeter type (not just for 
finite ones, as was the assumption in that paper).
\begin{lemma}\label{lem:ashiftedS} 
Let $\J$ be a two-sided cell of $W$ and $a=\mathtt{a}(\J)$. For any $x\in \J$, define 
$\widetilde{\rC}_x:=\rB_x(a)\in \cS$. Then, in $\cS_{\J}$, for all $x,y\in \J$, there is an isomorphism   
\[
\widetilde{\rC}_x\widetilde{\rC}_y \cong \bigoplus_{z\in \J} \widetilde{\rC}_z^{\oplus \vpar^{a} h_{x,y,z}}.
\]
In particular, if $\widetilde{\rC}_z(k)$ is a direct summand of $\widetilde{\rC}_x\widetilde{\rC}_y$, 
then $k\geq 0$. 
\end{lemma}

\begin{remark} 
Note that the grading shift of $\rB_x$ in \cite[Lemma 3.2 and Section 3.2]{MMMTZ2} is opposite to the one 
in Lemma~\ref{lem:ashiftedS}. The reason is that we work with algebra objects in this paper, instead of coalgebra objects. 
For example, given a Duflo involution $d\in \J$, the object $\rB_d(a)$ in $\cC_\J$ is naturally an algebra object, 
whereas $\rB_d(-a)$ is naturally a coalgebra object. We will see this in more detail below.  
\end{remark}

Now, let $\H$ be a diagonal $H$-cell in $W$ and $a=\mathtt{a}(\H)$, and let $d\in \H$ be the so-called {\em Duflo involution} (a.k.a. distinguished involution). This element in $\H$ is uniquely determined by a certain property, which we will not recall, but which implies that there is a non-zero morphism $\eta_d\colon R\to \widetilde{\rC}_d$, unique up to a non-zero 
scalar, in $\cS_\H$. By a slight abuse of terminology, we will also refer to $\widetilde{\rC}_d$ as a Duflo involution. Before we discuss the analog of 
\cite[Section 3.2]{MMMTZ2}, we introduce some notation. For $s\in \bbZ$, let
\begin{gather}
\H^{\oplus}_s:=\left\{\rB_w(s)\mid w\in \H \right\}^\oplus, \\
 \H^{\oplus}_{\geq s}:=\left\{\rB_w(t)\mid w\in \H,\; t\geq s \right\}^\oplus, \quad 
\H^{\oplus}_{>s}:=\left\{\rB_w(t)\mid w\in \H,\; t>s \right\}^\oplus.
\end{gather}
By definition, we have $\rB_w\in \H^\oplus_0$ and $\widetilde{\rC}_w\in \H^{\oplus}_a$, for $w\in \H$. 
Note that $\H^\oplus_{\geq a}$ and $\H^\oplus_{> a}$ are not monoidal categories, because they do not contain an identity object. 
By Lemma~\ref{lem:ashiftedS}, both inherit the monoidal product from $\cS_\H$, which is associative up to the usual canonical isomorphisms (even the usual tensor product of vector spaces is not strictly associative, but only up to canonical isomorphisms). 

The oplax identity object in $\H^\oplus_{\geq a}$ is given by $\widetilde{\rC}_d$ with the left and right oplax unitors $\ell_{\rX}$ and $r_{\rX}$, for $\rX\in \H^\oplus_{\geq a}$, 
being induced by the aforementioned morphism $\eta_d\colon R\to \widetilde{\rC}_d$ via 
\begin{equation}\label{eq:unitors}
\ell_{\rX}\colon \rX\cong R\rX\xrightarrow{\eta_d \circ_{\mathrm{h}} \mathrm{id}_{\rX}} \widetilde{\rC}_d\rX\quad \text{and}\quad 
r_{\rX}\colon \rX\cong \rX R\xrightarrow{\mathrm{id}_{\rX}\circ_{\mathrm{h}} \eta_d} \rX\widetilde{\rC}_d.
\end{equation}
Just as in \cite[(3.4)]{MMMTZ2}, left and right unitality of $R$ in $\cS_\H$ implies left and right oplax unitality of $\widetilde{\rC}_d$ in $\H^\oplus_{\geq a}$. For any 
$\rX\in \H^\oplus_{\geq a}$, there is an isomorphism of the form 
\begin{equation}\label{eq:oplaxidentity}
\widetilde{\rC}_d \rX \cong \rX\widetilde{\rC}_d\cong \rX \oplus \rX'
\end{equation}
for some $\rX'\in \H^\oplus_{>a}$, which is often non-zero, so $\widetilde{\rC}_d$ is indeed only an oplax identity object in general.  

As for $\H^\oplus_{> a}$, there is not even an oplax identity object, because that would have to be $\widetilde{\rC}_d$, but 
$\widetilde{\rC}_d$ does not belong to $\H^\oplus_{>a}$. Thus $\H^\oplus_{>a}$ is only a non-unital monoidal category. Thanks to Lemma~\ref{lem:ashiftedS}, it is a non-unital monoidal subcategory of $\H^\oplus_{\geq a}$, so we can define 
the quotient   
\begin{equation}\label{eq:AH}
\cA_{\H}:=\H^\oplus_{\geq a} / (\H^\oplus_{>a}),
\end{equation} 
where $(\H^\oplus_{>a})$ is the monoidal ideal generated by $\H^\oplus_{>a}$ in $\H^\oplus_{\geq a}$. 
We will denote the image of $\widetilde{\rC}_w$ in the $\cA_\H$ by $\rA_w$, for $w\in \cS$. The isomorphisms in~\eqref{eq:oplaxidentity} imply that $\rA_d$ 
is an honest identity object in $\cA_\H$, hence $\cA_\H$ is a monoidal category, which we call the {\em asymptotic category} 
for $\H$.  

Soergel's hom-formula in \eqref{eq:Soehomform1} implies that, for all $u,v\in \H$, we have    
\begin{equation}\label{eq:Soehomformassymp}
\dim_\mathbb{C}\left(\mathrm{Hom}_{\cA_\H}(\rA_u, \rA_v)\right) =\delta_{u,v},
\end{equation}
so $\cA_\H$ is semisimple.

\begin{remark}\label{rem:asympJ}
Although we will mostly consider asymptotic categories associated to diagonal $H$-cells, one can easily extend the above definition 
to two-sided cells. For every two-sided cell $\J$ with finitely many left cells in $W$, there is an asymptotic category $\cA_\J$, 
which is semisimple and has a monoidal product, with the identity object being given by the direct sum of all Duflo involutions in $\J$. Moreover, it is a pivotal category, with duality defined as above. 

This happens, of course, when $W$ is a finite Coxeter group, because then it has only finitely many cells of any type (two-sided, left, right and diagonal), but it can also happen when $W$ is infinite. For example, Lusztig proved in \cite[Theorem 2.2]{Lu2} that the same happens when $W$ is any affine Weyl group, although each cell contains infinitely elements in general, and Dyer proved in \cite[Lemmas 3.3 and 6.1]{Dy} that it also happens when $W$ is any finite rank universal Coxeter group.

If $W$ is a Coxeter group with a two-sided cell $\J$ containing infinitely many left cells, 
then $\cA_\J$ has infinitely many Duflo involutions (one for each left cell), so the identity object of the monoidal structure would be 
an infinite direct sum. One way to get around this technical problem is to consider $\cA_J$ as a bicategory whose objects are the left cells $\Gamma$ in $\J$ and whose Hom-categories are given by 
\[
\mathrm{Hom}_{\cA_\J}(\Gamma_1,\Gamma_2):=\cA_{\Gamma_1\cap \Gamma_2^*}
\]  
for any two left cells $\Gamma_1, \Gamma_2$ in $\J$. Horizontal and vertical composition are defined in the obvious way. This bicategory is almost fiab and semisimple, in other words, it is an almost fusion bicategory. 
\end{remark}

\subsection{Relation with Lusztig's definition}\label{sec:relLusztig} 
Our definition of $\cA_{\H}$ differs from Lusztig's definition in \cite[\S 18.15]{Lu}, because he defines the asymptotic category 
as $\H^\oplus_0\subset \cS_\H$ with a modified monoidal product (the {\em truncated monoidal product}).
To recall the latter monoidal product, we adapt (part of) the excellent discussion in \cite[Sections 5C and 5D]{ERT} to our setting. 
Let $s\in \mathbb{Z}$. Every $\rX\in \H^{\oplus}_{\geq s}$ admits a decomposition 
of the form 
\[
\rX\cong \rX_s\oplus \rX_{>s}
\] 
such that $\rX_s\in \H^\oplus_{s}$ and $\rX_{>s}\in \H^\oplus_{s}$. This decomposition is essentially unique, thanks to the Krull-Schmidt property, but it is not canonical. However, there is a canonical choice for $\rX_s$ and the projection $\rX\to \rX_s$ is independent of the choice of $\rX_{>s}$, because $\mathrm{Hom}_{\Soeext_n}(\rX_{>s},\rX_s)=0$ for all possible choices of 
$\rX_{>s}$ by \eqref{eq:Soehomform1}. This result can be obtained from \cite[Lemma 5C.2]{ERT} by applying the contravariant involution $\mathbb{D}\colon \cS\to \cS$, called {\em duality} in that paper. We call $\rX_s$ the {\em $(-s)$-degree quotient} of 
$\rX$ and note that it can be zero. In general, there is no canonical choice for $\rX_{>s}$ nor for the inclusion of $\rX_s$ into $\rX$. 

Given $\rX, \rY\in \H^\oplus_0$, the formulas in \eqref{eq:proddecomp} and \eqref{eq:gamma} imply that 
$\rX\rY\in \H^\oplus_{\geq -a}$, and Lusztig's truncated monoidal product 
$\rX\bullet\rY$ is defined as 
\[
((\rX\rY)_{-a})(a) \in \H^\oplus_0.
\]  
The identity object for this truncated monoidal product is given by $\rB_d$. The canonicity of the 
$(-a)$-degree quotients implies that in this way $\H^\oplus_0$ becomes a monoidal category, see~\cite[\S 18.15]{Lu}. 

By the above, there is an oplax monoidal functor $\Pi\colon \H^\oplus_{\geq a}\to \H^\oplus_0$ which is full and essentially 
surjective, sending every object 
$\rX\in \H^\oplus_{\geq a}$ to $(\rX_a)(-a)\in \H^\oplus_0$ and every morphism $f\colon \rX\to \rY$ in $\H^\oplus_{\geq a}$ to the induced morphism 
$f'\colon (\rX_a)(-a)\to (\rY_a)(-a)$ in $\H^\oplus_0$, which is well-defined by \eqref{eq:Soehomform1}. 
Moreover, there is a natural $2$-isomorphism 
\[
\Pi_{\rX,\rY}\colon \Pi(\rX\rY)=(\rX\rY)_a(-a)\to \rX_a(-a)\bullet \rY_a(-a)=\Pi(\rX)\bullet\Pi(\rY), 
\] 
for $\rX,\rY\in \H^\oplus_{\geq a}$. In particular, for $\rX\in \H^\oplus_{\geq a}$, we have 
\begin{gather*}
(\widetilde{\rC}_d \rX)_a(-a)\cong \rB_d\bullet \rX_a(-a)\cong \rX_a(-a);\\
(\rX \widetilde{\rC}_d)_a(-a)\cong \rX_a(-a) \bullet \rB_d\cong \rX_a(-a),
\end{gather*} 
so $\Pi_{\widetilde{\rC}_d, \rX}$ and $\Pi_{\rX,\widetilde{\rC}_d}$ define the unital isomorphisms for the 
image of the oplax identity object. By definition, the induced morphism $f'$ is zero if and only if $f'$ factors through $\H^\oplus_{>a}$, so $\Pi$ descends to $\cA_\H$ and the result is an equivalence $\cA_\H\simeq \H^\oplus_0$ of monoidal categories. 
Henceforth, we will tacitly identify $\cA_\H$ and $\H^\oplus_0$, but with a caveat. 

\begin{remark}\label{eq:strictification} By definition, the monoidal category $\cA_\H$ is essentially strict, because all coherers are canonical, but $\H^\oplus_0$ can have non-canonical coherers, see e.g. \cite[Section 8]{MMMTZ2}. This shows that the 
isomorphisms $\Pi_{\rX,\rY}$ can be non-trivial in general and that our monoidal category $\cA_\H$ is 
a strictification of the monoidal category $\H^\oplus_0$. 
\end{remark}

Elias and Williamson~\cite[Proposition 5.5]{EW3} showed that $\H^\oplus_0$ is a pivotal category. 
The dual of $\rB_w$ is $\rB_{w^{-1}}$, just as in $\cS$, but the coevalution and evaluation morphisms are different. Using Hodge theoretic arguments, Elias and Williamson defined these morphisms and showed that they satisfy the usual coherence relations. 
Via the monoidal equivalence above, we can transport this pivotal structure to $\cA_\H$. 

The oplax monoidal functor $\Pi\colon \H^\oplus_{\geq a}\to \H^\oplus_0$ has a lax left adjoint 
$\Theta\colon \H^\oplus_0\to \H^\oplus_{\geq a}$, which sends every object $\rX\in \H^\oplus_0$ to 
$\rX(a)\in \H^\oplus_{\geq a}$ and every morphism 
$f\colon \rX\to \rY$ in $\H^\oplus_0$ to the induced morphism $f\colon  \rX(a)\to \rY(a)$ 
in $\H^\oplus_{\geq a}$. By definition, $\Theta$ is a lax monoidal functor 
because there is a canonical projection $\Theta_{\rX,\rY}\colon (\rX\rY)(2a) \to (\rX\bullet \rY)(a)$, for all $\rX,\rY\in \H^\oplus_0$. 
One immediately sees that $\Theta$ is left adjoint to $\Pi$, because, for all $t\geq a$, we have 
\begin{eqnarray*}
\mathrm{Hom}_{\H^\oplus_{\geq a}}(\Theta(\rB_x), \rB_y(t))&= &\mathrm{Hom}_{\H^\oplus_{\geq a}}(\rB_x(a), 
\rB_y(t)) \\
&\cong& \mathrm{Hom}_{\H^\oplus_0}(\rB_x, \rB_y(t-a))\\
& \cong & \mathrm{Hom}_{\H^\oplus_0}(\rB_x, \Pi(\rB_y(t))),
\end{eqnarray*} 
which is natural in $x,y\in \H$. Note that these morphism spaces are all zero for $t< a$ and one-dimensional for $t=a$, thanks to \eqref{eq:Soehomform1}, and that $\Pi\circ \Theta\cong \mathrm{Id}_{\H^\oplus_0}$ holds. Under the identification of 
$\cA_\H$ with $\H^\oplus_0$, we obtain an oplax monoidal functor $\Pi\colon \H^\oplus_{\geq a}\to \cA_\H$ and its left adjoint 
$\Theta\colon \cA_\H\to \H^\oplus_{\geq a}$, which is a lax monoidal functor. These are the analogs of $\Pi$ and $\Theta$ 
in \cite[Sections 3.2 and 3.3]{MMMTZ2}, but there $\Pi$ is lax and $\Theta$ oplax, because in that paper the authors used opposite shifts and inclusions.   

\subsection{Duflo involutions}
Let $W$ be a Coxeter group of finite rank, $\H$ a diagonal $H$-cell in $W$ and $d$ the (unique) Duflo involution in $\H$. 

First, suppose that $d=w_I$, the longest element of some finite parabolic subgroup $W_I$ of $W$. Then $\mathtt{a}(\H)=\ell(w_I)$, which we denote by $a$ for short. 

\begin{lemma}\label{lem:hvsgamma}
For any $x,y,z\in \H$, we have 
\begin{equation}\label{eq:gammacan}
h_{x,y,z}=\pi(I)\gamma_{x,y,z^{-1}},
\end{equation}
where $\pi(I):=\vpar^{\ell(w_I)}\sum_{w\in W_I}\vpar^{-2\ell(w)}$.
\end{lemma}
\begin{proof} Let $x,y,z\in \H$. Then \cite[(2.2.4)]{Wi} implies that $h_{x,y,z}=\pi(I)\lambda$, for some $\lambda\in \mathbb{N}_0[\vpar,\vpar^{-1}]$.  We also know that, for $x,y,z\in \H$, the minimal and the maximal degree of $h_{x,y,z}$ belong to the interval $[-a,a]=
[-\ell(w_I),\ell(w_I)]$, which is also true for the minimal and the maximal degree of $\pi(I)$. This shows that $h_{x,y,z}=\pi(I) \lambda$, for some $\lambda\in \mathbb{N}_0$. By comparing the coefficients of the terms of minimal (or, alternatively, maximal) degree, we see that \eqref{eq:gamma} implies that $\lambda=\gamma_{x,y,z^{-1}}$.
\end{proof}

The algebra $R$ is a graded Frobenius extension of $R^I$ of graded rank $\tilde{\pi}(I)=\sum_{w\in W_I}\vpar^{-2\ell(w)}$, 
with non-degenerate trace $\partial_I\colon R\to R^I(2a)$, where $\partial_I:=\partial_{w_I}$ is the Demazure operator associated to $w_I$. Choose any pair of dual bases $\{f_1,\ldots, f_r\}$ and $\{f^1,\ldots, f^r\}$ of $R$ over 
$R^I$, such that $\partial_I(f_i f^j)=\partial_I(f^j f_i)=\delta_{i,j}$, where $\delta_{i,j}$ is the Kronecker delta and $r=\vert W_I \vert$. By \cite[Proposition 7.4.3]{Wi}, we have
\begin{equation}\label{eq:canDuflo}
\widetilde{\rC}_d\cong R \otimes_{R^I} R(2a),
\end{equation}
as (indecomposable) objects in $\cS$. Moreover, \cite[Theorem 24.36]{EMTW} shows that $\widetilde{\rC}_d$ has naturally the structure of a graded Frobenius algebra (object) in $\cS$ with multiplication (morphism) $\mathbf{m}_d\colon \widetilde{\rC}_d\widetilde{\rC}_d\to 
\widetilde{\rC}_d$, unit $\mathbf{u}_d\colon R\to \widetilde{\rC}_d$, comultiplication $\mathbf{\delta}_d\colon \widetilde{\rC}_d\to 
\widetilde{\rC}_d\widetilde{\rC}_d(-2a)$ and counit $\mathbf{\epsilon}_d\colon \widetilde{\rC}_d(-2a)\to 
R$ being defined by  
\begin{eqnarray}\label{eq:algebrastructure}
\mathbf{m}_d(f\otimes g\otimes h):= f\partial_I(g)\otimes h=f\otimes \partial_I(g)h\;, & &\mathbf{u}_d\colon 1\mapsto \sum_{i=1}^r f_i\otimes f^i,\\ 
\label{eq:coalgebrastructure}
\mathbf{\delta}_d(f\otimes g):= f\otimes 1\otimes g\; , & & \quad \mathbf{\epsilon}_d\colon f\otimes g\mapsto fg,
\end{eqnarray} 
for $f,g,h\in R$.
Note that 
\begin{equation}\label{eq:decompCd1}
\begin{aligned}
\widetilde{\rC}_d \widetilde{\rC}_d&=R\otimes_{R^I} R \otimes_R R\otimes_{R^I} R(4a) \\
&\cong  R\otimes_{R^I} R \otimes_{R^I} R(4a)\\
& \cong  R\otimes_{R^I} R^{\oplus \vpar^{4a}\widetilde{\pi}(I)}\\
& \cong  \widetilde{\rC}_d\oplus \cdots \oplus \widetilde{\rC}_d(2a),
\end{aligned}
\end{equation}
and that the definition of $\mathbf{u}$ does not depend on the choice of dual bases. In the last line of \eqref{eq:decompCd1}, 
the multiplicity of $\widetilde{\rC}_d$ and $\widetilde{\rC}_d(2a)$ is exactly one, whereas the other direct summands 
are all of the form $ \widetilde{\rC}_d(b)$, for an even integer $b$ satisfying $0<  b < 2a$, and can have higher multiplicities. 

\begin{lemma}\label{lem:sep1}
The graded Frobenius algebra object $\widetilde{\rC}_d$ is separable. 
\end{lemma}
\begin{proof} Define the element $p_{\,\mathrm{top}}\in R_{2a}$ by
\[
p_{\,\mathrm{top}}:=\dfrac{1}{r}\sum_{i=1}^r f_i f^i 
\]
and the morphism $\sigma_d\colon \widetilde{\rC}_d\to  \widetilde{\rC}_d \widetilde{\rC}_d$ in $\cS$ by 
\[
\sigma_d(f\otimes g):=f\otimes p_{\,\mathrm{top}}\otimes g 
\]
for $f,g\in R$. Then $\sigma_d$ is a $\widetilde{\rC}_d\text{-}\widetilde{\rC}_d$-bimodule map, which defines a splitting 
of $\mathbf{m}_d$ because $\partial_I(p_{\,\mathrm{top}})=1$. 
\end{proof}

Of course, the image of $\widetilde{\rC}_d$ in $\cS_\H$ (for which we use the same notation) is also a separable Frobenius 
algebra object.

Klein~\cite[Conjecture 5.2.3]{Kl} and Elias and Hogancamp~\cite[Conjecture 4.41]{EH} conjectured that $\widetilde{\rC}_d$ is always a graded Frobenius algebra object in $\cS$, even when $d$ is not of the form $w_I$ for some finitary parabolic subset $I\subset S$. 
\begin{remark}
Below, we will formulate their conjecture for the image of $\widetilde{C}_d$ in $\cS_\H$ (using the same notation), rather than $\cS$. Strictly speaking, this version is slightly weaker than their original one, but it is slightly easier to formulate, it has already been proved for all finite Coxeter groups, and it is exactly what we need in this paper. 
\end{remark}
First note that, in $\cS_\H$, there is an isomorphism   
\begin{equation}\label{eq:decompCd2}
\widetilde{\rC}_d\widetilde{\rC}_d\cong \widetilde{\rC}_d\oplus \cdots \oplus  \widetilde{\rC}_d(2a),
\end{equation}
where $a=\mathtt{a}(\H)$ and all intermediate direct summands are of the form $\widetilde{\rC}_w(b)$ for some $w\in \H$ 
(in general, they do not need to be equal to $d$, contrary to the special case in \eqref{eq:decompCd1}) and some integer $b$ satisfying 
$0<b<2a$. This is a consequence of \eqref{eq:proddecomp}, \eqref{eq:gamma}, the invariance of $h_{x,y,z}$ under the $\mathbb{Z}$-linear involution $\vpar \leftrightarrow \vpar^{-1}$, the fact that $\cS_\H$ has only one non-trivial two-sided cell, namely $\H$, and the fact that $\gamma_{d,d,x}=\delta_{d,x}$, for $x\in \H$ (see \cite[Conjectures 14.2: P2, P7)]{Lu}). Translated to our setting, Klein, Elias and Hogancamp's (KEH) conjecture can now be formulated as follows:
 
\begin{conjecture}[Weak KEH Conjecture]\label{conj:EH}
Let $d$ be any Duflo involution in $\H$ with $\mathtt{a}(d)=a$. The $1$-morphism $\widetilde{\rC}_d$ is a graded Frobenius 
algebra object in $\cS_\H$ with product $\mathbf{m}_d$ and coproduct $\delta_d$ being defined by a compatible choice of projection $\widetilde{\rC}_d\widetilde{\rC}_d \to \widetilde{\rC}_d$ and inclusion $\widetilde{\rC}_d\to \widetilde{\rC}_d\widetilde{\rC}_d(-2a)$ in \eqref{eq:decompCd2}, respectively, and unit $\mathbf{u}_d$ and counit $\epsilon_d$ being defined by a compatible 
choice of morphisms $R\to \widetilde{\rC}_d$ and $\widetilde{\rC}_d(-2a)\to R$, respectively.
\end{conjecture}

Each of these four morphisms is unique up to a non-zero scalar, thanks to the general version of Soergel's hom-formula, 
which we did not recall. To get a Frobenius structure, one has to fix one of them and show that the others can be scaled in 
a compatible way (there is at most one such way). The conjecture says that this is always possible, which was indeed shown to be true for any Duflo involution $d$ in a finite Coxeter group $W$, see \cite[Section 4.4]{MMMTZ2}, but remains open for infinite Coxeter groups $W$ when $d$ is not the longest element of some finite parabolic subgroup of $W$. 

If the Weak KEH Conjecture turns out to be true in general, then it is not hard to show that $\widetilde{\rC}_d$ is also separable in general, which was actually included, but not proved, in the formulation of \cite[Conjecture 5.2.3]{Kl}. In this more general case, the role of $p_{\,\mathrm{top}}$ is 
played by $\rho^a\in R_{2a}$, where $\rho$ is the image in $\mathfrak{h}^*_\mathbb{C}$ of a suitably scaled {\em dominant regular element} in $\mathfrak{h}^*$, see e.g. \cite[Definition 5D.3]{ERT}. By an analog of \cite[Corollary 5D.8]{ERT},  the morphism $\sigma_d\colon \widetilde{\rC}_d\to \widetilde{\rC}_d\widetilde{\rC}_d$, defined by $\sigma_d:=M_\rho^a\circ \delta_d$, is a splitting of $\mathbf{m}_d$. Here $M_\rho^a \colon \widetilde{\rC}_d\widetilde{\rC}_d(-2a)
\to \widetilde{\rC}_d\widetilde{\rC}_d$ is the morphism given by 
\[
M_\rho(c\otimes c'):=c\rho^a \otimes c'=c\otimes \rho^a c',
\]
for $c,c'\in \widetilde{\rC}_d$, as in \cite[(5D.4)]{ERT}.  


\subsection{The monoidal category $\cB_\H$.} As in the previous subsection, let $\H$ be a diagonal $H$-cell in $W$ and $d\in \H$ the unique Duflo involution. Define 
\[
\H^{\oplus}_\mathbb{Z}=\left\{\rB_x(t)\mid x\in \H,\; t\in \mathbb{Z} \right\}^{\oplus}.
\]
and denote by $\cB_\H$ the $H$-simple quotient of the monoidal category of biprojective 
$\widetilde{\rC}_d\text{-}\widetilde{\rC}_d$-bimodules in $\H^{\oplus}_\mathbb{Z}\subset \cS_\H$, as in Section~\ref{sec:dcthm}. The following proposition is the analog of \cite[Proposition 6.11]{MMMTZ2} in the special case when $d=w_I$, the longest element 
of a finite parabolic subgroup $W_I$ of $W$. Unfortunately, in its proof we cannot invoke Proposition~\ref{Afusion} to conclude that $\B_\H$ is semisimple, because we have not yet proved that $\H^{\oplus}_{\mathbb{Z}}$ is cosparse. The latter fact cannot be proved directly because, while we show in Proposition~\ref{prop:connfin} that $\H$ is connection-finite and adjoint connection-finite, we do not a priori know that $\H^{\oplus}_{\mathbb{Z}}$ is radical-f.g. Both cosparseness of $\H$ and the radical-f.g. property of $\H^{\oplus}_{\mathbb{Z}}$ are only established in Corollary~\ref{cor:finitepropsH2} as a consequence of Proposition~\ref{prop:HomSH}. So far, we can only prove all these properties of $\H$ and $\H^{\oplus}_{\mathbb{Z}}$ 
when $d=w_I$, but we conjecture that they hold for any diagonal $H$-cell.

\begin{proposition}\label{prop:bimodbicat}
Suppose that $d=w_I$, where $w_I$ is the longest element in some finite parabolic subgroup $W_I\subset W$. 
There is an equivalence of almost fusion categories
\[
\bigoplus_{t\in \mathbb{Z}}\cA_\H(t)\simeq \cB_\H.
\]
\end{proposition}
\begin{proof} 

We claim that every object $\rA\in \H^\oplus_\bbZ$, where $a=\ell(w_I)$, has a canonical 
$\widetilde{\rC}_d\text{-}\widetilde{\rC}_d$-bimodule structure. Since $d=w_I$, there is, 
for any $\rA\in \H^\oplus_{\geq a}$, an isomorphism  
\begin{equation}\label{eq:AI}
\rA\cong R\otimes_{R^I} \rA^I\otimes_{R^I} R,
\end{equation}
for some $R^I\text{-}R^I$-bimodule $\rA^I$, see \cite[Proposition 7.4.3]{Wi}. In particular, we have $\widetilde{C}_d^I\cong R^I(2a)$, because 
\begin{equation}\label{eq:BdI}
R\otimes_{R^I} R^I \otimes_{R^I} R(2a)\cong R\otimes_{R^I} R(2a)\cong \widetilde{C}_d.
\end{equation}
The bimodules $\rA^I$ are called {\em singular Soergel bimodules}, which were introduced and studied in depth by Williamson in \cite{Wi}. By \eqref{eq:canDuflo}, \eqref{eq:algebrastructure} and \eqref{eq:coalgebrastructure}, we have 
\begin{equation}\label{eq:CdXCd}
\widetilde{\rC}_d\rA \widetilde{\rC}_d\cong R\otimes_{R^I} R\otimes_R R\otimes_{R^I} \rA^I\otimes_{R^I} R\otimes_R R\otimes_{R^I} R(4a) 
\end{equation}
and the biaction $ \widetilde{\rC}_d\rA \widetilde{\rC}_d\to \rA$ is given by
\begin{eqnarray*}
a\otimes b\otimes c\otimes p \otimes d\otimes e \otimes f \mapsto a\partial_I(bc)\otimes x\otimes \partial_I(de)f,
\end{eqnarray*}
for $a,b,c,d,e,f\in R$ and $p\in \rA^I$. This proves the claim. 

In particular, we see that, for every $x\in \H$ and $t\in \mathbb{Z}$, the object $\widetilde{\rC}_x(t)$ in $\cS_\H$ has a 
canonical $\widetilde{\rC}_d\text{-}\widetilde{\rC}_d$-bimodule structure. By \eqref{lem:ashiftedS} and \eqref{eq:gammacan}, the decompositions of $\widetilde{\rC}_d \widetilde{\rC}_x(t)$ and $\widetilde{\rC}_x(t)\widetilde{\rC}_d$ contain each exactly one indecomposable object isomorphic to $\widetilde{\rC}_x(t)$, which implies that the left and the right 
$\widetilde{\rC}_d$-action morphisms are both unique up to a non-zero scalar, thanks to Soergel's hom-formula in \eqref{eq:Soehomform1}. By the unitality condition for the algebra structure of $\widetilde{\rC}_d$, there is actually only one possible choice for these scalars and the action morphisms are unique. 

Let $\rA\in \H^{\oplus}_\bbZ$. Then $\rA$ can be decomposed into indecomposable objects of the form 
$\widetilde{\rC}_x(t)$, for certain $x\in \H$ and $t\in\mathbb{Z}$. By the same arguments as in the proof of \cite[Proposition 6.11]{MMMTZ2}, this decomposition is also the decomposition of $\rA$ into indecomposable $\widetilde{\rC}_d\text{-}\widetilde{\rC}_d$-bimodule objects.  Note that the dot diagram in that proof can be defined as in \cite[Remark 6.4]{MMMTZ2}, where the dumbbell corresponds to the endomorphism of $R$ given by multiplication with the element $p_{\mathrm{top}}$, defined in the proof of Lemma~\ref{lem:sep1}. This implies, in particular, that $\rA$ has a canonical bimodule structure.

Next, we are going to relate $\cB_\H$ to $\cA_\H$. Every $\rX\in \cA_\H$ is naturally an $\rA_d\text{-}\rA_d$-bimodule 
object, because $\rA_d$ is the identity object in $\cA_\H$. Since $\Theta$ is a lax monoidal functor, this implies that, for every 
$\rX\in \cA_H$, its image $\Theta(\rX)$ is naturally a $\widetilde{\rC}_d\text{-}\widetilde{\rC}_d$-bimodule object in 
$\H^\oplus_{\geq a}\subset \cS_\H$. From the above, we know that this bimodule structure on $\Theta(\rX)$ has to be the canonical one. In particular, this is true for the objects $\widetilde{\rC}_x=\Theta(\rA_x)$, with $x\in \H$. 
Similarly, the image $\Theta(f)\colon \Theta(\rX)\to \Theta(\rY)$ of a morphism $f\colon \rX\to \rY$ in $\cA_\H$ is naturally a 
morphism of $\widetilde{\rC}_d\text{-}\widetilde{\rC}_d$-bimodule objects in $\H^\oplus_{\geq a}\subset \cS_\H$. 

We claim that $\Theta$ induces a monoidal functor (with the same notation) 
\begin{equation}\label{eq:Thetabim}
\Theta\colon \cA_\H\to \cB_\H.
\end{equation}
To prove this claim, it only remains to show that, for any $\rX,\rY\in \cA_\H$, there is a natural isomorphism 
\begin{equation}\label{eq:ThetaXY}
\Theta(\rX\rY)\cong \Theta(\rX)\otimes_{\widetilde{\rC}_d} \Theta(\rY) 
\end{equation}
in $\cB_\H$. Let $\rA:=\Theta(\rX)$ and $\rB:=\Theta(\rY)$. To determine 
$\rA\otimes_{\widetilde{\rC}_d} \rB$, we have to compute the coequalizer of the morphisms $\rho_{\rA,d,\rB}: \rA\widetilde{\rC}_d\rB=(\rA\widetilde{\rC}_d)\rB\to \rA\rB$ 
and  $\lambda_{\rA,d,\rB}: \rA\widetilde{\rC}_d\rB=\rA(\widetilde{\rC}_d\rB)\to \rA\rB$, 
defined by the right $\widetilde{\rC}_d$-action on $\rA$ and the left $\widetilde{\rC}_d$-action on $\rB$, respectively. As above, let 
$A^I$ and $B^I$ be two $R^I\text{-}R^I$-bimodules such that 
\[
A\cong R\otimes_{R^I} A^I \otimes_{R^I} R\quad \text{and}\quad B\cong R\otimes_{R^I} B^I \otimes_{R^I} R.
\]
Using \eqref{eq:AI} and \eqref{eq:CdXCd} and the isomorphisms 
$R\otimes_R R\cong R$ and $R\otimes_{R^I} R^I\cong R\cong R^I \otimes_{R^I} R$, the $2$-morphisms  $\rho_{\rA,d,\rB}$ 
and $\lambda_{\rA,d,\rB}$ are given by the $R\text{-}R$-bimodule maps 
\begin{multline}
R\otimes_{R^I} \rA^I \otimes_{R^I} R\otimes_{R^I} R \otimes_{R^I} \rB^I \otimes_{R^I} R(2a)  \xrightarrow{\rho_{\rA,d,\rB}, \;\lambda_{\rA,d,\rB}} \\ 
\hfill R\otimes_{R^I} \rA^I \otimes_{R^I} R\otimes_{R^I} \rB^I \otimes_{R^I} R 
\end{multline}
\begin{eqnarray*}
\rho_{\rA,d,\rB}\colon a \otimes x \otimes b \otimes c \otimes y \otimes e &\mapsto&  
a\otimes x\partial_I(b) \otimes c \otimes y \otimes e  \\
\lambda_{\rA,d,\rB}\colon a \otimes x \otimes b \otimes c \otimes y \otimes e & \mapsto&  
a \otimes x \otimes b\otimes \partial_I(c)y \otimes e, 
\end{eqnarray*}
for $a,b,c,e\in R$, $x\in \rA^I$ and $y\in \rB^I$. The fact that $\partial_I(z)=0$ unless $z$ is a non-zero multiple of 
$p_{\mathrm{top}}$, implies that, in the coequalizer of 
$\rho_{\rA,\rB}$ and $\lambda_{\rA,\rB}$, only elements of the form $a\otimes x \otimes p_{\mathrm{top}} \otimes y \otimes e$ survive. The coequalizer is thus isomorphic to    
\begin{equation}\label{eq:coeqrholambda}
R\otimes_{R^I} \rA^I \otimes_{R^I} R^I\otimes_{R^I} \rB^I \otimes_{R^I} R(-2a)\cong  R\otimes_{R^I} \rA^I \otimes_{R^I} \rB^I \otimes_{R^I} R(-2a),
\end{equation}
which, by \cite[Theorems 2 and 3, and Definition 2.3.1]{Wi}, is isomorphic to the direct sum of indecomposable regular Soergel bimodules of the form $\widetilde{\rC}_x$, for $x\in \H$, where one has to take into account that we are working in $\cS_\H$. To see this more explicitly, suppose that $A^I=\widetilde{\rC}_x$ and $B^I=\widetilde{\rC}_y$, for some $x,y\in \H$. Then Lemma~\ref{lem:ashiftedS}, 
\eqref{eq:gammacan} and \cite[Theorems 2 and 3, and Definition 2.3.1]{Wi} imply that there is an isomorphism   
\[
\widetilde{\rC}_x^I\otimes_{R^I}\widetilde{\rC}_y^I\cong \bigoplus_{z\in \H} (\widetilde{\rC}_z^I)^{\oplus \vpar^{2a}\gamma_{x,y,z^{-1}}}.
\]
This proves the existence of the isomorphism in \eqref{eq:ThetaXY}, which implies that 
$\Theta$ induces a monoidal functor $\Theta\colon \cA_\H\to \cB_\H$. 

By Soergel's hom-formula in \eqref{eq:Soehomform1}, the full monoidal subcategory $\H^\oplus_a$ of $\cB_\H$ is semisimple 
and $\Theta$ is fully faithful and essentially surjective when corestricted to this monoidal subcategory. 
The pivotal structure of $\cA_\H$ is naturally transported to $\H^\oplus_a$, so that both are equivalent as almost fusion categories. 

To complete the proof of the proposition, it now suffices to show that 
\begin{equation}\label{eq:BimSchur}
\dim\left(\mathrm{Hom}_{\cB_\H}(\widetilde{\rC}_x,\widetilde{\rC}_y(t))\right)=\delta_{x,y}\delta_{t,0}
\end{equation} 
for $x,y\in \H$ and $t\in \mathbb{Z}$. By Soergel's hom-formula in \eqref{eq:Soehomform1}, we know that \eqref{eq:BimSchur} is true for $t\leq 0$. Now, let $\mathrm{rad}(\cB_\H)$ be the radical of the underlying category of $\cB_\H$. 
Since \eqref{eq:BimSchur} holds for $t\leq 0$ and we have 
\[
\widetilde{\rC}_x(r)\otimes_{\widetilde{\rC}_d} \widetilde{\rC}_y(t)\cong \bigoplus_{z\in \H} \widetilde{\rC}_z^{\oplus \gamma_{x,y,z^{-1}}}(r+t),
\]
for all $x,y\in H$ and $r,t\in \mathbb{Z}$, the radical $\mathrm{rad}(\cB_\H)$ is actually a monoidal ideal. This implies that it has to be zero, because $\cB_\H$ has no non-zero proper monoidal ideals, as $\cS_\H$ is $H$-simple by definition, 
see Section \ref{sec:Jsimple}. We, therefore, see that $\cB_\H$ is semisimple, which shows that \eqref{eq:BimSchur} 
holds for all $t\in \mathbb{Z}$. 
\end{proof}

If Conjecture \ref{conj:EH} is true for all Duflo involutions, then we conjecture 
that Proposition~\ref{prop:bimodbicat} also holds for all Duflo involutions, although the proof of the 
isomorphism in \eqref{eq:ThetaXY} might not be completely straightforward.


\subsection{Finiteness properties of $\cS_\H$}\label{sec:actfin}

\begin{proposition}\label{prop:connfin}
Let $(W,S)$ be any Coxeter system of finite rank. Suppose that $\H$ is a diagonal $H$-cell in $W$ containing an element of the form $d=w_I$, where $W_I$ is a finite parabolic subgroup of $W$. Then $\H$ is connection-finite and adjoint connection-finite. 
\end{proposition}

\begin{proof} We first show connection-finiteness. By Lemma~\ref{lem:ashiftedS}, we have to show that, given $x,z\in \H$, the set 
\begin{equation}\label{eq:connfin1}
\left\{ y\in \H\mid h_{x,y,z}\ne 0\right\} 
\end{equation}
is finite. 

Before we do that, note that the set 
\begin{equation}\label{eq:leftactfin0}
\left\{ y\in \H\mid \gamma_{x,y,z}\ne 0\right\}
\end{equation}
is finite, because $\gamma_{x,y,z}=\gamma_{z,x,y}$, see \cite[Proposition 13.9(b) and Conjectures 14.2 P7]{Lu} (Lusztig's conjectures are now known to be true for all Coxeter groups of finite rank, see \cite{CH} and references therein). 

By Lemma \ref{lem:hvsgamma}, we have $h_{x,y,z}=\pi(I)\gamma_{x,y,z^{-1}}$, for any $x,y,z\in \H$. The finiteness of the set in 
\eqref{eq:leftactfin0} implies, therefore, that $\left\{y\in \H\mid h_{x,y,z}\ne 0\right\}$ is finite, which completes the 
proof that $\H$ is connection-finite. 

Next, we show that $\H$ is adjoint connection-finite. Since $h_{y,x,z}=h_{x^{-1},y^{-1},z^{-1}}$ by \cite[13.1(e)]{Lu}, finiteness of the set in 
\eqref{eq:connfin1} for $x^{-1},z^{-1}\in \H$ implies finiteness of the set 
\begin{equation}\label{eq:connfin2}
\left\{y\in \H\mid h_{y,x,z}\ne 0\right\}.
\end{equation} 
To prove adjoint connection-finiteness of $\H$ we have to show that, for any $x,z\in \H$, the set 
\begin{equation}\label{eq:connfin3}
\left\{y\in \H\mid \sum_{w\in \H} h_{x,y,w} h_{w,x^{-1},z} \ne 0\right\}
\end{equation}
is finite. Given $x,z\in \H$, finiteness of the sets in \eqref{eq:connfin1} and \eqref{eq:connfin2} implies that there are only finitely many $w\in \H$ such that $h_{w,x^{-1},z}\ne 0$ and that, for every $w\in \H$, there are only finitely many $y\in \H$ such that $h_{x,y,w}\ne 0$. This shows that the set in \eqref{eq:connfin3} is finite.
\end{proof}

\begin{remark}
We conjecture that every two-sided cell $\J$ of $W$ is both connection-finite and adjoint connection-finite as a two-sided cell of $\cS$ (as a two-sided cell of $\cA_\J$, this is always true, see Remark \ref{rem:finAJ}).

We currently do not know how to prove this, but we can show that the proof reduces to the case when $z=d$, the unique 
Duflo involution in the right cell of $z$. For connection-finiteness, note that \cite[18.8(b)]{Lu} can be restricted to $\J$, giving  
\begin{equation}\label{eq:leftactfin1}
\sum_{w\in J} h_{x_1,x_2,w}\gamma_{w,x_3,y^{-1}}=\sum_{w\in J} h_{x_1,w,y}\gamma_{x_2,x_3,w^{-1}}
\end{equation}
for all $x_1,x_2,x_3,y\in \J$. Suppose that $y=d$ and $d\sim_L x_3$ (which is equivalent to $d\sim_R x_3^{-1}$). Then \eqref{eq:leftactfin1} becomes 
\begin{equation}\label{eq:leftactfin2}
\sum_{w\in J} h_{x_1,x_2,w}\gamma_{w,x_3,d}=\sum_{w\in J} h_{x_1,w,d}\gamma_{x_2,x_3,w^{-1}}.
\end{equation}
Since $\gamma_{w,x_3,d}=\delta_{w,x_3^{-1}}$, see  \cite[Conjectures 14.2 P2]{Lu} (which is now known to be true for 
all Coxeter groups of finite rank, see our comment above), this reduces to 
\begin{equation}\label{eq:leftactfin3}
h_{x_1,x_2,x_3^{-1}}=\sum_{w\in \J} h_{x_1,w,d}\gamma_{x_2,x_3,w^{-1}}.
\end{equation}
As in \eqref{eq:leftactfin0}, the set $\{x_2\in \J\mid \gamma_{x_2,x_3,w^{-1}}\}$ is finite for any choice of $x_3,w\in \J$. Therefore, 
the equation in \eqref{eq:leftactfin3} implies that, for any $x_1,x_3\in J$, the set $\{x_2\in J\mid h_{x_1,x_2,x_3^{-1}}\ne 0\}$ is finite if and only if the set $\{w\in J\mid h_{x_1,w,d}\ne 0\}$ is finite. 

As in the proof of Proposition~\ref{prop:connfin}, connection-finiteness of $\J$ implies adjoint connection-finiteness of $\J$.
\end{remark}

\begin{remark}\label{rem:finAJ}
Note that the arguments, which prove that the set in \eqref{eq:leftactfin0} is finite, also prove that 
\begin{equation}\label{eq:leftactfinJ0}
\left\{ y\in \J\mid \gamma_{x,y,z}\ne 0\right\}
\end{equation}
is finite. By arguments analogous to the ones in the proof of Proposition~\ref{prop:connfin}, 
this implies that $\J$ is adjoint connection-finite as a two-sided cell of $\cA_\J$.
\end{remark}

\begin{proposition}\label{prop:HomCH}
Let $W$ be any Coxeter group of finite rank. Suppose that $\H$ is a diagonal $H$-cell in $W$ containing an element of the form $d=w_I$, where $W_I$ is a finite parabolic subgroup of $W$, 
and let $\mathbf{C}_\H$ be the associated left $\cS_\H$-birepresentation. 
Then, for any $x,y\in \H$, we have 
\begin{equation}\label{eq:HomCH}
\mathrm{grdim}_{\mathbb{C}}\left(\mathrm{Hom}_{\mathbf{C}_\H}(\widetilde{\rC}_x,\widetilde{\rC}_y)\right)=
\begin{cases} 
\vpar^{\ell(w_I)}\pi(I), & \text{if}\; x=y;\\
0, &\text{if}\; x\ne y.
\end{cases}
\end{equation}
\end{proposition}
\begin{proof}
Recall that $\mathbf{C}_\H\cong \mathbf{proj}_{\cS_\H}(\widetilde{\rC}_d)$ and that Lemma~\ref{lem:AstarAproj} tells us that   
\begin{equation}\label{eq:HomCH1}
\mathbf{proj}_{\cS_\H}(\widetilde{\rC}_d)\cong (\widetilde{\rC}_d^*\widetilde{\rC}_d)\mathbf{proj}_{\cB_\H},
\end{equation}
where the equivalence is given by $\widetilde{\rC}_x\mapsto \widetilde{\rC}_d^*\widetilde{\rC}_x$, for $x\in \H$.

Write $a:=\ell(w_I)$ for short. As in \eqref{eq:decompCd1}, the isomorphism in \eqref{eq:AI} for $\rA=\widetilde{\rC}_x$ implies that
\begin{equation}\label{eq:HomCH2}
\widetilde{\rC}_d^* \widetilde{\rC}_x\cong \widetilde{\rC}_d(-2a) \widetilde{\rC}_x \cong \widetilde{\rC}_x^{\oplus \vpar^a\pi(I)},
\end{equation} 
for all $x\in \H$. Thanks to \eqref{eq:HomCH2} and \eqref{freeforget}, we have 
\begin{equation}\label{eq:HomSH3}
\begin{aligned}
\mathrm{Hom}_{\mathbf{proj}_{\cS_\H}(\widetilde{\rC}_d)}(\widetilde{\rC}_x,\widetilde{\rC}_y) & \cong 
\mathrm{Hom}_{(\widetilde{\rC}_d^*\widetilde{\rC}_d)\mathbf{proj}_{\cB_\H}}(\widetilde{\rC}_x^{\oplus \vpar^a\pi(I)},\widetilde{\rC}_y^{\oplus \vpar^a\pi(I)}) \\
&\cong  \mathrm{Hom}_{\cB_\H}(\widetilde{\rC}_x,\widetilde{\rC}_y^{\oplus \vpar^a\pi(I)}),
\end{aligned}
\end{equation}
for all $x,y\in \H$. By \eqref{eq:BimSchur}, the latter morphism space has graded dimension equal to $\vpar^a\pi(I)$.
\end{proof}

\begin{corollary}\label{cor:finitepropCH}
In the same setting as in Proposition~\ref{prop:HomCH}, the underlying category of $\overline{\mathbf{C}_\H}$ is finite-length abelian.
\end{corollary}
\begin{proof} The underlying category of $\mathbf{C}_\H$ is hom-finite by definition, and cosparse and radical-f.g. by Proposition~\ref{prop:HomCH}. By Lemma~\ref{whenabelian}, this implies that the underlying category of 
$\overline{\mathbf{C}_\H}$ is finite-length abelian.
\end{proof}

\begin{proposition}\label{prop:HomSH}
Let the setting be the same as in Proposition~\ref{prop:HomCH}. Then, for any $x,y\in \H$, we have 
\begin{equation}\label{eq:HomSH}
\mathrm{grdim}_{\mathbb{R}}\left(\mathrm{Hom}_{\cS_\H}(\widetilde{\rC}_x,\widetilde{\rC}_y)\right)= 
\begin{cases} 
d_x, & \text{if}\; x=y;\\
0, &\text{if}\; x\ne y
\end{cases}
\end{equation}
for some non-zero $d_x\in \mathbb{N}[v,v^{-1}]$.
\end{proposition}
\begin{proof}
By Lemma~\ref{bimoddct}, we have $\H^{\oplus}\cong (\widetilde{\rC}_d^*\widetilde{\rC}_d)\mathbf{biproj}_{\cB_\H}(\widetilde{\rC}_d^*\widetilde{\rC}_d)$, where the equivalence is given by $\widetilde{\rC}_x\mapsto \widetilde{\rC}_d^*\widetilde{\rC}_x\widetilde{\rC}_d$, for $x\in \H$.

By \eqref{eq:decompCd1}, we have 
\begin{equation}\label{eq:HomSH2}
\widetilde{\rC}_d^* \widetilde{\rC}_x \widetilde{\rC}_d\cong \widetilde{\rC}_x^{\oplus \pi(I)^2},
\end{equation} 
for all $x\in \H$. Thanks to \eqref{eq:HomSH2} and \eqref{freeforget}, we have 
\begin{equation}\label{eq:HomSH3}
\begin{aligned}
\mathrm{Hom}_{\cS_\H}(\widetilde{\rC}_x,\widetilde{\rC}_y) & \cong 
\mathrm{Hom}_{ (\widetilde{\rC}_d^*\widetilde{\rC}_d)\mathbf{biproj}_{\cB_\H}(\widetilde{\rC}_d^*\widetilde{\rC}_d)}(\widetilde{\rC}_x^{\oplus \pi(I)^2},\widetilde{\rC}_y^{\oplus \pi(I)^2}),
\end{aligned}
\end{equation}
for all $x,y\in \H$. By \eqref{eq:BimSchur}, the latter morphism space is zero if $x\ne y$ and has non-zero finite graded dimension if $x=y$.
\end{proof}

\begin{corollary}\label{cor:finitepropsH2}
In the same setting as in Proposition~\ref{prop:HomCH}, the cell $\H$ is cosparse and $\H^\oplus_\bbZ$ is radical-f.g.
\end{corollary}

The theorem below is the analog of \cite[Theorems 2.20 and 7.1]{MMMTZ2}. Denote by $\cA_\H\text{-}\mathrm{safmod}^\bbZ$ the trival $\bbZ$-cover of $\cA_\H\text{-}\mathrm{safmod}$, which is the $2$-category whose objects and $1$-morphisms are formal shifts of objects and $1$-morphisms in $\cA_\H\text{-}\mathrm{safmod}$, respectively, and 
\[
2\mathrm{Hom}_{\cA_\H\text{-}\mathrm{safmod}^\bbZ}(\rX(s),\rY(t)):=
\begin{cases}
2\mathrm{Hom}_{\cA_\H\text{-}\mathrm{safmod}}(\rX,\rY),&\text{if}\; s=t;\\
0, &\text{else},
\end{cases}
\]
for any $1$-morphisms $\rX,\rY\in \cA_\H\text{-}\mathrm{safmod}$ and any $s,t\in \bbZ$.
\begin{theorem}\label{thm:mainthmforS}
Let $(W,S)$ be any Coxeter system of finite rank. Suppose that $\J$ is a two-sided cell in $W$ and $\H$ is a diagonal $H$-cell in 
$\J$ containing a Duflo involution of the form $d=w_I$, where $W_I$ is a finite parabolic subgroup of $W$. Then there are biequivalences 
\begin{equation}\label{eq:mainthmforS}
\cS\text{-}\mathrm{safmod}_\J\simeq \cS_\H\text{-}\mathrm{safmod}_\H\simeq \cB_\H\text{-}\mathrm{safmod}\simeq 
\cA_\H\text{-}\mathrm{safmod}^\bbZ.
\end{equation}
\end{theorem}
\begin{proof}
The results in Proposition~\ref{prop:connfin} and Corollary~\ref{cor:finitepropsH2} imply that $\H$ satisfies the assumptions in  Theorems \ref{copHcellreduc} and \ref{thm:biequivalence}, proving the first two biequivalences. The proof of the 
last biequivalence is completely analogous to the proof of \cite[Proposition 7.5]{MMMTZ2}, the only difference being that 
we restrict to degree-preserving $2$-morphisms in this theorem. 
\end{proof}

\begin{remark}\label{rem:maindifference} Theorem \ref{thm:mainthmforS} is the analog 
of the combination of \cite[Theorem 2.20]{MMMTZ2} (strong $H$-reduction) and \cite[Theorem 7.1]{MMMTZ2}, but with one big technical restriction: it only holds for diagonal $H$-cells containing the longest element of a finite parabolic subgroup of $W$. In practice, unfortunately, most Coxeter groups have two-sided cells which do not contain such elements. As we have been indicating in the previous sections, we hope that this restriction can be lifted once Conjecture \ref{conj:EH} (the weak KEH-conjecture) has been proved.  
\end{remark}

Nevertheless, there are exceptions to the restriction described in Remark \ref{rem:maindifference}. 
In the next two sections, we consider two classes of infinite Coxeter groups (including an extended version of them) whose two-sided cells all contain the longest element of a finite parabolic subgroup. Moreover, for (extended) affine Weyl groups of type $A$ (the first class of examples), the asymptotic categories associated to the diagonal $H$-cells have been described explicitly in the literature. For universal Coxeter groups (the other class of examples), we conjecture that there is an explicit description as well. 

Before we continue, we prove that strong $H$-reduction holds in full generality for the asymptotic (bi)categories associated to two-sided cells, see Remark~\ref{rem:asympJ}. 

\begin{theorem}\label{thm:mainthmforJ}
Let $(W,S)$ be any Coxeter system of finite rank and $\J$ any two-sided cell in $W$. For any diagonal $H$-cell $\H$ in $\J$, 
there is a biequivalence 
\begin{equation}\label{eq:mainthmforJ}
\cA_\J\text{-}\mathrm{safmod}\simeq \cA_\H\text{-}\mathrm{safmod}.
\end{equation}
\end{theorem}
\begin{proof}
Note that semisimplicity implies that $\cA_\J$ is Frobenius, cosparse and radical-f.g.. Remark~\ref{rem:finAJ} shows that the cell $\J$, as a two-sided cell of $\cA_\J$, is adjoint connection-finite.  This proves that the category $\cA_\J$ satisfies the conditions in Theorem~\ref{copHcellreduc}. 
\end{proof}


\section{(Extended) affine type $A$ Weyl groups}\label{secafftypeA}

Let $G$ be a connected reductive algebraic group $G$ over $\mathbb{C}$ with finite Weyl group $W_f$, root and coroot 
lattice $Q$ and $Q^\vee$, respectively, and weight and coweight lattice $P$ and $P^\vee$, respectively. 
The associated {\em affine Weyl group} $W$ and {\em extended affine Weyl group} $W^{\mathrm{ext}}$ are defined as  
\begin{equation}\label{eq:affWeyl}
W:=W_f\ltimes Q^\vee\quad\text{and}\quad W^{\mathrm{ext}}:=W_f\ltimes P^\vee,
\end{equation}
respectively. The Langlands dual group $G^\vee$ is another connected reductive algebraic group over $\mathbb{C}$ with the 
same finite Weyl group, but root and coroot lattice $Q^\vee$ and $Q$, respectively, 
and weight and coweight lattice $P^\vee$ and $P$, respectively.

For $G=\mathrm{GL}(n,\mathbb{C})$, we have $G^\vee\cong G$ and $W$ is isomorphic to the affine symmetric group $\Saff_n$, which is a Coxeter group of type $\widehat{A}_{n-1}$, and $W^{\mathrm{ext}}$ is isomorphic to the extended affine symmetric group $\Saffext_n$, whose definition we will recall below. 

The coweight lattice of $G=\mathrm{SL}(n,\mathbb{C})$ is isomorphic to the weight lattice of 
$G^\vee\cong \mathrm{PGL}(n,\mathbb{C})$, which is the adjoint group of $\mathrm{GL}(n,\mathbb{C})$. The latter fact implies that 
the weight lattice of $\mathrm{PGL}(n,\mathbb{C})$ is isomorphic to the root lattice of $\mathrm{GL}(n,\mathbb{C})$, 
hence $\Saff_n$ is also isomorphic to the extended affine Weyl group of $\mathrm{SL}(n,\mathbb{C})$. 

\begin{remark}\label{rem:weightcoweight} We use the coweight and the coroot lattice of $G$ in \eqref{eq:affWeyl}, 
which is a common but not a universal convention. Some authors use the weight and root lattice instead, which corresponds to 
$G^\vee$ in our convention. This explains, for example, why Xi obtains $\Saff_n$ as the extended affine Weyl group of 
$\mathrm{PGL}(n,\mathbb{C})$ in \cite[Section 18.3]{Xi}, instead of $\mathrm{SL}(n,\mathbb{C})$. 
\end{remark}

In \cite[Theorem 3]{Be} and \cite[Theorem 4]{BO}, the authors give a geometric description of the almost fusion category 
$\cA_\J$, for any two-sided cell $\J$ of $W^{\mathrm{ext}}$, proving and generalizing a conjecture due to Lusztig and 
using his bijection from \cite[Theorem 4.8]{Lu3} between two-sided cells $\J$ in $W^{\mathrm{ext}}$ and conjugacy classes of unipotent elements $\mathtt{u}_\J$ in $G^\vee$. Note that Lusztig does not mention $G^\vee$, but he uses the weight lattice 
of $G$ to define $W^{\mathrm{ext}}$ in \cite[Section 1.6]{Lu3}, which amounts to the same, see Remark~\ref{rem:weightcoweight}. Bezrukavnikov and Ostrik follow Lusztig's conventions, so they also use $G^\vee$ 
implicitly. For the so-called {\em canonical} diagonal $H$-cell $\H_{\mathtt{c}(\J)}$, Bezrukavnikov~\cite[Theorem 3]{Be} shows that $\cA_{\H_{\mathtt{c}(\J)}}\simeq \mathrm{Rep}(F_\J)$, the semisimple category of finite dimensional rational 
complex representations of $F_\J$, where $F_\J$ is the Levi factor of $Z_{G^\vee}(\mathtt{u}_\J)$. 
 
When $W$ is of (extended) affine type $A_{n-1}$, every two-sided cell contains the longest 
element of a finite parabolic subgroup of $W$, which is a key hypothesis in Theorem \ref{thm:mainthmforS}. 
The diagonal $H$-cells containing these elements are not the canonical $H$-cells in general, but fortunately that does not matter, thanks to Theorem~\ref{thm:mainthmforJ}.  As a matter of fact, in (extended) affine type $A_{n-1}$, the asymptotic categories for all diagonal $H$-cells in a given two-sided cell are even equivalent as pivotal categories. 

Contrary to $\Saff_n$, the group $\Saffext_n$ is not a Coxeter group, hence we will first recall the definition of Soergel bimodules for $\Saffext_n$ below and explain how the results from the previous sections extend to these bimodules. After that, we will briefly recall Bezrukavnikov and Ostrik's result for $\Saffext_n$, and then how to adapt it for $\Saff_n$. 

\subsection{The (extended) affine symmetric group}
\label{subsec:affsymS}
Let $n\in \mathbb{N}_{\geq 3}$. The {\em affine symmetric group} 
$\Saff_n$ is the infinite Coxeter group generated by the simple reflections $s_0,\ldots, s_{n-1}$ subject to the relations
\begin{equation}\label{eq:Saff}
s_i^2=e,\quad s_is_j=s_js_i\;\text{for}\;\vert i-j\vert>1,\quad s_i s_{i+1} s_i=
s_{i+1} s_i s_{i+1},
\end{equation}
for $i,j\in \{0,\ldots, n-1\}$, where one has to take the indices in the relations modulo $n$, e.g., $s_n=s_0$ by definition. The {\em extended affine symmetric group} 
$\Saffext_n$ is the semidirect product 
\begin{equation}\label{eq:affextS1}
\Saffext_n:=\langle \rho \rangle  \ltimes \Saff_n,
\end{equation}
where $\langle \rho\rangle$ is an infinite cyclic group generated by $\rho$, acting 
on $\Saff_n$ by 
\begin{equation}\label{eq:affextS2}
\rho s_i \rho^{-1}=s_{i+1},
\end{equation}
for $i=0,\ldots, n-1$, where $i$ has to be taken modulo $n$ again. The {\em 
symmetric group} $\mathfrak{S}_n$ is the finite subgroup of $\Saff_n$ generated by 
$s_1,\ldots, s_{n-1}$. Note that these groups can also be defined for $n=1,2$, by 
putting $\mathfrak{S}_1=\Saff_1=\{e\}$ and $\Saffext_1=\langle \rho\rangle$, and letting $\Saff_2$ be the group generated by $s_0$ and $s_1$, subject only to the quadratic relations $s_0^2=s_1^2=e$. Note that $\Saff_2$ is isomorphic to the infinite dihedral group, which is the Coxeter group of type $I_2(\infty)$.


\subsection{Soergel bimodules for $\Saffext_n$}
Let $R:=\mathbb{C}[y,x_1,\ldots, x_n]$ with grading determined by twice the polynomial degree. The group $\Saffext_n$ acts on $R$ by degree-preserving $\mathbb{C}[y]$-linear automorphisms such that the elements of $\mathfrak{S}_n$ permute the $x_i$, whereas $s_0$ and $\rho$ act by 
\begin{gather}\label{eq:action}
s_0(x_i)=
\begin{cases}
x_i& \text{if}\; i=2,\ldots, n-1;\\
x_n-y& \text{if}\; i=1;\\
x_1+y& \text{if}\; i=n,
\end{cases}
\qquad 
\rho(x_i)=
\begin{cases}
x_{i+1}& \text{if}\; i=1,\ldots, n-1;\\
x_1+y & \text{if}\; i=n.
\end{cases}
\end{gather}

Since $\Saff_n$ is a Coxeter group, the category $\Soeaff_n$ of Soergel bimodules for $\Saff_n$ can be defined as in Section~\ref{sec:Soergel}. As a monoidal category, it is generated by the $R$-bimodules $\rB_i:=\rB_{s_i}$, for $i=1,\ldots, n-1$. For every $w\in \Saff_n$, there is an indecomposable $R$-bimodule $\rB_w\in \Soeaff_n$, with a particular choice of grading, and these bimodules are all pairwise non-isomorphic and exhaust the indecomposable objects in $\Soeaff_n$ 
up to isomorphism and grading shift.  

To define the category $\Soeext_n$ of Soergel bimodules for $\Saffext_n$, we have to introduce the additional $R$-bimodules $\rB_{\rho^k}$, for $k\in \mathbb{Z}$. As a left $R$-module, we take $\rB_{\rho^k}:=R$. As a right $R$-module, $\rB_{\rho^k}$ is also isomorphic to $R$, but in a non-trivial way because the right $R$-action is defined by $f\cdot g:=f\rho^k(g)$, for $f,g\in R$. Note 
that $\rB_{\rho^k}$ is isomorphic to $R$ as a left and as a right $R$-module, but not as an $R$-bimodule, unless $k=0$. 

These $R$-bimodules are all 
indecomposable, of course, and it is easy to check that, for any $k,l\in \mathbb{Z}$, there is a canonical isomorphism of graded $R$-bimodules
\begin{equation}\label{eq:twistedproducts1}
\rB_{\rho^k}\rB_{\rho^m}\cong \rB_{\rho^{k+m}}
\end{equation}
and an isomorphism of left (resp. right) graded $R$-modules
\begin{equation}\label{eq:twistedhoms1}
\mathrm{HOM}_{(R\text{-gbim})^*}(\rB_{\rho^k},\rB_{\rho^m})\cong 
\begin{cases}
R, & \text{if}\; k=m;\\
\{0\}, & \text{else}.
\end{cases}
\end{equation}

\begin{definition}\label{def:Soeext} 
The category $\Soeext_n$ of {\em Soergel bimodules} for $\Saffext_n$ is the smallest idempotent complete full monoidal subcategory of $R$-gbim generated by $\Saff_n$ and the $R$-bimodules $\rB_{\rho^k}\{ t\}$, for $k,t\in \mathbb{Z}$.  
\end{definition}

The isomorphism in \eqref{eq:twistedproducts1} implies that there are degree-preserving isomorphisms of $R$-bimodules $\rB_\rho\rB_{\rho^{-1}}\cong \rB_{\rho^{-1}}\rB_\rho\cong R$. 
For any $i=0,\ldots,  n-1$, there is also a canonical degree-preserving isomorphism of $R$-bimodules 
\begin{equation}\label{eq:twistedproducts2}
\rB_{\rho}\rB_i \rB_{\rho^{-1}}\cong \rB_{i+1},
\end{equation}
which implies that any indecomposable object in $\Soeext_n$ is isomorphic to 
$\rB_{\rho^k}	\rB_w\{ t\}$, for some $w\in \Saff_n$ and $k,t\in \mathbb{Z}$. The isomorphism in \eqref{eq:twistedhoms1} 
implies that, for any $u,v\in \Saff_n$ and $k,m\in \mathbb{Z}$, there is an isomorphism of right graded $R$-modules
\begin{equation}\label{eq:twistedhoms2}
\mathrm{HOM}_{(\Soeext_n)^*}(\rB_{\rho^k}\rB_u,\rB_{\rho^m}\rB_v)\cong 
\begin{cases}
\mathrm{HOM}_{\Soeaff_n^*}(\rB_u,\rB_v), & \text{if}\; k=m;\\
0, & \text{else}.
\end{cases}
\end{equation}
Thanks to \eqref{eq:twistedhoms2}, the results from the previous sections also hold for $\Soeext_n$. In particular, there are 
asymptotic categories $\cA_\J$ and $\cA_\H$ for every two-sided cell $\J$ of $\Saffext_n$ and every diagonal $H$-cell 
$\H\subseteq \J$. 


\subsection{The asymptotic categories for $\Saffext_n$}\label{sec:asympbicextaffA}
The two-sided Kazhdan--Lusztig cells of $\Saffext_n$ correspond bijectively to the partitions of $n$, see~\cite[Section 2.2]{Xi}. We denote the set of partitions of $n$ 
by $\mathrm{P}(n)$ from now on. Let $\lambda=(\lambda_1,\ldots, \lambda_r)\in \mathrm{P}(n)$ and $\J_\lambda$ be the corresponding two-sided cell in 
$\Saffext_n$. Xi's conventions are consistent with Lusztig's bijection in~\cite[Theorem 4.8]{Lu3}, hence the conjugacy class of the unipotent element $\mathtt{u}_\lambda$ in $\mathrm{GL}(n,\mathbb{C})$ corresponding to $\J_\lambda$ has Jordan normal form determined by the dual partition $\mu=(\mu_1,\ldots, \mu_{r'})\in \mathrm{P}(n)$ and the number of left cells in $\J_\lambda$ is equal to 
\begin{equation}\label{eq:nlambda}
n_\lambda:=\dfrac{n!}{\mu_1!\cdots \mu_{r'}!}.
\end{equation}

Let $\mathfrak{S}_\lambda\cong \mathfrak{S}_{\lambda_1}\times \cdots \times \mathfrak{S}_{\lambda_r}$ be the finite parabolic subgroup of $\Saffext_n$ generated by the simple reflections 
\[
s_1,\ldots, s_{\lambda_1-1}, s_{\lambda_1+1},\ldots, s_{\lambda_1+\lambda_2-1},s_{\lambda_1+\lambda_2+1},\ldots, s_{\lambda_1+\cdots+\lambda_r-1}.
\]
The longest element $w_\lambda$ in $\mathfrak{S}_\lambda$, which is a Duflo involution, belongs to $\J_\lambda$ and will be denoted by $d_\lambda$. The left cell and the diagonal $H$-cell containing $d_\lambda$ are denoted by $\Gamma_\lambda$ 
and $\H_\lambda:=\Gamma_\lambda\cap \Gamma_\lambda^{-1}$, respectively. For lack of a better name, we call these the {\em parabolic left cell} and the {\em parabolic $H$-cell} in $\J_\lambda$, respectively.

For any $w\in \Saffext_n$, let $\cR(w)$ be the right descent set of $w$ and define 
$Y_0:=\left\{w\in \Saffext_n \mid \cR(w)\subseteq \{s_0\}\right\}$. By \cite[Theorem 1.2]{LN}, the intersection 
$\Gamma_{\mathtt{c}(\lambda)}:=Y_0\cap \J_\lambda$ is exactly one left cell of $\Saffext_n$, called the {\em canonical left cell} in $\J_\lambda$. 
We denote the corresponding Duflo involution by $d_{\mathrm{c}(\lambda)}$ 
and the diagonal $H$-cell containing it by $\H_{\mathtt{c}(\lambda)}:=\Gamma_{\mathtt{c}(\lambda)}\cap \Gamma_{\mathtt{c}(\lambda)}^{-1}$, which we call the {\em canonical $H$-cell} in $\J_\lambda$.

Given two Duflo involutions $d_1, d_2$ in $\J_\lambda$ such that $\rho^j d_1 \rho^{-j} = d_2$, for some integer $j$, it is clear that conjugation by $\rB_{\rho^j}$ defines an equivalence between $\cS_{\H_1}$ and $\cS_{\H_2}$, where $\H_1$ and $\H_2$ are the respective diagonal $H$-cells, which implies that $\cA_{\H_1}$ and $\cA_{\H_2}$ 
are also equivalent. When $d_1$ and $d_2$ are not conjugate to each other by a power of $\rho$, it is still true that $\cA_{\H_1}$ and $\cA_{\H_2}$ are equivalent (in this particular affine Weyl type), but the proof is much harder, see \cite[Section 8.2]{Xi} (and references therein). However, the bicategories $\cS_{\H_1}$ and $\cS_{\H_2}$ may be inequivalent, see Example~\ref{ex:Dufloinv} below.

Using the notation with multiplicities for partitions, write $\mu=(i_1^{m_1} i_2^{m_2} \ldots i_{k}^{m_k})$ with 
$\mu_1=i_1> i_2>\ldots >i_{k}=\mu_{r'}$. By the above, the almost fusion categories $\cA_{\H_\lambda}$ and $\cA_{\H_{\mathtt{c}(\lambda)}}$ are both equivalent to $\mathrm{Rep}(F_\lambda)$, where  
\begin{equation}\label{eq:Flambda}
F_\lambda\cong \mathrm{GL}(m_1,\mathbb{C})\times \cdots \times \mathrm{GL}(m_k,\mathbb{C}), 
\end{equation}
which is the Levi subgroup of the centraliser $Z_{ \mathrm{GL}(n,\mathbb{C})}(\mathtt{u}_\lambda)$, see \cite[Section 8.2]{Xi} (including references therein), \cite[Theorem 3]{Be} and \cite[Section 5.2 and Example 5.5(b)]{BO}. 

\begin{example}\label{ex:Dufloinv} Let $n=4$ and $\lambda=(2^2)$. Then $\mu=\lambda$, $\mathtt{a}(\J_\lambda)=2$ and $\J_\lambda$ has six left cells according to \eqref{eq:nlambda}, with respective Duflo involutions $d_1,\ldots, d_6$ given by 
\[
d_1=s_1s_3,\; d_2=s_0s_2,\; d_3=s_0 s_1 s_3 s_0,\;d_4=s_1 s_0 s_2 s_1,\; d_5=s_2 s_1 s_3 s_2,\; d_6=s_3 s_0 s_2 s_3.
\] 
Note that because $\rho d_i\rho^{-1}=d_{i+1}$, for $i=1,3,4,5$, there are two $\rho$-conjugacy classes, namely $\{d_1,d_2\}$ and $\{d_3,d_4,d_5,d_6\}$. One sees immediately that $d_\lambda=d_1$ and $d_{\mathtt{c}(\lambda)}=d_3$, which are also conjugate to each other, but not by a power of $\rho$. 

In this case, it is not hard to show directly that $\cA_{\H_\lambda}$ and $\cA_{\H_{\mathtt{c}(\lambda)}}$ are equivalent, because we have 
\[
\widetilde{\rC}_{0130} \cong  \widetilde{\rC}_{0}\widetilde{\rC}_{13}\widetilde{\rC}_{0}^*
\]
in $\cS_{\J_\lambda}$, so $\widetilde{\rC}_{0130}$ is obtained from $\widetilde{\rC}_{13}$ by framing with $\widetilde{\rC}_{0}$. By Lemma~\ref{framing}, 
this framing induces an equivalence between $\cB_{\H_\lambda}$ and $\cB_{\H_{\mathtt{c}(\lambda)}}$, 
which restricts to an equivalence between $\cA_{\H_\lambda}$ and $\cA_{\H_{\mathtt{c}(\lambda)}}$, thanks to Proposition~\ref{prop:bimodbicat}. This implies that all 
$\cA_{\H_i}$, for $i=1,\ldots, 6$, are equivalent to $\mathrm{Rep}( \mathrm{GL}(2,\mathbb{C}))$.

Now, let $\lambda=(4)$. Then $\mu=(1^4)$, $\mathtt{a}(\J_\lambda)=6$ (the maximal $\mathtt{a}$-value in $\Saffext_4$) and $\J_\lambda$ has $24$ left cells, each with a unique Duflo involution. Two of these Duflo involutions are 
\[
d_1:=s_1 s_0 s_ 1 s_ 2 s_ 0 s_ 1\quad\text{and}\quad d_2:= s_3 s_1 s_0 s_1 s_2 s_1 s_0 s_3.
\]   

Using SageMath~\cite{sagemath}, an open-source mathematics software system, we computed that 
\begin{eqnarray}\label{eq:decomp1}
\rB_{101201}^2&\cong &  \rB_{101201}^{\oplus  \left(\vpar^6+3\vpar^4+5\vpar^2 + 6 +5\vpar^{-2} + 3\vpar^{-4} + \vpar^{-6}\right) }\\ \label{eq:decomp2}
\rB_{31012103}^2&\cong & \rB_{31012103}^{\oplus  \left(\vpar^6+3\vpar^4+5\vpar^2 + 6 +5\vpar^{-2} + 3\vpar^{-4} + \vpar^{-6}\right) } \oplus \rB_{30121030121031}^{\oplus \left(\vpar^2+2+\vpar^{-2}\right) }
\end{eqnarray}
Let $\H_1$ and $\H_2$ be the diagonal $H$-cells containing $d_1$ and $d_2$, respectively. The decompositions in 
\eqref{eq:decomp1} and \eqref{eq:decomp2} imply that $\cS_{\H_1}$ and $\cS_{\H_2}$ cannot be equivalent as monoidal 
categories, because such an equivalence $\phi\colon \cS_{\H_1}\to \cS_{\H_2}$ would have to satisfy 
\[
\phi(R)\cong R\quad\text{and}\quad \phi(\rB_{d_1})\cong \rB_{d_2},
\]
as $R$ is the identity $1$-morphism and $\rB_{d_i}$ is the unique, up to isomorphism, indecomposable self-dual Soergel bimodule in $\H_i^\oplus\subset \cS_{\H_i}$ such that $\mathrm{Hom}_{\cS}(R,\rB_{d_i}(6))\ne 0$, for $i=1,2$. This would contradict the decompositions above if $\phi$ were to preserve the monoidal product up to natural isomorphism. Note that on the asymptotic 
level things are much simpler, because $\rA_{d_i}^2\cong \rA_{d_i}$ and
\[
\cA_{\H_i}\simeq \mathrm{Rep}(\mathrm{GL}(4,\mathbb{C})),
\]
for $i=1,2$.
\end{example} 





\subsection{Simple birepresentations of $\Soeext_n$} \label{secsimpleextA}

In this subsection, we will illustrate how Theorem~\ref{thm:mainthmforS} 
works in extended affine type $A_{n-1}$. Specifically, we will show how semisimple geometrical module categories over finite products 
of complex general linear groups, defined and studied in \cite{Ge}, give rise to certain simple almost finitary birepresentations of $\Soeext_n$, which we will also call {\em geometrical} by association. Although geometrical module categories have 
a concrete definition, they have not been classified in general. Moreover, not all semisimple module categories are geometrical, so this method does certainly not yield a complete classification of simple almost finitary $\Soeext_n$-birepresentations, but it does provide an interesting class of simple birepresentations of $\Soeext_n$.

We first recall Gelaki's~\cite{Ge} results in the case when $G$ is a connected reductive complex algebraic group. 
Given a (Zariski-)closed subgroup $H$ of $G$, the restriction functor 
\[
\mathrm{Res}^G_H\colon \mathrm{Rep}(G)\to \mathrm{Rep}(H)
\] 
is a monoidal functor, so it induces a functorial action $\circ\colon \mathrm{Rep}(G)\boxtimes_\bbC \mathrm{Rep}(H)\to  \mathrm{Rep}(H)$, which on objects is given by 
\[
V \circ W:=\mathrm{Res}^G_H(V)\otimes_\bbC W
\]
for $V\in \mathrm{Rep}(G)$ and $W\in  \mathrm{Rep}(H)$, and on morphisms is defined similarly. We denote the 
resulting $\mathrm{Rep}(G)$-birepresentation by $\mathbf{Rep}(H)$ and note that it is always transitive, but the underlying category is only semisimple if $H$ is reductive. If $H$ is not reductive, the negligible morphisms in $\mathrm{Rep}(H)$ 
form a non-zero $\mathrm{Rep}(G)$-stable ideal $I$, because they form a tensor ideal in $\mathrm{Rep}(H)$. 
This means that $\mathbf{Rep}(H)$ is not simple as a $\mathrm{Rep}(G)$-birepresentation, but it has a simple quotient 
with underlying category $\mathrm{Rep}(H)/I$, which is semisimple and equivalent to $\mathrm{Rep}(L)$, where $L$ is any Levi subgroup (maximal reductive subgroup) of $H$. This shows that $\mathbf{Rep}(H)$ is a simple 
$\mathrm{Rep}(G)$-birepresentation if and only if $H$ is reductive. 

Given a (normalized) algebraic group $2$-cocycle $\psi\colon H\times H \to \mathbb{G}_m$, where $\mathbb{G}_m\cong \bbC^\times$, the restriction functor $\mathrm{Res}^G_H$ also induces a functorial action on $\mathrm{Rep}^\psi(H)$, the category of projective rational finite dimensional complex $H$-representations with Schur multiplier $\psi$. As is well-known, the resulting $\mathrm{Rep}(G)$-birepresentation $\mathbf{Rep}^\psi(H)$ only depends on the cohomology class $[\psi]$ of $\psi$ in $H^2(H,\mathbb{G}_m)$. Recall also that any $2$-cocycle is cohomologous with a normalized one, satisfying $\psi(h,e)=\psi(e,h)=1$ for all $h\in H$, so we can always assume that $\psi$ is normalized. As before, the birepresentation $\mathbf{Rep}^\psi(H)$ is simple if and only if $H$ is reductive.

Given two closed subgroups $H_1$ and $H_2$ of $G$ and two Schur multipliers $\psi_1$ and $\psi_2$ on them, 
the $\mathrm{Rep}(G)$-birepresentations $\mathbf{Rep}^{\psi_1}(H_1)$ and 
$\mathbf{Rep}^{\psi_2}(H_2)$ are equivalent if and only if $(H_1,\psi_1)$ and $(H_2,\psi_2)$ are conjugate to each other, 
which means that there is an element $g\in G$ such that $H_2=gH_1g^{-1}$ and $[\psi_2]=[\psi_1^g]\in H^2(H_2,\mathbb{G}_m)$, 
where $\psi_1^g(ghg^{-1},gh'g^{-1}):=\psi_1(h,h')$ for all $h,h'\in H_1$. 

Gelaki calls $\mathrm{Rep}(G)$-birepresentations of the above form {\em geometrical}, cf. \cite[Theorem 4.5]{Ge}. 
Thus there is a bijection between equivalence classes of simple geometrical $\mathrm{Rep}(G)$-birepresentations and conjugacy classes of pairs $(H,[\psi])$, where $H$ is a reductive closed subgroup of $G$ and $[\psi]\in H^2(H,\mathbb{G}_m)$. Not all 
$\mathrm{Rep}(G)$-birepresentations are geometrical in general, e.g., Etingof and Ostrik showed in 
\cite[Section 3.2]{EO} that that are (non-symmetric) fiber functors $\mathrm{Rep}(\mathrm{SL}(2,\bbC))\to \mathrm{Vect}_\bbC$ that are not naturally isomorphic to the forgetful functor $F=\mathrm{Res}^{\mathrm{SL}(2,\bbC)}_{\{e\}}$ as monoidal functors. 
Every fiber functor yields a rank-one simple $\mathrm{SL}(2,\bbC)$-birepresentation, but only the one induced by $F$ is geometrical. 
Fiber functors $\mathrm{Rep}(\mathrm{SL}(n,\bbC))\to \mathrm{Vect}_\bbC$ have been classified for $n=2$ and $n=3$, 
see \cite[Section 3.2]{EO} (and references therein) and \cite[Proposition 3.1 and Definition 3.3]{Oh}, but not for $n>3$. 
If $G$ is abelian, fiber functors have been classified too, see \cite[Proposition 2.6.1]{EGNO}. 

Even for simple geometrical birepresentations of $\mathrm{Rep}(G)$, there is no general concrete classification, although 
it exists for some $G$. For example, if $G=\mathbb{G}_m^k$ (the $k$-dimensional complex torus), for some $k\in \mathbb{N}$, 
then any closed reductive subgroup $H$ of $G$ satisfies
\begin{equation}\label{eq:HforTorus1}
H\cong T\times A,
\end{equation}
for some complex torus $T:=\mathbb{G}_m^l$ of dimension $l$, with $0\leq l\leq k$, and some finite abelian group $A$, see \cite[Section 16.2]{Hu}. Moreover, the cohomology group 
$H^2(H,\mathbb{G}_m)$ is also known for all $H$. If $A\cong \prod_{i=1}^r \bbZ/d_i\bbZ$ in \eqref{eq:HforTorus1}, then 
\begin{equation}\label{eq:HforTorus2}
H^2(H, \mathbb{G}_m)\cong H^2(A, \mathbb{G}_m)\cong \bigoplus_{1\leq i< j\leq r} \bbZ/\mathrm{gcd}(d_i,d_j)
\bbZ.
\end{equation} 
The second isomorphism in \eqref{eq:HforTorus2} is a standard result in group cohomology of finite groups, see \cite[Proposition 2.1.1 and Theorem 2.2.10]{Ka}. To see why the 
first isomorphism is true, consider the Lyndon-Hochschild-Serre (LHS) spectral sequence 
\[
E^{p,q}_2=H^p(A,H^q(T,\mathbb{G}_m)) \Longrightarrow E^{p+q}_\infty = H^{p+q}(G,\mathbb{G}_m), 
\]
induced by short (split) exact sequence of algebraic groups $1\to T \to H\to A\to 1$. It is well-known that 
\[
H^q(T,\mathbb{G}_m)\cong 
\begin{cases}
\mathbb{G}_m, & \text{if}\; q=0;\\
\bbZ^l, & \text{if}\; q=1;\\
0, & \text{if}\; q>1,
\end{cases}
\]
see \cite[Proposition 15.8]{Mi} for the case when $n\geq 2$, hence $E^{p,q}_2=0$, for all $p\geq 0$ and $q\geq 2$, and 
\begin{eqnarray*}
H^2(A, H^0(T,\mathbb{G}_m))&\cong& H^2(A,\mathbb{G}_m),\\
H^1(A,H^1(T,\mathbb{G}_m)) &\cong &\mathrm{Hom}(A,\bbZ^l)=0.
\end{eqnarray*}
This implies that the LHS spectral sequence collapses on the $E_2$-page and the first isomorphism in \eqref{eq:HforTorus2} holds.

For $G=\mathrm{GL}(2,\bbC)$, there is also a complete classification of the conjugacy classes of all 
closed reductive subgroups $H$ of $G$, see \cite{NPT}. For any such $H$, let $H_0$ be the connected component containing the 
neutral element, then $H/H_0$ is a finite group and $H^2(H,\mathbb{G}_m)\cong H^2(H/H_0,\mathbb{G}_m)$, which is known 
explicitly for all the finite groups showing up in the classification in \cite{NPT} (we do not have a reference for the group 
cohomology of all these groups, but they can be computed by the usual techniques). A comprehensive review of many interesting classification results of the closed subgroups of $G$, for $G\in \{\mathrm{GL}(n,\bbC), \mathrm{SL}(n,\bbC), \mathrm{PSL}(n,\bbC)\}$, with $n\geq 2$, can be found in \cite{Ca}. 

To use these results in the case of $\Soeext_n$, let $\lambda\in \mathrm{P}(n)$ and recall that $F_\lambda$ is reductive, see \eqref{eq:Flambda}. By Theorem~\ref{thm:mainthmforS}, there is a bijection between the equivalence classes of simple almost 
finitary $\Soeext_n$-birepresentations $\mathbf{M}$ with apex $\J_\lambda$ and the equivalence classes of simple almost finitary birepresentations $\mathbf{N}$ of $\mathrm{Rep}(F_\lambda)$, up to overall shifts.  Following Gelaki's terminology, we say that $\mathbf{M}$ is {\em geometrical} if $\mathbf{N}$ is geometrical. Thus, there is a bijection between the equivalence classes of simple geometrical almost finitary $\Soeext_n$-birepresentations with apex $\J_\lambda$ and the conjugacy classes of pairs $(H,[\psi])$, where $H$ is a reductive closed subgroup of $F_\lambda$ and $[\psi]\in H^2(H,\mathbb{G}_m)$. 

By the remarks in the paragraph containing \eqref{eq:HforTorus1} and \eqref{eq:HforTorus2}, this bijection 
yields a complete classification of all simple geometrical almost finitary $\Soeext_n$-birepresentations with apex $\J_\lambda$, when $F_\lambda$ is isomorphic to $\mathbb{G}_m^k$, for some integer $k\geq 0$, or to $\mathrm{GL}(2,\bbC)$. The former case occurs when the partition $\mu=(\mu_1,\ldots, \mu_{r'})$ dual to $\lambda$ satisfies $\mu_1>\ldots>\mu_{r'}>0$, while the latter case occurs when $n=2k$ and $\mu=(k,k)$. When $\mu$ has at most two rows, it always satisfies one of these two conditions. 
As is well-known, the two-sided cells $\J_\lambda$ for the corresponding $\lambda$ (with at most two columns) are exactly the ones 
containing the fully commutative elements in $\Saffext_n$, which are the elements whose rexes can be transformed into 
one another by conjugation with powers of $\rho$ and the application of commutation relations only. The quotient of $\Soeext_n$ 
by the monoidal ideal generated by all two-sided cells $\J_\lambda$, for $\lambda$ with more than two rows, 
categorifies the extended affine type $A$ Temperley-Lieb algebra, which was shown to be an affine cellular algebra by Graham and Lehrer in \cite{GL}, who also classified its finite dimensional cell representations and their irreducible quotients. 
Thus we have obtained a complete classification of all simple geometrical almost finitary birepresentations, including the finitary ones of finite rank, of this categorified extended affine type $A$ Temperley-Lieb algebra. In Conjecture~\ref{conj:relswithEvandGL}, 
we propose a precise connection with Graham and Lehrer's work.

To illustrate how Theorem~\ref{thm:mainthmforS} works in practice, let us explain one particular example of simple geometrical birepresentation with trivial Schur multiplier (the asymptotic part of this story is well-known to experts, but we hope that a detailed account of it helps to make the paper accessible to a wider audience). Let $G:=\mathrm{GL}(m,\bbC)$, for some $m\in \bbN$, and consider the forgetful functor 
\[
F \colon \mathrm{Rep}(G)\to \mathrm{Vect}_\bbC
\] 
and the corresponding simple geometrical $\mathrm{Rep}(G)$-birepresentation $\mathbf{Vect}_\bbC$ (which is, of course, 
finitary because it has rank one). We claim that the algebra object $[\mathbb{C},\mathbb{C}]\in \mathrm{Rep}(G)^\sqcup$, obtained via the internal hom-construction, is isomorphic to the algebra $\mathcal{O}(G)$ of regular complex functions on $G$ (polynomials with complex coefficients in the matrix entries $x_{ij}$ and the inverse of the determinant) with pointwise multiplication. 

To simplify notation, we'll write $\otimes:=\otimes_\bbC$, $\mathrm{Vect}:=\mathrm{Vect}_\bbC$, $\mathrm{Hom}_\bbC:=\mathrm{Hom}_{\mathrm{Vect}_\bbC^\sqcup}$ and $\mathrm{Hom}_G:=\mathrm{Hom}_{\mathrm{Rep}(G)^\sqcup}$ 
in the proof of the claim. Recall that there is an essentially unique irreducible finite dimensional complex rational 
$G$-representation $L_\nu$, for every dominant integral $G$-weight $\nu\in \Lambda_m^+$, and that the (algebraic) Peter--Weyl Theorem says that 
\begin{equation}\label{eq:PW}
\mathcal{O}(G)\cong \bigoplus_{\nu\in  \Lambda_m^+} L_\nu\otimes L_\nu^*,
\end{equation}
where $L_\nu^*$ is the dual representation. In the first place, this implies that $\mathcal{O}(G)$ belongs to $\mathrm{Rep}(G)^\sqcup$, where the 
$G$-action on $\mathcal{O}(G)$, 
defined by $g\cdot f(x)=f(xg)$ for $g,x\in G$ and $f\in \mathcal{O}(G)$, corresponds to the 
$G$-action on the left tensor factor in $L_\nu\otimes L_\nu^*$ for every $\nu\in  \Lambda_m^+$. The other tensor factor serves merely as a multiplicity space for the left $G$-action. Secondly, since $\mathrm{Rep}(G)$ is semisimple, the isomorphism in \eqref{eq:PW} also implies that, for every $\nu\in  \Lambda_m^+$, there is an isomorphism
\[
\mathrm{Hom}_{\bbC}(F(L_\nu), \mathbb{C})\cong F(L_\nu^*)\cong 
\mathrm{Hom}_{G}(L_\nu,\mathcal{O}(G)),
\]
which proves that $[\bbC,\bbC]\cong \mathcal{O}(G)$ as $1$-morphisms in $\mathrm{Rep}(\mathcal{O}(G))^\sqcup$. 

To see that multiplication in $[\bbC,\bbC]$ corresponds to pointwise multiplication in 
$\mathcal{O}(G)$, first note that the evaluation morphism $\mathrm{ev}\in \mathrm{Hom}_\bbC(\mathcal{O}(G)\otimes \bbC,\bbC)$ is given by $\mathrm{ev}(f\otimes 1)=f(e)$, which corresponds exactly to the identity on $\mathcal{O}(G)$ under the isomorphism 
$\mathrm{Hom}_\bbC(\mathcal{O}(G)\otimes \bbC,\bbC)\cong 
\mathrm{Hom}_G(\mathcal{O}(G), \mathcal{O}(G))$ because $g\cdot f(e)=f(g)$, for $g\in G$ and $f\in \mathcal{O}(G)$. Multiplication in $[\bbC,\bbC]\cong \mathcal{O}(G)$ is defined by composition of 
\begin{gather*}
\mathcal{O}(G)\otimes \mathcal{O}(G)\otimes \bbC\xrightarrow{\mathrm{id}\otimes \mathrm{ev}} 
\mathcal{O}(G)\otimes \bbC\xrightarrow{\mathrm{ev}} \bbC,\\
f\otimes f' \otimes 1 \mapsto f\otimes f'(e)\mapsto f(e)f'(e). 
\end{gather*}
Since the multiplication in $\mathcal{O}(G)$, denoted temporarily by $\star$,  is required to be a morphism in 
$\mathrm{Rep}(G)^\sqcup$, we get 
\[
f\star f'(g)=g\cdot (f\star f')(e)=(g\cdot f)(e)(g\cdot f')(e)=f(g)f'(g),
\] 
for every $g\in G$. This is exactly pointwise multiplication of $f$ and $f'$ and, therefore, completes the proof of above claim.

In general, for $\lambda\in \mathrm{P}(n)$, let $F\colon \mathrm{Rep}(F_\lambda)\to \mathrm{Vect}_\bbC$ be the forgetful functor, where $F_\lambda$ is given in \eqref{eq:Flambda}. Then algebra object $[\mathbb{C},\mathbb{C}]$ is 
isomorphic to 
\[
\mathcal{O}(F_\lambda)\cong \mathcal{O}(\mathrm{GL}(m_1,\bbC))\otimes_\bbC \cdots \otimes_\bbC \mathcal{O}(\mathrm{GL}(m_k,\bbC))
\]
in $\mathrm{Rep}(F_\lambda)^\sqcup$, with tensor factorwise pointwise multiplication. 
In \cite[Lemma 5.2.4]{Xi}, Xi establishes a bijection between the elements of $\H_\lambda$ and 
the dominant integral weights of $F_\lambda$. By the Peter-Weyl Theorem and this bijection, the algebra object 
$\mathcal{O}(F_\lambda)$ corresponds to an algebra object $\rA_\lambda$ in $\cA_{\H_\lambda}^\sqcup$, and the lax 
monoidal functor $\Theta$ maps $\rA_\lambda$ to an algebra object $\rB_\lambda$ in $\cS_{\H_{\lambda}}^\sqcup$, where 
$\cS=\Soeext_n$. The corresponding rank-one $\cS_{\H_\lambda}$-birepresentation $\mathbf{proj}_{\cS_{\H_{\lambda}^{\sqcup}}}(\rB_\lambda)$ is simple with apex $\H_\lambda$ and, in turn, induces the simple $\Soeext_n$-birepresentation 
\begin{equation}\label{eq:forgetcellbirep}
\mathbf{C}_\lambda^\mathrm{f}:=\mathbf{proj}_{\cS_{\J_\lambda}^{\sqcup}}(\rB_\lambda)
\end{equation} 
with apex $\J_\lambda$, which we propose to call the {\em forgetful cell birepresentation} of $\Soeext_n$ associated to the 
left cell $\Gamma_\lambda$.

We claim that the rank of $\mathbf{C}_\lambda^\mathrm{f}$ is equal to the number $n_\lambda$ defined in \eqref{eq:nlambda}. To see this, note that this rank is equal to the rank of the decategorification of the 
corresponding asymptotic birepresentation, which is a representation of $A_{\J_\lambda}$, the asymptotic Hecke algebra 
(a.k.a. Lusztig's $J$-ring) associated to $\J_\lambda$. By \cite[Theorem 2.3.2(c)]{Xi}, the latter algebra is isomorphic to the matrix algebra $\mathrm{Mat}_{n_\lambda\times n_\lambda}(R_{F_\lambda})$, where $R_{F_\lambda}$ is the representation ring of $F_\lambda$. The forgetful functor $F$, therefore, induces a homomorphism of $\bbZ$-algebras
\[
f\colon \mathrm{Mat}_{n_\lambda\times n_\lambda}(R_{F_\lambda})\to \mathrm{Mat}_{n_\lambda\times n_\lambda}(\bbZ),
\]
which defines the $A_{\J_\lambda}$-action on $\bbZ^{n_\lambda}$. In particular, this implies that the rank of $\mathbf{C}_\lambda^\mathrm{f}$ is equal to $n_\lambda$, proving our claim. 

\begin{conjecture}\label{conj:relswithEvandGL} Let $\cS_n$ be the monoidal category of Soergel bimodules for the finite 
symmetric group $\Sfin_n$ and $\mathbf{C}_\lambda^\mathrm{fin}$ the $\cS_n$-cell birepresentation associated to the (finite) left cell in $\Sfin_n$ containing $w_\lambda$. By \cite[Section 5.1 and Theorem 5.4]{MMV1}, there is a so-called {\em evaluation functor} $\mathcal{E}v_{r,s}\colon \Soeext_n\to K^b(\cS_n)$, for every $r,s\in \bbZ$. This is 
a monoidal linear functor, hence every finitary $\cS_n$-birepresentation $\mathbf{M}$  
gives rise to a triangulated $\Soeext_n$ -birepresentation $\mathbf{M}^{\mathcal{E}v_{r,s}}$ via pull-back along $\mathcal{E}v_{r,s}$. In \cite[Proposition 6.2]{MMV1}, it was shown that 
$\mathbf{M}^{\mathcal{E}v_{r,s}}$ always has a {\em finitary cover}, which is a finitary $\Soeext_n$-birepresentation $\mathbf{L}$ together with a faithful morphism of $\Soeext_n$-birepresentations 
$\Pi\colon \mathbf{L}\to \mathbf{M}^{\mathcal{E}v_{r,s}}$ whose essential image generates $\mathbf{M}^{\mathcal{E}v_{r,s}}$ as a triangulated birepresentation. 

We conjecture that the forgetful cell birepresentation $\mathbf{C}_\lambda^\mathrm{f}$ is a finitary cover of the evaluation birepresentation $(\mathbf{C}_\lambda^\mathrm{fin})^{\mathcal{E}v_{n-2,2-n}}$. Moreover, we conjecture that this is essentially the unique finitary cover of minimal rank.

If $\lambda=(2^{k} 1^{n-2k})\in \mathrm{P}(n)$, for some $k\in \bbN_0$ such that $0\leq 2k\leq n$, 
then $\mu=(n-k,k)\in \mathrm{P}(n)$ and $\J_\lambda$ only contains fully commutative elements. 
We conjecture that $\mathbf{C}_\lambda^\mathrm{f}$ decategorifies to the Graham-Lehrer cell module $W_{t,z}(n)$ for 
$k=(n-t)/2$ and some parameter $z$, see \cite[Definition 2.6]{GL} and \cite[Remark 6.6]{MMV1}.
\end{conjecture}


\subsection{The asymptotic categories for $\Saff_n$}
As already remarked, the affine Weyl group $\Saff_n$ is isomorphic to the extended affine Weyl group of $\mathrm{SL}(n,\mathbb{C})$ and its two-sided cells are in bijective correspondence with the conjugacy classes of unipotent elements in $\mathrm{PGL}(n,\mathbb{C})$. The latter group is the adjoint group of $\mathrm{GL}(n,\mathbb{C})$ and the conjugacy classes of its unipotent elements are also indexed by the partitions of $n$, but the centralizers of these elements in $\mathrm{PGL}(n,\mathbb{C})$ are different. To be precise, for $\lambda\in \mathrm{P}(n)$ and its dual $\mu=(i_1^{m_1} i_2^{m_2} \ldots i_{k}^{m_k})\in \mathrm{P}(n)$ with 
$\mu_1=i_1> i_2>\ldots >i_{k}=\mu_{r'}$, we have
\[
F_\lambda\cong \left(\mathrm{GL}(m_1,\mathbb{C})\times \cdots \times \mathrm{GL}(m_k,\mathbb{C})\right)/\mathbb{C}^*,
\] 
where $\mathbb{C}^*$ is the subgroup of scalar matrices in $\mathrm{GL}(m_1,\mathbb{C})\times \cdots \times \mathrm{GL}(m_k,\mathbb{C})$ (with the same scalar 
for all $m_i$).

Again, it is true that $\cA_H\simeq \mathrm{Rep}(F_\lambda)$ for all diagonal $H$-cells of $J_\lambda$, see \cite[Section 8.3]{Xi}, but the structure of $\cA_{\J_\lambda}$ is more complicated now. However, it would take us too far away to recall it here and we refer the interested reader to \cite[Section 8.3]{Xi} and \cite[Section 5.5]{BO}.

\section{Universal Coxeter groups}
Recall that a Coxeter group $(W,S)$ is called {\em universal} if the element $st$, for any pair of distinct $s,t \in S$, has infinite order. Let $n:=\vert S\vert$ be the rank of $W$. When $n=2$, the universal Coxeter group is isomorphic to $\Saff_2$, but when $n >2$, the universal Coxeter group is not an affine Weyl group. The universal Coxeter groups have only two two-sided cells: the cell $J_e$, containing only the neutral element $e$, and the cell 
$J_{\ne e}$, containing all other elements of $W$. The left cells in $\J_{\ne e}$ are indexed by the elements of $S$, because two 
elements in $\J_{\ne e}$ are left equivalent if and only if their reduced expressions (which are unique) end with the same simple reflection, see \cite[Lemmas 3.3 and 6.1]{Dy}. Both $\J_e$ and $\J_{\ne e}$ contain, therefore, the longest element of a parabolic subgroup $W_I$, where in the former case we have $I=\emptyset$ and in the latter case we can take $I=\{s\}$, for any $s\in S$. Note that $W$ has no other finite parabolic subgroups. 

Using the complexification of the real realization of $W$ in \cite[Remark 3.13]{EL}, we have $\cA_{\J_e}\simeq \mathrm{Vect}_{\mathbb{C}}$, and we conjecture that $\cA_{\J_{\ne e}}$ is biequivalent to the $n$-colored Temperley--Lieb $2$-category over $\mathbb{C}$ with $\delta=2$, defined in \cite[Section 2]{EL}. For each $s\in S$, there is a diagonal $H$-cell $\H_s$ in $\J_{\ne e}$ and we conjecture that $\cA_{\H_s}$ is equivalent to the full monoidal subcategory generated by all $n$-colored Temperley--Lieb diagrams with both outermost regions colored by $s$. If our conjecture is true, all these asymptotic categories, for $s\in S$, are equivalent to each other as pivotal categories.  

The case for $n=2$ corresponds to the algebraic Satake correspondence in \cite{El}. 
As is explained in \cite[Example 2.9]{EL}, the $2$-colored Temperley-Lieb category is just a biategorical version of $\mathrm{Rep}(SL(2,\mathbb{C}))$. Let $L_n$, for $n\in \mathbb{N}_0$, denote the unique, up to isomorphism, irreducible complex $\mathrm{SL}(2,\mathbb{C})$-representation of dimension $n+1$, 
and let $\mathrm{Rep}^{\mathrm{even}}(SL(2,\mathbb{C}))$ and $\mathrm{Rep}^{\mathrm{odd}}(SL(2,\mathbb{C}))$ be the full subcategories of 
$\mathrm{Rep}(SL(2,\mathbb{C}))$ generated by the $L_n$ for $n$ even and $n$ odd, respectively. Note that $\mathrm{Rep}^{\mathrm{even}}(SL(2,\mathbb{C}))\cong \mathrm{Rep}(SO(3,\mathbb{C}))$, which is a monoidal category, whereas $\mathrm{Rep}^{\mathrm{odd}}(SL(2,\mathbb{C}))$ is not closed under the monoidal product but is naturally a left and right birepresentation of $\mathrm{Rep}^{\mathrm{even}}(SL(2,\mathbb{C}))$. The bicategory $\cA_{\J_{\ne e}}$ has two objects, 
which can be identified with the elements of $\mathbb{Z}/2\mathbb{Z}$, and  
\[
\mathrm{Hom}_{\cA_{\J_{\ne e}}}(\mathtt{i},\mathtt{j})\simeq 
\begin{cases}
\mathrm{Rep}^{\mathrm{even}}(SL(2,\mathbb{C})), & \mathrm{if}\; \mathtt{i}\equiv \mathtt{j}\bmod 2;\\
 \mathrm{Rep}^{\mathrm{odd}}(SL(2,\mathbb{C})), & \mathrm{else}.
\end{cases}
\]  
In particular, we see that for both diagonal $H$-cells in $\J_{\ne e}$, we have 
\[
\cA_{\H}\simeq \mathrm{Rep}(SO(3,\mathbb{C})).
\]
For more information, see \cite{El}.

\end{document}